\documentclass[10pt,a4paper]{article}

\pdfoutput=1

\usepackage{fullpage,times, mathptmx,enumitem,float,tabularx}
\usepackage[ruled,lined,linesnumbered]{algorithm2e}
\usepackage{tikz,pgf}
\usepackage[small,bf]{caption}
\usepackage{amsmath}
\usepackage{amssymb,amsthm,thmtools,thm-restate,latexsym,epsfig,subcaption,amscd,mathrsfs,euscript,eufrak,bm,physics,stmaryrd,bbm,cancel}
\usepackage{color}
\usepackage{booktabs}% http://ctan.org/pkg/booktabs
\usepackage{array}% http://ctan.org/pkg/array
\usepackage{amsfonts}
\usepackage{graphicx}
\usepackage{url} 
\usepackage{empheq,bm,authblk}

\usepackage{mathtools}

\newenvironment{sloppypar*}
{\sloppy\ignorespaces}
{\par}

\title{Adaptive Piecewise Poly-Sinc Methods for Ordinary Differential Equations}
\author[1]{Omar  Khalil}
\author[1,2] {Hany El-Sharkawy}
\author[3]{Maha Youssef}
\author[1,4]{Gerd Baumann}

\affil[1]{Mathematics Department, German University in Cairo, New Cairo City, Cairo 11835, Egypt}
\affil[2]{Department of Mathematics, Faculty of Science, Ain Shams University, 11566 Abbassia, Cairo, Egypt}
\affil[3]{ Institute of Applied Analysis and Numerical Simulation, 
	University of Stuttgart, Pfaffenwaldring 57, D-70569 Stuttgart, Germany}
\affil[4]{Faculty of Natural Sciences, University of Ulm, Albert-Einstein-Allee 11, D-89069 Ulm, Germany}

\DeclareMathOperator*{\argmax}{arg\,max}

\definecolor{color1}{rgb}{0.368417, 0.506779, 0.709798}
\definecolor{color2}{rgb}{0.880722, 0.611041, 0.142051}
\definecolor{color3}{rgb}{0.560181, 0.691569, 0.194885}
\DeclareRobustCommand{\lineone}{\raisebox{2pt}{\tikz{\draw[-,color1,solid,line width = 0.9pt](0,0) -- (5mm,0);}}}
\DeclareRobustCommand{\linetwo}{\raisebox{2pt}{\tikz{\draw[-,color2,solid,line width = 0.9pt](0,0) -- (5mm,0);}}}
\DeclareRobustCommand{\linethree}{\raisebox{2pt}{\tikz{\draw[-,color3,solid,line width = 0.9pt](0,0) -- (5mm,0);}}}

\DeclareRobustCommand{\dotSinc} {\tikz\draw[color1,fill=color1] (0,0) circle (.5ex);}

\DeclareRobustCommand{\squareChebT}{\tikz\draw[color2,fill=color2] (0,0) rectangle (1ex,1ex);}

\usepackage{mathtools,physics,bm,bbm,cancel}

\usepackage[normalem]{ulem}

\usepackage[ruled,lined,linesnumbered]{algorithm2e}

\usepackage[labelformat=simple]{subcaption}

\DeclareCaptionLabelFormat{subcaptionlabel}{\normalfont(\textbf{#2}\normalfont)}
\usepackage[export]{adjustbox} %valign

\newtheorem{theorem}{Theorem}

\newtheorem{example}{Example}

\providecommand{\keywords}[1]{\textbf{\textit{Keywords---}} #1}

\usepackage[colorlinks]{hyperref}

\usepackage{cite}

\begin{document}

\maketitle

\begin{abstract}
We propose a new method of adaptive piecewise approximation based on Sinc points for ordinary differential equations.
The adaptive method is a piecewise collocation method which utilizes Poly-Sinc interpolation to reach a preset level of accuracy for the approximation.
Our work extends the adaptive piecewise Poly-Sinc method to function approximation, for which we derived an {\it a priori}  error estimate for our adaptive method and showed its exponential convergence in the number of iterations.
In this work, we show the exponential convergence in the number of iterations of the {\it a priori} error estimate obtained from the piecewise collocation method, provided that a good estimate of the exact solution of the ordinary differential equation at the Sinc points exists. 
We use a statistical approach for partition refinement.
The adaptive greedy piecewise Poly-Sinc algorithm is validated on regular and stiff ordinary differential equations.
\end{abstract}

\keywords{adaptive approximation; Poly-Sinc interpolation;	Sinc methods;
	Lagrange interpolation;	initial value problems; boundary value problems; exponential convergence; regular differential equations; stiff differential equations}

\section{Introduction}

Numerous phenomena in engineering, physics, and mathematics are modeled either
by initial value problems {(IVPs)} or by boundary value problems {(BVPs)}
described by ordinary differential equations {(ODEs)}. Accordingly, the numerical solution of IVPs for deterministic and random ODEs is a basic problem in the sciences. For a review of the state of the art on theory and algorithms for numerical initial value solvers{,} we refer to the monographs~\cite{HairerWanner1993,HairerWanner1996,HairerWannerLubich2006,Ascher95,Stakgold2000,Axelsson20001,KellerBVP18} and the references therein.

Exact solutions may not be available for some ODEs. This has led to the development of a number of methods to estimate the {\it a posteriori} error, 
which is based on the residual of the ODE~\cite{Eriksson95_adaptive},
forming  the basis for adaptive {methods} for ODEs.
The {\it a posteriori} error estimates have been derived for different numerical methods, such as piecewise polynomial collocation methods~\cite{wright2007adaptive,Cheng2021} and Galerkin methods~\cite{logg2003_ODE_I,logg2004_ODE_II,Baccouch16_apost_nonlinear_ODE,Baccouch17_second_IVP}.
An {\it a posteriori} error estimate in connection with adjoint methods was developed in~\cite{Petzold2004aposteriori}.
Kehlet  {et al.}~\cite{Kehlet17} incorporated numerical round-off errors in their {\it a posteriori} estimates.
An {\it a posteriori} error estimate based on the variational principle was derived in~\cite{moon2003variational}.
Convergence rates for the adaptive approximation of ODEs using {\it a posteriori} error estimation were discussed in~\cite{moon2003convergence,moon2005_OSPDE}.
A less common form is the {\it a priori} error estimate~\cite{Johnson88,Eriksson95_adaptive}.
Hybrid {\it a priori}--{\it a posteriori} error estimates for ODEs were developed in~\cite{EF94,estep2012apost_multirate}.
An advantage of the {\it a priori} error estimate over the  {\it a posteriori} error estimate is that the {\it a priori} error estimate does not require the computation of the residual of the ODE.
However, some knowledge about the exact solution of the ODE is required for the {\it a priori} error estimate.
It was shown in~\cite{Stenger2009,StengerYoussefNiebsch2013} that the {\it a priori} error estimate of the Poly-Sinc approximation is exponentially convergent in the number of Sinc points, provided that the exact solution belongs to the set of analytic functions.

We propose an adaptive piecewise method, in which the points in a given partition are used as partitioning points. This piecewise property allows for a greater flexibility of constructing the polynomials of arbitrary degree in each partition.
Recently, we developed an {\it a priori} error estimate for the adaptive method based on piecewise Poly-Sinc interpolation for function approximation~\cite{KEYB2022APNUMSUBMIT}.
In this work~\cite{KEYB2022APNUMSUBMIT}, we used a statistical approach for partition refinement in which we computed the fraction of a standard deviation~\cite{CH79a,CH79b,CH81} as the ratio of the mean absolute deviation to the sample standard deviation.
It was shown in~\cite{Geary35} that the ratio approaches $\sqrt{\frac{2}{\pi}}\approx 0.798$ for an infinite number of normal samples.
We extend the work~\cite{KEYB2022APNUMSUBMIT} for regular and stiff ODEs.
In this paper, we discuss the adaptive piecewise Poly-Sinc method for regular and stiff ODEs, and show that the exponentially convergent {\it a priori} error estimate for our adaptive method differs from that for function approximation~\cite{KEYB2022APNUMSUBMIT} by a small constant.

This paper is organized as follows. Section~\ref{sec:Background} provides an overview of {the} Poly-Sinc approximation, {the} residual computation, {the} indefinite integral approximation, and the collocation method.
Section~\ref{sec:pw_colloc_method} discusses the piecewise collocation method, which is the cornerstone of the adaptive piecewise Poly-Sinc algorithm.
In Section~\ref{sec:adapt_piecewise_Poly-Sinc-ODEs}, we present the adaptive piecewise Poly-Sinc algorithm for ODEs and the statistical approach for partition refinement.
We also demonstrate the exponential convergence of the {\it a priori} error estimate for our adaptive method.
We validate our adaptive Poly-Sinc method on regular ODEs and ODEs whose exact solutions exhibit an interior layer, a boundary layer, and a shock layer in Section~\ref{sec:results}.
Finally, we present our concluding remarks in Section~\ref{sec:conc}.
%%%%%%%%%%%%%%%%%%%%%%%%%%%%%%%%%%%%%%%%%%
%\section{Materials and Methods}
\section{{Background}}\label{sec:Background}
\subsection{{Poly-Sinc Approximation}}%\label{sec:Poly-Sinc}
A novel family of polynomial approximation called Poly-Sinc interpolation which interpolate data of the form $\{x_k,y_k\}_{k=-M}^N$ where $\{x_k\}_{k=-M}^N$ are Sinc points, were derived in~\cite{Stenger2009,Stenger11} and extended in~\cite{StengerYoussefNiebsch2013}.
The interpolation to this type of data is accurate provided that the function $y$ with values $y_k=y(x_k)$ belong to the space  of
analytic functions~\cite{Stenger11,youssef2021}.
For the ease of presentation and discussion,
we assume that $M=N$.
Poly-Sinc approximation was developed in order to mitigate the poor accuracy associated with differentiating the Sinc approximation when approximating the derivative of functions~\cite{Stenger2009}.
Moreover, Poly-Sinc approximation is characterized by its ease of implementation.
Theoretical frameworks on the error analysis of function approximation, quadrature, and the stability of {the} Poly-Sinc approximation were studied in~\cite{Stenger2009,StengerYoussefNiebsch2013,stenger2014lebesgue,youssef2016lebesgue}.
Furthermore, Poly-Sinc approximation was used to solve BVPs in ordinary and partial differential equations~\cite{Youssef2016,Youssef17,youssef2019troesch,youssef2021,KhalilBaumann2021a,KhalilBaumann2021b}.
We start with a brief overview of Lagrange interpolation.
Then, we discuss the generation of Sinc points using conformal mappings.
\subsubsection{{Lagrange Interpolation}}
Lagrange interpolation is a polynomial interpolation scheme~\cite{stoer2002introduction}, which is constructed by Lagrange basis polynomials
\begin{equation*}
	u_k(x)=\dfrac{g(x)}{(x-x_k)g'(x_k)},\quad k=1,2,\ldots,m,
\end{equation*}
where $\{x_k\}_{k=1}^{m}$ are the interpolation points and $g(x)=\prod_{l=1}^{m}(x-x_l)$.
The Lagrange basis polynomials satisfy the property
\begin{equation*}
	u_k(x_j)=\begin{cases*}
		1{,}& if $k=j$, \\
		0{,}     & if $k\neq j$.
	\end{cases*}
\end{equation*}
{Hence}, the polynomial approximation in the Lagrange form can be written as
\begin{equation}\label{eq:LagrPol_ab}
	y_h(x) =\sum_{k=1}^{m}y(x_k)u_k(x),
\end{equation}
where $y_h(x)$ is a polynomial of degree $m-1$ and it interpolates the function $f(x)$ at the interpolation points, i.e., $y_h(x_k)=y(x_k)$. For Sinc points, the polynomial approximation $y_h(x) $ becomes
\begin{equation}\label{eq:LagrPol_ab_Sinc}
	y_h(x) =\sum_{k=-N}^{N}y(x_k)u_k(x),
\end{equation}
where $m=2N+1$ is the number of  Sinc points. If the coefficients $y(x_k)$ are unknown, then we replace $y(x_k)$ with $c_k$, and Equations~\eqref{eq:LagrPol_ab} and~\eqref{eq:LagrPol_ab_Sinc} become
\begin{equation}\label{eq:LagrPol_ab_yc}
	y_{\rm{c}}(x) =\sum_{k=1}^{m}c_k\,u_k(x)
\end{equation}
and
\begin{equation}\label{eq:LagrPol_ab_Sinc_yc}
	y_{\rm{c}}(x) =\sum_{k=-N}^{N}c_k\,u_k(x),
\end{equation}
respectively.

\subsubsection{{Conformal Mappings and Function Space}}\label{sec:conf_map_fn_space}
We introduce some notations related to Sinc methods~\cite{Stenger2009,Stenger11,StengerYoussefNiebsch2013}. Let $ \varphi\ :\ D\to D_d  $ be a conformal map that maps a simply connected region $ D \subset  \mathbb{C}$ onto the strip
\[D_d=\{z \in \mathbb{C}:\vert \Im (z)\vert <d\},
\]
where $ d $ is a given positive number. The region $D$ has a  boundary $\partial D$, and let $a$ and $b$ be two distinct points on $\partial D$.
Let $\psi=\varphi^{-1},\psi \ :\ D_d \to D$ be the inverse conformal map.
Let $ \Gamma $ be an arc defined by

\[
\Gamma=\{z\in[a,b]\ :\ z=\psi(x),\ x \in \mathbb{R} \},
\]
where $ a=\psi(-\infty) $ and $ b=\psi(\infty) $.
For real finite numbers $a,\, b$, and $\Gamma \subseteq {\mathbb{R}}$, $ \varphi(x)={\ln}((x-a)/(b-x))  $ and $ x_k=\psi(kh)=(a+be^{kh})/(1+e^{kh}) $ are the Sinc points with spacing  $  h(d,\beta_{\mathsf{s}})=\left(\dfrac{\pi d}{\beta_{\mathsf{s}} N}\right)^{1/2},\  \beta_{\mathsf{s}}>0 $~\cite{baumann2015sinc,Stenger11}.
Sinc points can be also generated for semi-infinite or infinite intervals.
For a comprehensive list of conformal maps, see~\cite{StengerYoussefNiebsch2013,Stenger11}.

We briefly discuss the function space for $y$.
Let $\rho=e^\varphi$, $ \alpha_{\mathsf{s}} $ be an arbitrary positive integer number, and $ \mathbb{L}_{\alpha_{\mathsf{s}},\beta_{\mathsf{s}}}(D ) $ be the family of all functions that are analytic in $ D=\varphi^{-1}(D_d) $ such that for all $ z\in D $, we have
\[\vert y(z)\vert \leq C\dfrac{\vert \rho(z)\vert ^{\alpha_{\mathsf{s}}}}{[1+\vert \rho(z)\vert ]^{\alpha_{\mathsf{s}}+\beta_{\mathsf{s}}}}.
\]

We next set the restrictions on  $ \alpha_{\mathsf{s}},\beta_{\mathsf{s}}$, and $ d $ such that $ 0<\alpha_{\mathsf{s}}\leq 1 $, $ 0<\beta_{\mathsf{s}}\leq 1 $, and $ 0<d<\pi $. Let $ \mathbb{M}_{\alpha_{\mathsf{s}},\beta_{\mathsf{s}}} (D )$ be the set of all functions $ g $ defined on $ D $ that have finite limits $ g(a)=\lim_{z\to a}g(z) $ and $ g(b)=\lim_{z\to b}g(z) $, where the limits are taken from within $ D $, and such that $ y\in  \mathbb{L}_{\alpha_{\mathsf{s}},\beta_{\mathsf{s}}} (D ) $, where
\[y=g-\dfrac{g(a)+\rho\, g(b)}{1+\rho}.
\]
The transformation guarantees that $y$ vanishes at the endpoints of $(a,b)$.
We assume that $y$ is analytic and uniformly bounded by $ \mathcal{B}(y) $, i.e.,~$ \vert y(x)\vert \leq \mathcal{B}(y) $, in the larger region
\[ D_2=D\cup_{t\in (a,b)}B(t,r),
\]
where $ r>0 $ and $ B(t,r)=\{z\in \mathbb{C}:\vert z-t\vert <r  \} $.

\subsection{Residual}\label{sec:residual}
The residual is used as a measure of the accuracy of the adaptive Poly-Sinc method. The general form of a second-order ODE can be expressed as~\cite{coddington89IntroODE}
\begin{equation}\label{eq:ODE_gen_form}
	F(x,y,y',y'')=0.
\end{equation}
An exact solution $y$ satisfies~\eqref{eq:ODE_gen_form}. If the exact solution $y$ is unknown, we replace it with the approximation $ y_{\rm{c}} $ and Equation~\eqref{eq:ODE_gen_form} becomes

\begin{equation}\label{eq:ODE_gen_form_residual}
	F(x,y_{\rm{c}},y_{\rm{c}}',y_{\rm{c}}'')=R(x),
\end{equation}
where $R(x)$ is the residual.
The residual in~\eqref{eq:ODE_gen_form_residual} for the $i-$th iteration becomes

\begin{equation*}
	F\left(x,y_{\rm{c}}^{(i)},\left(y_{\rm{c}}^{(i)}\right)',\left(y_{\rm{c}}^{(i)}\right)''\right)=R^{(i)}(x),\quad i=1,2,\ldots,\kappa,
\end{equation*}
where $ \kappa $ is the number of iterations.

We will denote the residual for integral and differential equations with $ R_{\mathrm{I}} $ and $ R_{\mathrm{D}} $, respectively.
The residual is used as an indicator for partition refinement as discussed in~{Algorithm}~4 (see Section~\ref{sec:adapt_piecewise_Poly-Sinc-ODEs}).

\subsection{Error Analysis}
We briefly discuss the error analysis for Poly-Sinc approximation over the global interval $ [a,b] $.  
At the end of this section, we will discuss the error analysis of Poly-Sinc approximation for IVPs and BVPs.

For {the} Poly-Sinc approximation on a finite interval~\cite{Stenger2009,StengerYoussefNiebsch2013,stenger2014lebesgue,youssef2016lebesgue,KhalilBaumann2021b}, it was shown that
\begin{equation*}
	{\max_{x\in [a,b]}}\ 	|y(x)-y_h(x)|\leq  \dfrac{A}{r^m}\sqrt{N}e^{-\beta\sqrt{N}},
\end{equation*}
where $ y(x) $ is the exact solution and $ y_h(x) $ is its Poly-Sinc approximation, $A$ is a constant independent of $N$, $ m=2N+1 $ is the number of Sinc points in the interval, $ r $ is the radius of the ball containing the $ m $ Sinc points, and $ \beta >0 $ is the convergence rate parameter. On a finite interval $ [a,b] $, it was shown that~\cite{StengerYoussefNiebsch2013,KhalilBaumann2021b}
\begin{equation}\label{eq:error_Poly_Sinc_ab}
	\max_{x\in [a,b]}\ |y(x)-y_h(x)|\leq A   \left | \dfrac{b-a}{2r}\right|^{m} N^{3/2}\tanh[-4](\dfrac{\eta}{4\sqrt{N}})\exp(-\dfrac{\pi^2 \sqrt{N}}{2\eta}){,}
\end{equation}
{where $\eta$ is a positive constant}.
Inequality~\eqref{eq:error_Poly_Sinc_ab} can be written as
\begin{equation*}
	\max_{x\in [a,b]}\ |y(x)-y_h(x)|\leq A   \left | \dfrac{b-a}{2r}\right|^{m} N^{\alpha}\exp(-\gamma N^\beta).
\end{equation*}
Next, we discuss the collocation method for IVPs and BVPs.

\subsection{Collocation Method}\label{sec:colloc_method}
A collocation method~\cite{Lund92,Stenger93} is a technique in which a system of algebraic equations is constructed from the ODE via the use of collocation points. Here, we adopt the Poly-Sinc collocation method~\cite{youssef2014,youssef2019troesch}, in which the collocation points are the Sinc points and the basis functions are the Lagrange polynomials with Sinc points.

\subsubsection{Initial Value Problem}\label{sec:IVP}
The IVP is transformed into an integral equation.
We briefly discuss the approximation of indefinite integrals using Poly-Sinc methods~(\cite{Baumann2021} \S\,9.3). Define
\begin{equation}\label{eq:indefinite_int_+}
	(\mathcal{J}^+w\,y)(x)=\int_a^x y(t) w(t)\dd{t},
\end{equation}
where the weight function $ w(x) $ is positive on the interval $ (a,b) $ and has the property that the moments $ \int_a^b x^j w(x)\dd{x} $ do not vanish for $ j=0,1,2,\ldots $.
Let $ A^+ $ be an $ m \times m $ matrix whose entries are
\[
[A^+]_{k\,j}=\int_{a}^{x_k}u_j(x)w(x)\dd{x},
\]
where $ {u_j(x)},j=-N,\ldots,N ${,} are the Lagrange basis polynomials stacked in a vector $ L(x)=({u_{-N}(x)},\ldots,{u_N(x)})^\top, $ and $(\cdot)^{\top}$ is the transpose operator.
The interpolation points $\{x_j\}_{j=-N}^N$ are the Sinc points generated in the interval $[a,b]$ as discussed in Section~\ref{sec:conf_map_fn_space}.

Then, the indefinite integral~\eqref{eq:indefinite_int_+} can be approximated as
%https://latex-tutorial.com/align-equations/
\begin{equation*}
	\begin{aligned}
		(\mathcal{J}_m^+w\,y)(x)&=\sum_{j=-N}^{N}({\mathcal{J}_m^+}w\,y)(x_j)u_j(x)	\\
		&\approx \sum_{j=-N}^{N} \left[\sum_{k=-N}^{N} y(x_k)[A^+]_{j\,k}\right]u_j(x)\\
		&=\sum_{j=-N}^{N}u_j(x) \sum_{k=-N}^{N} y(x_k)[A^+]_{j\,k}\\
		&={L(x)^\top A^+Vy,}
	\end{aligned}
\end{equation*}
where $ Vy=(y(x_{-N}),\ldots,y(x_{N}))^\top $.
\sloppy{We state the following theorem for IVPs~\cite{Stenger11}.}
%https://tex.stackexchange.com/questions/13945/cite-in-theorem-environment-argument
\begin{theorem}[{Initial Value Problem~(\cite{Stenger11} \S1.5.8)}]
	If $ y\in \mathbb{M}_{\alpha_{\mathsf{s}},\beta_{\mathsf{s}}}(D) $, then, for all $ N> 1$
	\begin{equation*}
		\lVert\mathcal{J}^+w\,y-\mathcal{J}_m^+w\,y\rVert={O}(\epsilon_N), 		\end{equation*}
	where $ \epsilon_N=\sqrt{N}e^{-\beta\sqrt{N}} $.
\end{theorem}

\subsubsection{Boundary Value Problem}
For a BVP, the collocation method solves for the unknown coefficients $c_k$ in~\eqref{eq:LagrPol_ab_yc} or~\eqref{eq:LagrPol_ab_Sinc_yc} by setting
\[
R_{\mathrm{D}}(x_k)=0,\quad k=-N,\ldots,N.
\]
However, we replace the two equations corresponding to $x_{-N}$ and $x_{N}$ with the boundary conditions $y(a)=y_a$ and $y(b)=y_b$, respectively.
We state the following theorem for BVPs.

%https://tex.stackexchange.com/questions/13945/cite-in-theorem-environment-argument
\begin{theorem}[{Boundary Value Problem~(\cite{Stenger11} \S1.5.6)}]
	If $ y\in \mathbb{M}_{\alpha_{\mathsf{s}},\beta_{\mathsf{s}}} (D) $ and $ \bm{c}=(c_{-N},\ldots,c_N)^\top $ is a complex vector of order $ m $, such that for some $ \delta>0 $,
	\begin{equation*}
		\left(	\sum_{j=-N}^N\left|y(x_j)-c_j\right|^2\right)^{1/2}< \delta,
	\end{equation*}
	then,
	\begin{equation*}
		\lVert y-L(x)^\top\bm{c} \rVert< C\epsilon_N+\delta.
	\end{equation*}
\end{theorem}

\section{{Piecewise Collocation Method}}\label{sec:pw_colloc_method}
We discuss the piecewise collocation method, in which the domain $ I=[a,b] $ is discretized into $ K \in \mathbb{N} $ non-overlapping partitions $ I_n=[x^{n-1},x^{n}),n=1,2,\ldots, K-1 $, $ x^0=a,\ x^K=b $ with $ I_K=[x^{K-1},b] $ and $ \cup_{n=1}^{K}I_n =I=[a,b]$.
The space of piecewise discontinuous polynomials can be defined as
\[\mathfrak{D}_k(I)=\{v:v|_{I_n} \in \mathbb{P}_k(I_n),\ \  n=1,{2,}\ldots,K \},\]
where $ \mathbb{P}_k(I_n) $ denotes the space of polynomials of degree at most $ k $ on $ I_n $.
The piecewise collocation method solves the collocation method in Section~\ref{sec:colloc_method} over partitions. The approximate solution in the global partition $[a,b]$ can be written as

\begin{equation}\label{eq:yh_global}
	y_h(x)=\sum_{k=1}^K	y_{h,\,k}(x)\mathbbm{1}_{x\in I_k}=\sum_{k=1}^K\mathbbm{1}_{x\in I_k}\sum_{j=1}^{m_k} y_{j,\,k} u_{j,\,k}(x),
\end{equation}
where $	y_{h,\,k}(x) = \sum_{j=1}^{m_k} y_{j,\,k}  u_{j,\,k}(x)$
is the Lagrange interpolation in the $ k- $th partition.
The basis functions
\begin{equation*}
	u_{j,\,k}(x)=\dfrac{g_k(x)}{(x-x_{j,\,k})g_k'(x_{j,\,k})},\quad j=1,2,\ldots,m_k,\ k=1,2\ldots, K,
\end{equation*}
where $\{x_{j,\,k}\}_{j=1}^{m_k}$ are the interpolation points in the $k-$th partition, $g_k(x)=\prod_{l=1}^{m_k}(x-x_{l,\,k})$, and $m_k$ is the number of points in the $k-$th partition.
The function  $ \mathbbm{1}_{\mathcal{C}} $  is an indicator function which outputs 1 if the condition $ \mathcal{C} $ is satisfied and otherwise 0.
If the coefficients $y_{j,\,k}$ are unknown, then we replace $y_{j,\,k}$ with $c_{j,\,k}$, and Equation~\eqref{eq:yh_global} becomes
\begin{equation}\label{eq:yc_global}
	y_{\rm{c}}(x)=\sum_{k=1}^K	y_{{\rm{c}},\,k}(x)\mathbbm{1}_{x\in I_k}=\sum_{k=1}^K\mathbbm{1}_{x\in I_k}\sum_{j=1}^{m_k} c_{j,\,k} u_{j,\,k}(x).
\end{equation}

The residual for the $k-$th partition can be written as
\begin{equation*}				F\left(x,y_{{\rm{c}},\,k},\left(y_{{\rm{c}},\,k}\right)',\left(y_{{\rm{c}},\,k}\right)''\right)=R_k(x),\quad x\in I_k,\ k=1,{2,}\ldots,K.
\end{equation*}

The collocation method solves for the unknowns $c_{j,\,k}$  by setting $R_k(x_{j,\,k})=0,${ $\, j=1,{2,}\ldots,m_k, k=1,{2,}\ldots,K$}, which we discuss next for IVPs and BVPs.

\subsection{{Initial Value Problem}}\label{sec:pw_IVP}
In this section, we provide examples for first-order and second-order IVPs.

%\subsubsection*{\upshape \bfseries Relaxation Problem}
\subsubsection*{Relaxation Problem}
We discuss the piecewise collocation method for a first-order IVP in integral form.
Consider the following relaxation or decay Equation \cite{Vesely2001_cpm_phys_intro} on the interval $ [a,b] $
\begin{equation}\label{eq:IVP_relaxation_problem}
	\dfrac{\dd{y(x)}}{\dd{x}}=-\alpha y(x) ,\quad y(a)=y_a,
\end{equation}
\sloppy{where $ \alpha >0 $ is the relaxation parameter.
	The exact solution is
	$y(x)=y_a\exp(-\alpha(x-a)) $. We transform the IVP~\eqref{eq:IVP_relaxation_problem} into an integral form}
\[
y(x)=y_a-\alpha\int_{a}^{x}y(t)\dd{t}.
\]
The  residual becomes
\[
R_{\mathrm{I}}(x)=y_{\rm{c}}(x)-y_a+\alpha\int_{a}^{x}y(t)\dd{t}{.}
\]
We approximate the indefinite integral as discussed in Section~\ref{sec:IVP} and the approximate residual becomes
\[
\tilde{R}_{\mathrm{I}}(x,a,y_a,y_{\rm{c}}(x))=y_{\rm{c}}(x)-y_a+\alpha(\mathcal{J}_m^+y_{\rm{c}})(x).
\]
The domain $ [a,b] $ is partitioned as discussed in Section~\ref{sec:pw_colloc_method}.
For the $ k- $th partition, $ k=1,{2,}\ldots,K $, we replace $ a $ with $ x^{k-1} $ {and $ y_a $ with $y_{{\rm{c}},\,k-1 }(x^{k-1})$. The approximate residual becomes}
\begin{equation*}
	\tilde{R}_{\mathrm{I},\,k}(x,x^{k-1},{y_{{\rm{c}},\,k-1 }(x^{k-1})},{y_{{\rm{c}},\,k }(x))}={y_{{\rm{c}},\,k }(x)}-{y_{{\rm{c}},\,k-1 }(x^{k-1})}+\alpha(\mathcal{J}_m^+y_{\rm{c}})(x),\qquad x \in I_k.
\end{equation*}

We remove the equation corresponding to the leftmost Sinc point in each partition, and replace them with the conditions
\begin{subequations}
	\begin{align}
		y_{\rm{c}}(x^0) &= y_a, \label{eq:colloc_cond_IVP_relax_problem_IC} \\
		y_{{\rm{c}},\,k }(x^{k-1}) &=y_{{\rm{c}},\,k-1 }(x^{k-1}),\quad  k=2,{3,}\ldots,K,\label{ctuity_IVP_relax_problem}
	\end{align}
\end{subequations}
and the set of equations
\begin{equation}\label{eq:residual_IVP_algorithm_relax_problem}
	\tilde{R}_{\mathrm{I},\,k}(x_{j,\,k},x^{k-1},y_{{\rm{c}},\,k-1 }(x^{k-1})  ,y_{{\rm{c}},\,k}(x_{j,\,k}))=0,\  j=2,{3,}\ldots,m_k,k=1,{2,}\ldots,K.
\end{equation}
The set of equations~\eqref{ctuity_IVP_relax_problem} is known as the continuity equations at the interior boundaries~\cite{CF_75_ortho_colloc}.
The collocation algorithm for the IVP~\eqref{eq:IVP_relaxation_problem} is outlined in~Algorithm~\ref{alg:piecewise_colloc_IVP_relax_problem}.

\vspace{12pt}
\begin{algorithm}[H]
	
	%https://tex.stackexchange.com/questions/153646/algorithm2e-disabling-line-numbers-for-specific-lines
	\let\oldnl\nl% Store \nl in \oldnl
	\DontPrintSemicolon
	\LinesNumberedHidden
	\SetAlgoHangIndent{0pt}
	\newcommand{\whilecond}{$A^{(i-1)}> \varepsilon_{\mathrm{stop}}$}
	\SetKwInOut{Input}{input}
	\SetKwInOut{Output}{output}
	\SetKwProg{Init}{init:}{}{}
	\Input{$ K: $ number of partitions\\
		$ m_k: $ number of Sinc points in the $ k- $th partition
	}
	\Output{ $ y_{\rm{c}}(x) $:	approximate solution}
	
	Replace $ y(x) $ with the global approximate solution~\eqref{eq:yc_global}.\;
	Solve for the	$ m_k\,K $ unknowns $ \{c_{j,k} \}_{j=1,\,k=1}^{m_k,\,K}$ using the
	initial condition~\eqref{eq:colloc_cond_IVP_relax_problem_IC}, continuity {Equation}~\eqref{ctuity_IVP_relax_problem}, and the set of equations for the residual~\eqref{eq:residual_IVP_algorithm_relax_problem}.
	\caption{ Piecewise Poly-Sinc Algorithm (IVP~\eqref{eq:IVP_relaxation_problem}). }
	\label{alg:piecewise_colloc_IVP_relax_problem}
\end{algorithm}

% \subsubsection*{\upshape \bfseries Hanging Bar Problem}
\subsection{{Hanging Bar Problem}}\label{eq:IVP_2nd_order}
We discuss the piecewise collocation method for a second-order IVP in integral form.
Considering the following IVP on the interval $ [a,b] $
\begin{equation}\label{eq:IVP_hanging_bar}
	\begin{aligned}
		\forall x\in \left [ a,b \right ] ,\quad -(\tilde{K}(x)y'(x))'&=f(x),\\
		y(a)&=y_a,\\
		y'(a)&=\tilde{y}_a.
	\end{aligned}
\end{equation}
where $ y(x) $ is the sought-for solution.
In the context of the hanging bar problem~\cite{strang2007computational}, $ y(x) $ and $ \tilde{K}(x) $ are the displacement and the material property of the bar at the position $ x $, respectively.
{For simplicity, we set $\tilde{K}(x)=1$.}
Equation~\eqref{eq:IVP_hanging_bar} can be written as a system of first-order equations
\begin{align}\label{eq:IVP_hanging_bar_first_order}
	y'=q,\quad q'=-f,\quad y(a)=y_a,\ q(a)=\tilde{y}_a.
\end{align}
The integral form of~\eqref{eq:IVP_hanging_bar_first_order} is
\begin{align}
	y(x)&=y_a+\int_a^xq(t)\dd{t},\label{eq:IVP_hanging_bar_first_order_a}\\
	q(x)&=\tilde{y}_a -\int_a^xf(t)\dd{t}\label{eq:IVP_hanging_bar_first_order_b}.
\end{align}
Plugging~\eqref{eq:IVP_hanging_bar_first_order_b} into~\eqref{eq:IVP_hanging_bar_first_order_a}, we obtain
\begin{equation*}
	\begin{split}
		y(x)&=y_a+\int_a^x\left[\tilde{y}_a -\int_a^{{t}}f(s)\dd{s}\right]\dd{t}\\
		&=y_a+\tilde{y}_a(x-a)-\int_a^x\int_a^t f(s)\dd{s}\dd{t}.
	\end{split}
\end{equation*}
Using integration by parts $ \int u\dd{v}=uv-\int v\dd{u} $~\cite{Wazwaz2011} with $ u(t)=\int_a^t f(s)\dd{s} $, $ \ $
\begin{align*}
	\int_a^x \int_a^t f(s)\dd{s}\dd{t}&=t\int_a^tf(s)\dd{s}\bigg|_{t=a}^{t=x} -\int_a^x tf(t)\dd{t}\\
	&=x\int_a^xf(s)\dd{s}-\int_a^xtf(t)\dd{t}\\
	&=x\int_a^xf(s)\dd{s}-\int_a^xsf(s)\dd{s}\\
	&=\int_a^x(x-s)f(s)\dd{s},
\end{align*}
where we set $ t=s $.
Thus, the integral form of the solution to~\eqref{eq:IVP_hanging_bar} becomes
\begin{equation*}
	y(x)=y_a+(x-a)\tilde{y}_a-x\int_a^xf(s)\dd{s}+\int_a^xsf(s)\dd{s}.
\end{equation*}
The residual can be written as
\begin{equation}\label{eq:IVP_hanging_bar_integral_form_residual}
	R_{\mathrm{I}}(x,a,y_a,\tilde{y}_a,y_{\rm{c}}(x))=	y_{\rm{c}}(x)-y_a-(x-a)\tilde{y}_a+x\int_a^xf(s)\dd{s}-\int_a^xsf(s)\dd{s}.
\end{equation}
Approximating the indefinite integral in~\eqref{eq:IVP_hanging_bar_integral_form_residual}, the approximate residual becomes
\begin{equation}\label{eq:IVP_hanging_bar_integral_form_residual_approx}
	\tilde{R}_{\mathrm{I}}(x,a,y_a,\tilde{y}_a,y_{\rm{c}}(x))=	y_{\rm{c}}(x)-y_a-(x-a)\tilde{y}_a+x(\mathcal{J}_m^+f)(x)-(\mathcal{J}_m^+x\,f)(x).
\end{equation}
For the $ k- $th partition, $ k=1,2,\ldots,K $, we replace $ a $ with $ x^{k-1} $ {, $ y_a $ with $ y_{{\rm{c}},\,k-1 }(x^{k-1}) $, and $ \tilde{y}_a $ with $y_{{\rm{c}},\,k-1 }'(x^{k-1})  $. The approximate residual becomes}
\begin{multline*}
	\tilde{R}_{\mathrm{I},\,k}(x,x^{k-1},{y_{{\rm{c}},\,k-1 }(x^{k-1}) },{y_{{\rm{c}},\,k-1 }'(x^{k-1}) },{y_{{\rm{c}},\,k}(x)})=	{y_{{\rm{c}},\,k}(x)}-{y_{{\rm{c}},\,k-1 }(x^{k-1}) }\\-(x-{x^{k-1}}){y_{{\rm{c}},\,k-1 }'(x^{k-1}) }+x(\mathcal{J}_m^+f)(x)-(\mathcal{J}_m^+x\,f)(x),\ x \in I_k.
\end{multline*}

%The collocation algorithm for the IVP~\eqref{eq:IVP_hanging_bar} is outlined in~Algorithm~\ref{alg:piecewise_colloc_IVP_hanging_bar}.
%The domain $ [a,b] $ is partitioned as discussed in Section~\ref{sec:pw_colloc_method}.
%In the $ k- $th partition, the initial conditions $ y_a $ and $ \tilde{y}_a $ in~\eqref{eq:IVP_hanging_bar_integral_form_residual_approx} are replaced with $ y_{{\rm{c}},\,k-1 }(x^{k-1}) $ and $y_{{\rm{c}},\,k-1 }'(x^{k-1})  $ for $ k=2,\ldots,K $, respectively.

We remove the equations corresponding to the leftmost and rightmost Sinc points in each partition, and replace them with the conditions
\begin{subequations}\label{eq:colloc_cond_IVP}
	\begin{align}
		{y_{\rm{c}}(x^0)} &= y_a, \label{eq:colloc_cond_IVP_ICa}\\
		{y_{\rm{c}}'(x^0)} &= \tilde{y}_a, \label{eq:colloc_cond_IVP_ICb}\\
		y_{{\rm{c}},\,k }(x^{k-1}) &=y_{{\rm{c}},\,k-1 }(x^{k-1}),\quad  k=2,\ldots,K,\label{ctuity_a_IVP}\\
		y_{{\rm{c}},\,k }'(x^{k-1}) &=y_{{\rm{c}},\,k-1 }'(x^{k-1}),\quad k=2,\ldots,K\label{ctuity_b_IVP},
	\end{align}
\end{subequations}
and the set of equations
\begin{equation}\label{eq:residual_IVP_algorithm}
	\tilde{R}_{\mathrm{I},\,k}(x_{j,\,k},x^{k-1},y_{{\rm{c}},\,k-1 }(x^{k-1}) ,y_{{\rm{c}},\,k-1 }'(x^{k-1}) ,y_{{\rm{c}},\,k }(x_{j,\,k}))=0,
\end{equation}
$ j=2,\ldots,m_k-1,k=1,\ldots,K $.
Equations~\eqref{ctuity_a_IVP}--\eqref{ctuity_b_IVP} are known as the continuity equations at the interior boundaries~\cite{CF_75_ortho_colloc}.
The collocation algorithm for the IVP~\eqref{eq:IVP_hanging_bar} is outlined in~Algorithm~\ref{alg:piecewise_colloc_IVP_hanging_bar}.
\vspace{12pt}
\begin{algorithm}[H]
	
	%https://tex.stackexchange.com/questions/153646/algorithm2e-disabling-line-numbers-for-specific-lines
	\let\oldnl\nl% Store \nl in \oldnl
	\newcommand{\nonl}{\renewcommand{\nl}{\let\nl\oldnl}}% Remove line number for one line
	\DontPrintSemicolon
	\LinesNumberedHidden
	\SetAlgoHangIndent{0pt}
	\newcommand{\whilecond}{$A^{(i-1)}> \varepsilon_{\mathrm{stop}}$}
	\SetKwInOut{Input}{input}
	\SetKwInOut{Output}{output}
	\SetKwProg{Init}{init:}{}{}
	\Input{$ K: $ number of partitions\\
		$ m_k: $ number of Sinc points in the $ k- $th partition
	}
	\Output{ $ y_{\rm{c}}(x) $:	approximate solution}
	
	Replace $ y(x) $ with the global approximate solution~\eqref{eq:yc_global}.\;
	Solve for the	$ m_k\,K $ unknowns $ \{c_{j,k} \}_{j=1,\,k=1}^{m,\,K}$ using
	initial conditions~\eqref{eq:colloc_cond_IVP_ICa}--\eqref{eq:colloc_cond_IVP_ICb}, 
	continuity {Equations}~\eqref{ctuity_a_IVP}--\eqref{ctuity_b_IVP}, and the set of equations for the residual~\eqref{eq:residual_IVP_algorithm}.
	\caption{ Piecewise Poly-Sinc Algorithm (IVP~\eqref{eq:IVP_hanging_bar}). }\label{alg:piecewise_colloc_IVP_hanging_bar}
\end{algorithm}

\subsection{{Boundary Value Problem}}\label{sec:pw_BVP}
The collocation method for the BVP is similar to that of the IVP in Section~\ref{eq:IVP_2nd_order}, except that we replace the set of equations~\eqref{eq:colloc_cond_IVP} with
\begin{subequations}
	\begin{align}
		y_{\rm{c}}(x^0) &= y_a{,}\label{eq:colloc_cond_BVP_ICa}\\
		y_{\rm{c}}(x^K) &= y_b{,} \label{eq:colloc_cond_BVP_ICb}\\
		y_{{\rm{c}},\,k }(x^{k-1}) &=y_{{\rm{c}},\,k-1 }(x^{k-1}),\quad  k=2,{3,}\ldots,K{,}\label{ctuity_a_BVP}\\
		y_{{\rm{c}},\,k }'(x^{k-1}) &=y_{{\rm{c}},\,k-1 }'(x^{k-1}),\quad k=2,{3,}\ldots,K\label{ctuity_b_BVP},
	\end{align}
\end{subequations}
and the set of equations of the residual for the BVP becomes
\begin{equation}\label{eq:residual_BVP_algorithm}
	R_{\mathrm{D},\,k}(x_{j,\,k})=0,\  j=2,{3,}\ldots,m_k-1,k=1,{2,}\ldots,K,
\end{equation}
where $ R_{\mathrm{D},\,k} $ is the residual of the differential equation in the $ k- $th partition.
The piecewise Poly-Sinc collocation algorithm for the BVP is outlined in~Algorithm~\ref{alg:piecewise_colloc_BVP}.

\vspace{10pt}
\begin{algorithm}[H]
	%https://tex.stackexchange.com/questions/153646/algorithm2e-disabling-line-numbers-for-specific-lines
	\let\oldnl\nl% Store \nl in \oldnl
	\newcommand{\nonl}{\renewcommand{\nl}{\let\nl\oldnl}}% Remove line number for one line
	\DontPrintSemicolon
	\LinesNumberedHidden
	\SetAlgoHangIndent{0pt}
	\newcommand{\whilecond}{$A^{(i-1)}> \varepsilon_{\mathrm{stop}}$}
	\SetKwInOut{Input}{input}
	\SetKwInOut{Output}{output}
	\SetKwProg{Init}{init:}{}{}
	\Input{$ K: $ number of partitions\\
		$ m_k: $ number of Sinc points in the $ k- $th partition
	}
	\Output{ $ y_{\rm{c}}(x) $:	approximate solution}
	
	Replace $ y(x) $ with the global approximate solution~\eqref{eq:yc_global}.\;
	Solve for the	$ m_k\,K $ unknowns $ \{c_{j,k} \}_{j=1,\,k=1}^{m,\,K}$ using
	boundary conditions~\eqref{eq:colloc_cond_BVP_ICa}--\eqref{eq:colloc_cond_BVP_ICb}, 
	continuity {Equations}~\eqref{ctuity_a_BVP}--\eqref{ctuity_b_BVP}, and the set of equations for the residual~\eqref{eq:residual_BVP_algorithm}.
	\caption{ Piecewise Poly-Sinc Algorithm (BVP). }\label{alg:piecewise_colloc_BVP}
\end{algorithm}

\section{Adaptive Piecewise Poly-Sinc Algorithm}\label{sec:adapt_piecewise_Poly-Sinc-ODEs}
This section introduces the greedy algorithmic approach used in adaptive piecewise Poly-Sinc methods. The core feature used is the non-overlapping properties of Sinc points and the uniform exponential convergence on each partition of the approximation interval.
Greedy algorithms seek the ``best'' candidate of possible solutions at a given step~\cite{cormen2009introalg}.
Greedy algorithms have been applied to model order reduction for parametrized partial differential equations~\cite{Haasdonk2008,GREPL201233}.
The adaptive piecewise Poly-Sinc algorithm is greedy in the sense that it makes a choice that aims to find the ``best'' approximation for the solution of the ODE in the current step~\cite{cormen2009introalg}.
The algorithm takes an iterative form in which it computes the $L^2$ norm values of the residual for all partitions constituting the global interval $I=[a,b]$.
At the $i-$th step, the algorithm refines the partitions for which the $L^2$ norm values of the residual are relatively large.
By refining the partitions as discussed above, it is expected that the 
mean value of the $L^2$ norm values over all partitions decreases in each step.
As the iteration proceeds, the algorithm expects to find the ``best'' polynomial approximation for the solution of the ODE.

\subsection{Algorithm Description}
We discuss the adaptive algorithm for {the} piecewise Poly-Sinc approximation.
The following steps of the adaptive algorithm are performed in an iterative loop~\cite{DeVore_2009_intro_theory_adaptive}
\begin{equation*}
	\mathsf{SOLVE} \to 	\mathsf{ESTIMATE} \to	\mathsf{MARK} \to	\mathsf{REFINE}.
\end{equation*}

The adaptive piecewise Poly-Sinc algorithm is outlined in~Algorithm~\ref{alg:adaptive_polySinc_alg_ODE}.
The refinement strategy is performed as follows.
For the $ i- $th iteration, we compute the set of $ L^2 $ norm values 
$ \left\{\left\|  R_{{k}}^{(i)}(x)\right\|_{L^2(I_k^{(i)})} \right\}_{k=1}^{K_i} $ over the $ K_i $ partitions, from which the sample mean $ \overline{R_i}=\frac{1}{K_i} \sum_{j=1}^{K_i}\left\|  R_j^{(i)}(x)\right\|_{L^2(I_j^{(i)})} $ and the sample standard deviation~\cite{Walpole2021}

\[s_i=\frac{1}{\sqrt{K_i-1}}\sqrt{\sum_{j=1}^{K_i}\left(\left\|  R_j^{(i)}(x)\right\|_{L^2(I_j^{(i)})} -\overline{R_i}\right)^2}\]
are computed~\cite{CH79a,CH79b,CH81}.
The residual $ R_j^{(i)}(x) $ for the $ j- $th partition and the $ i- $th iteration  is discussed in Section~\ref{sec:residual}.
\sloppy{The partitions with the indices  $\mathfrak{I}_i= \left\{j:\left\|  R_j^{(i)}(x)\right\|_{L^2(I_j^{(i)})} -\overline{R_i} \geq \omega_i\,s_i\right\} $ are marked for refinement, where the statistic~\cite{Geary35}
	\begin{equation*}
		\omega_i=\dfrac{\frac{1}{K_i}\sum_{j=1}^{K_i}\left|\left\|  R_j^{(i)}(x)\right\|_{L^2(I_j^{(i)})} -\overline{R_i}\right| }{s_i }.
	\end{equation*}
	Using H\"{o}lder's inequality for sums with $p=q=2$~(\cite{Abramowitz_Stegun72} \S\,3.2.8), one can show that $\omega_i\leq \sqrt{\frac{K_i-1}{K_i}}<1 $.
}
We restrict to second-order moments only.
The points in the partitions with the indices $ \mathfrak{I}_i $ are used as partitioning points and $ m=2N+1 $ Sinc points are inserted in the newly created partitions.
The algorithm terminates when the stopping criterion is satisfied.
The approximate solution $ y_{\rm{c}}^{(i)}(x),i=1,\ldots,\kappa ${,} for the $ i- $th iteration is computed using the collocation method
outlined in~Algorithms~\ref{alg:piecewise_colloc_IVP_relax_problem} and~\ref{alg:piecewise_colloc_IVP_hanging_bar} for IVPs
and~Algorithm~\ref{alg:piecewise_colloc_BVP} for BVPs.
We note that, for partition refinement, the residual is computed in its differential form $R_{\rm{D}}(x)$.

The definite integral in the $L^2$ norm~\cite{CruzUribe2013} is numerically computed using a Sinc quadrature~\cite{Stenger11}, i.e.,
\begin{equation*}
	\lVert f(x)\rVert_{L^2([a,b])}^2=\int_a^b |f(x)|^2 \dd{x}\approx h\sum_{k=-N}^N \dfrac{1}{\varphi'(x_k)}f^2(x_k),
\end{equation*}
where $\{x_k\}_{k=-N}^N\in [a,b]$ are the quadrature points, which are also Sinc points, and $\varphi(x)$ is the conformal mapping in Section~\ref{sec:conf_map_fn_space}.
The supremum norm on an interval $I=[a,b]$ is approximated as~(\cite{Gautschi2012book} Table~2.1)
\begin{equation*}
	\lVert f(x)\rVert_{I}\approx\max\ \{|f(x_k)|\}_{k=-N}^N,
\end{equation*}
where $\{x_k\}_{k=-N}^N $ are the Sinc points on $I$, whose generation is discussed in Section~\ref{sec:conf_map_fn_space}.

\subsection{Error Analysis}

We state below the main	theorem.

\begin{theorem}[Estimate of Upper Bound~\cite{KEYB2022APNUMSUBMIT}]
	Let $ y $ be in $  \mathbb{M}_{\alpha_{\mathsf{s}},\beta_{\mathsf{s}}}(\varphi) $, analytic and bounded in $ D_2 $, and let $ y_h^{(i)}(x) $ be the piecewise Poly-Sinc approximation in the $ i $-th iteration.
	Let $ \xi_i=\argmax_{k}  \left|\mathsf{Len}(I_k^{(i)})\right| $ be the index of the largest partition in the $ i- $th iteration and $\mathsf{Len}(I_k^{(i)}) $ be the length of the $k-$th partition in the $i-$th iteration.
	Let $ K_i $ be the number of partitions in the $ i- $th iteration.
	Then, there exists a constant $ A $, independent of the $ i- $th iteration, such that
	\begin{equation*}
		\max_{x\in [a,b]} \left|y(x)- y_h^{(i)}(x)\right|\leq K_i\, E_i,
	\end{equation*}
	where $ E_i=\dfrac{A}{(2r_{\xi_i})^{m_{\xi_i}}}\,\lambda^{m_{\xi_i}(i-1)}(b-a)^{m_{\xi_i}} $.
\end{theorem}
For fitting purposes, we compute the mean value of the error estimate, i.e.,
\begin{equation*}
	\max_{x\in [a,b]} \dfrac{1}{K_i}	\left|y(x)- y_h^{(i)}(x)\right|\leq E_i.
\end{equation*}
We state the following theorem on collocation.

\vspace{11pt}
\begin{algorithm}[H]
	%https://tex.stackexchange.com/questions/153646/algorithm2e-disabling-line-numbers-for-specific-lines
	\let\oldnl\nl% Store \nl in \oldnl
	\newcommand{\nonl}{\renewcommand{\nl}{\let\nl\oldnl}}% Remove line number for one line
	\DontPrintSemicolon
	%	\LinesNumberedHidden
	\SetAlgoHangIndent{0pt}
	\newcommand{\whilecond}{$\overline{R_{i-1}}> \varepsilon_{\mathrm{stop}}$}
	\SetKwInOut{Input}{input}
	\SetKwInOut{Output}{output}
	\SetKwProg{Init}{init:}{}{}
	\Input{Global partition $ [a,b] $,  threshold $ \varepsilon_{\mathrm{stop}} $, $ N $
	}
	\Output{$ \mathsf{S} $: set of points \\
		$ \kappa $: number of iterations \\
		$ \mathsf{R}=\{\overline{R_i}\}_{i=1}^\kappa $: set of mean values of $ \lVert R_k^{(i)}(x)\rVert_{L^2(I_k^{(i)})},\,k=1,\ldots,K_i,\,i=1,\ldots,\kappa $ \\
		$ y_{\rm{c}}^{(\kappa)}(x) $:	approximate solution}
	
	\Init{ $\mathsf{S}=\{\}  $,
		$ \mathsf{P}=\{x^0=a,x^1=b\} $,
		$  \mathsf{R}=\{\}$\;
		\normalfont  \Indp Start with $ 2N +1$ points in the global interval $ [a,b] $. Append the points to $ \mathsf{S} $ and $ \mathsf{P}. $\;
		\nonl	\textbf{(Solve).} Compute $ y_{\rm{c}}^{(i)}(x)$ for IVP (Sections~\ref{sec:pw_IVP} and~\ref{eq:IVP_2nd_order}) or BVP (Section~\ref{sec:pw_BVP}).\;
		\textbf{(Estimate).} 	Compute $ \{\| R^{(1)}(x)\| \}_{k=1}^{K_1=1} $ over the global interval  $ [a,b] $. Append the mean value $ \overline{R_1} $ in $ \mathsf{R} $. 	\;
		\textbf{(Mark).} 	Set $ \mathcal{I}_1=1 $.\;
		\textbf{(Refine). } Use the points in the partition with index $ \mathcal{I}_1=1 $ as partitioning points. Insert $ 2N +1$ Sinc points in each of the newly created partitions. Update $ \mathsf{S} $ and $ \mathsf{P} $.\;
		\Indm \hspace{1em}  Set $ i=2 $.
	}{}
	\While{\whilecond}{
		\textbf{(Solve).} Compute $ y_{\rm{c}}^{(i)}(x)$ for IVP
		(Sections~\ref{sec:pw_IVP} and~\ref{eq:IVP_2nd_order}) or BVP (Section~\ref{sec:pw_BVP}).\;
		\textbf{(Estimate).}  Compute $ \left\{\left\|  R_k^{(i)}(x)\right\|_{L^2(I_k^{(i)})} \right\}_{k=1}^{K_i} $ over $ K_i $ partitions. Compute the sample  mean value $ \overline{R_i}  $ and the sample standard deviation $ s_i $.
		Append $ \overline{R_i}  $ in $ \mathsf{R} $.\;
		\textbf{(Mark).}	Identify the partitions with indices $\mathfrak{I}_{i}= \left\{j:\left\|  R_j^{(i)}(x)\right\|_{L^2(I_j^{(i)})} -\overline{R_i} \geq \omega_i \, s_{i}\right\},\ j=1,\ldots,K_i $.\;
		\textbf{(Refine).} Use the $ 2N +1$ points in the partitions with indices $ \mathfrak{J}_i $ as partitioning points. Insert $ 2N +1$ points in each of the newly created partitions. Update $ \mathsf{S} $ and $ \mathsf{P}. $\;
		$ i\leftarrow i+1$.\;
	}
	$ \kappa\leftarrow i-1$.\;
	
	\caption{Adaptive Piecewise Poly-Sinc Algorithm.}\label{alg:adaptive_polySinc_alg_ODE}
\end{algorithm}

%https://tex.stackexchange.com/questions/13945/cite-in-theorem-environment-argument
\begin{theorem}
	Let $ y $ be in $  {  \mathbb{M}_{\alpha_{\mathsf{s}},\beta_{\mathsf{s}}}(\varphi) } $, analytic and bounded in $ D_2 $, and let $ y_{\rm{c}}^{(i)}(x) $ be the piecewise Poly-Sinc approximation in the $ i $-th iteration with the estimated coefficients $c_{j,\,k}^{(i)}$ using the {piecewise} collocation method in Section~\ref{sec:pw_colloc_method}.
	Let $m_k^{(i)}$ be the number of points in the $k-$th partition for the $i-$th iteration and $n_i=\sum_{k=1}^{K_i}m_k^{(i)}$ be the total number of points in the  $i-$th iteration.
	Let $\bm{c}_i=(c_{1,1}^{(i)},\ldots,c_{1,m_1^{(i)}}^{(i)},\ldots,c_{1,K_i}^{(i)},\ldots,c_{1,m_{K_i}^{(i)}}^{(i)})$ be the $n_i\times 1$ vector of estimated coefficients in $ y_{\rm{c}}^{(i)}(x) $ and $\bm{y}_i=y(x_{j,\,k}^{(i)})$ be the corresponding vector of the exact values of $ y(x) $ at the Sinc points $\{x_{j,\,k}^{(i)}\}$ with
	\[
	\|\bm{c}_i-\bm{y}_i\|_1<\delta,\qquad i=1,2,\ldots,\kappa.
	\]
	Then,
	\begin{equation*}
		\max_{x\in [a,b]} |y(x)-y_{\rm{c}}^{(i)}(x)|\leq K_i\, E_i+\delta  \left(\frac{1}{\pi}\ln(\overline{\overline{m}})+1.07618\right),
	\end{equation*}
	where $ \overline{\overline{m}}= \max_{k,\,i}\{m_k^{(i)}\}$ and $\|\cdot\|_1$ denotes the $\ell_1$ norm~(\cite{Horn2013book} Ch.~5).
\end{theorem}
\begin{proof}
	The derivation follows that of~(\cite{Stenger11} \S\,1.5.6).
	%https://tex.stackexchange.com/questions/358608/split-equations-inside-the-align-environment
	\begin{align*}
		\max_{x\in [a,b]} |y(x)-y_{\rm{c}}^{(i)}(x)|&=\|y(x)-y_{\rm{c}}^{(i)}(x)\|_I\\
		&=\|y(x)-y_{\rm{c}}^{(i)}(x)+y_{h}^{(i)}(x)-y_{h}^{(i)}(x)\|_I\\
		&\leq \|y(x)-y_{h}^{(i)}(x)\|_I+\|y_{\rm{c}}^{(i)}(x)-y_{h}^{(i)}(x)\|_I\\
		\begin{split}
			&\leq \bigg \|y(x)-{y_{h}^{(i)}(x)}\bigg \|_I+\\
			&\quad \bigg \|\sum_{k=1}^{K_i}\mathbbm{1}_{x\in {I_k^{(i)}}}\sum_{j=1}^{m_k^{(i)}} (c_{j,\,k}^{(i)}-y_{j,\,k}^{(i)}) u_{j,\,k}^{(i)}(x)\bigg \|_I\\
		\end{split}\\
		&\leq K_i\,E_i+\sum_{k=1}^{K_i}\bigg \|\sum_{j=1}^{m_k^{(i)}} (c_{j,\,k}^{(i)}-y_{j,\,k}^{(i)}) u_{j,\,k}^{(i)}(x)\bigg \|_{I_k^{(i)}}\\
		&\leq K_i\,E_i +\sum_{k=1}^{K_i}\bigg \|\sum_{j=1}^{m_k^{(i)}} |c_{j,\,k}^{(i)}-y_{j,\,k}^{(i)}| \cdot |u_{j,\,k}^{(i)}(x)|\bigg \|_{I_k^{(i)}}\\
		&\leq K_i\,E_i+\frac{\delta}{K_i}\sum_{k=1}^{K_i} \bigg \| \sum_{j=1}^{m_k^{(i)}} |u_{j,\,k}^{(i)}(x)|\bigg \|_{I_k^{(i)}}\\
		&\leq K_i\,E_i+\frac{\delta}{K_i}\sum_{k=1}^{K_i}\left( \frac{1}{\pi}\ln(m_k^{(i)})+1.07618\right)\\
		&\leq K_i\,E_i+\frac{\delta}{K_i} K_i  \left(\frac{1}{\pi}\ln(\overline{\overline{m}})+1.07618\right)\\
		&=K_i\,E_i+ \delta\left(\frac{1}{\pi}\ln(\overline{\overline{m}})+1.07618\right),
	\end{align*}
	where $\big \| \sum_{j=1}^{m_k^{(i)}} |u_{j,\,k}^{(i)}(x)|\big \|_{I_k^{(i)}}\approx \frac{1}{\pi}\ln(m_k^{(i)})+1.07618$ is the Lebesgue constant for Poly-Sinc approximation~\cite{youssef2016lebesgue,youssef2016multivariate,youssef2021}. On average, the term $|c_{j,\,k}^{(i)}-y_{j,\,k}^{(i)}|<\frac{\delta}{n_i}<\frac{\delta}{K_i \ \min_k\{m_k^{(i)}\}}<\frac{\delta}{K_i}$.
\end{proof}
For fitting purposes, we compute the mean value of the error estimate, i.e.,
\begin{equation}\label{eq:mainthm_gen_coll}
	\begin{aligned}
		{\max_{x\in [a,b]}}\ \dfrac{1}{K_i}	\left|y(x)- y_{\rm{c}}^{(i)}(x)\right|&\leq E_i+\frac{\delta}{K_i}\left(\frac{1}{\pi}\ln(\overline{\overline{m}})+1.07618\right)\\
		&<E_i+\delta\left(\frac{1}{\pi}\ln(\overline{\overline{m}})+1.07618\right).
	\end{aligned}
\end{equation}

%%%%%%%%%%%%%%%%%%%%%%%%%%%%%%%%%%%%%%%%%%
\section{Results}\label{sec:results}

The results in this section were computed using {\it Mathematica}~\cite{Mathematica130}.
We tested our adaptive algorithm on regular and stiff ODEs. The Sinc spacing is $h=\frac{\pi}{\sqrt{e\, N}}$. For all examples, we set $e=1/2$ and the number of points per partition to be constant, i.e., $m_k^{(i)}=m=2N+1$.
A precision of $ 200 $ digits is used.

% \GB{We observe that the processing times differ amongst the many problems we look at.
	% This is due to two factors: first, the fact that we are aiming for a very exact outcome; and second, the fact {that there are many sorts of challenges}.
	% Therefore, listing the computing time in seconds is meaningless since it would only indicate the computational power used on our machine, which will vary for various users.
	% Overall, it is evident that using adaptive techniques takes longer than using a simple collocation approach or FEM method with a lower accuracy.
	% As a result, our goal is to complete the computation in the most accurate way feasible rather than the quickest way possible.
	% This will inevitably lengthen the processing time for some problems, such as stiff or layer problems. }

\subsection{Norms}

The supremum norm has a theoretical advantage. However, its computation is slower than that of the $L^2$ norm~\cite{KEYB2022APNUMSUBMIT}.
Hence, we use {the} $L^2$ norm in our computations.

\subsection{Initial Value Problem}
We test our adaptive piecewise Poly-Sinc algorithm on regular first-order and second-order IVPs.
\begin{example}[Relaxation  Problem]
	We start with the relaxation problem in Section~\ref{sec:pw_IVP}.
	We set $ a=0 $ and the exact solution becomes $y(x)=\exp(-\alpha\,x) $.  We set the exponential decay parameter $ \alpha=20 $ and confine the domain of the solution to the interval $ [0,1] $.
	The approximate solution $ y_{\rm{c}}(x) $ is computed as discussed in Section~\ref{sec:pw_IVP}.
	
	We set the number of Sinc points to be inserted in all partitions as $m= 2N+1=5$.
	The stopping criterion $ \varepsilon_{\mathrm{stop}}=10^{-6} $ was used.
	The algorithm terminates after $\kappa=7 $ iterations and the number of points $|{\mathsf{S}}|=530$.
	
	Figure~\ref{fig1}a shows the approximate solution $ y_{\rm{c}}^{(7)}(x) $.
	A proper subset of the set of points $ \mathsf{S} $ is shown as red dots, which are projected onto the approximate solution $ y_{\rm{c}}^{(7)}(x) $.
	This proper subset is used to observe the approximate solution $ y_{\rm{c}}^{(7)}(x) $.
	We plot the statistic $\omega_i$ as a function of the iteration index 	$i,\,i=2,{3,}\ldots,7$, in Figure~\ref{fig1}b.
	The oscillations are decaying and the statistic $\omega_i$ is converging to an asymptotic value. The mean value $ \overline{\omega_i}\approx 0.6 $ is denoted by a horizontal line.

	We perform the least-squares fitting of the logarithm of the set $ \mathsf{R}$ to the logarithm of the upper bound~\eqref{eq:mainthm_gen_coll}.
	Figure~\ref{fig2}a shows the least-squares fitted model~\eqref{eq:mainthm_gen_coll} to the set $ \mathsf{R}$.
	The dots represent the set $\mathsf{R}$ and the solid line represents the least-squares fitted model~\eqref{eq:mainthm_gen_coll}.
	Figure~\ref{fig2}b shows the residual, absolute local approximation error, and the mean value for the last iteration.
	The mean value $ \overline{R_{7}} $ is below the threshold value $ 10^{-6} $.
	The $L^2$ norm of the approximation error $\|y(x)-y_{\rm{c}}^{(7)}(x)\|\approx 1.5\times 10^{-7}$.
\end{example}

\begin{example}[Hanging Bar Problem]
	We apply the collocation method on the hanging bar problem~\eqref{eq:IVP_hanging_bar} and $ y(x)=e^x(x-1)^2 $~\cite{riviere08dgm}.
	The approximate solution $ y_{\rm{c}}(x) $ is computed as discussed in Section~\ref{sec:pw_IVP}.
	
	We set the number of Sinc points to be inserted in all partitions as $ m=2N+1=7$.
	The stopping criterion $ \varepsilon_{\mathrm{stop}}=10^{-6} $ was used.
	The algorithm terminates after $\kappa=3 $ iterations and the number of points $|{\mathsf{S}}|=350$.
	
	Figure~\ref{fig:hanging_bar_problem_approx} shows the approximate solution $y_{\rm{c}}^{(3)}(x)$.
	A proper subset of the set of points $ \mathsf{S} $ is shown as red dots, which are projected onto the approximate solution $ y_{\rm{c}}^{(3)}(x) $.

	We perform the least-squares fitting of the logarithm of the set $ \mathsf{R}$ to the logarithm of the upper bound~\eqref{eq:mainthm_gen_coll}.
	Figure~\ref{fig4}a shows the least-squares fitted model~\eqref{eq:mainthm_gen_coll} to the set $ \mathsf{R}$.
	Figure~\ref{fig4}b shows the residual, absolute local approximation error, and the mean value for the last iteration.
	The mean value $ \overline{R_{3}} $ is below the threshold value $ 10^{-6} $.
	The $L^2$ norm of the approximation error $\|y(x)-y_{\rm{c}}^{(3)}(x)\|\approx 5.82\times 10^{-9}$.
\end{example}

\vspace{-7pt}
\begin{figure}[H]
		\centering
		\captionsetup{justification=centering}\subfloat[Visualization of the approximating polynomial $y_{\rm{c}}^{(7)}(x)$.]{\label{e_-ax_approx}\includegraphics[height=0.2\textheight,valign=m]{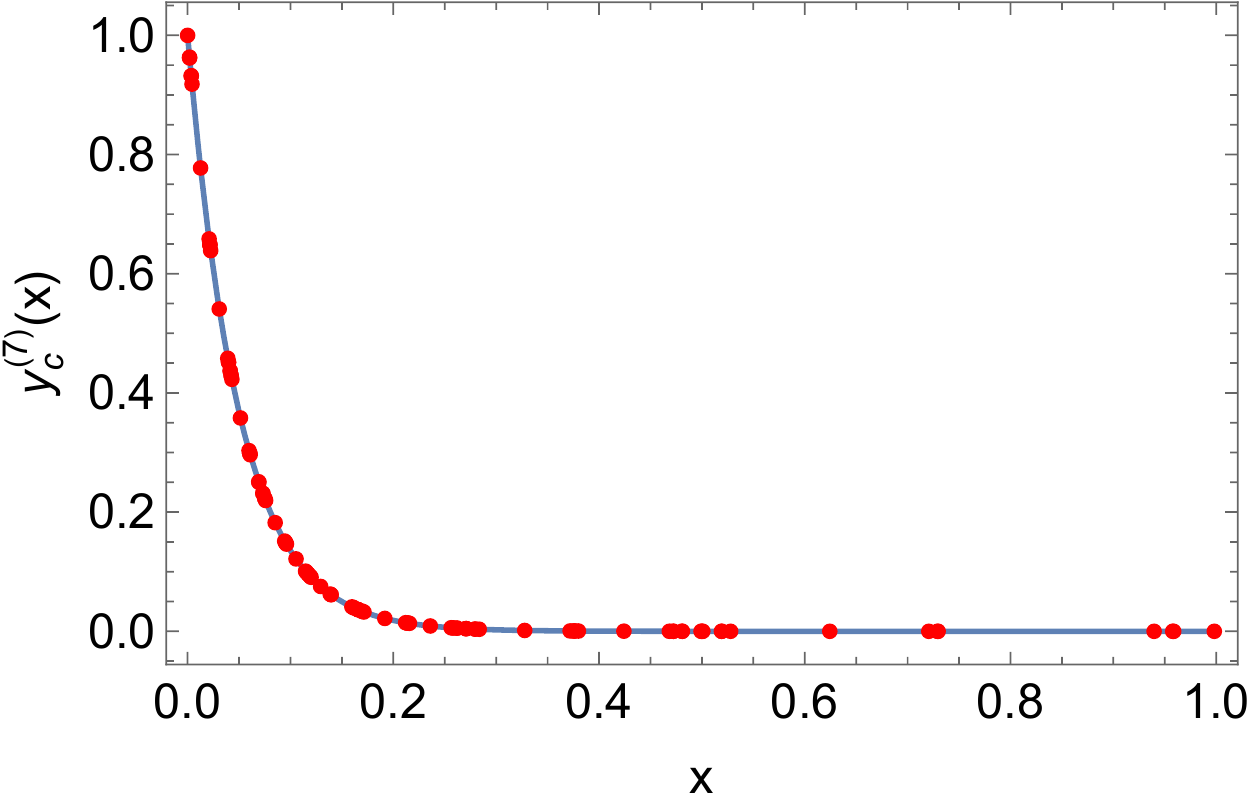} }
		\quad
		\captionsetup{justification=centering}	\subfloat[ Plot of the statistic $\omega_i,\,i,\,i=2,{3,}\ldots,7$. 	(\lineone) $ \overline{\omega_i} \approx 0.6$.]{\label{e_-ax_w_i}\includegraphics[height=0.2\textheight,valign=m]{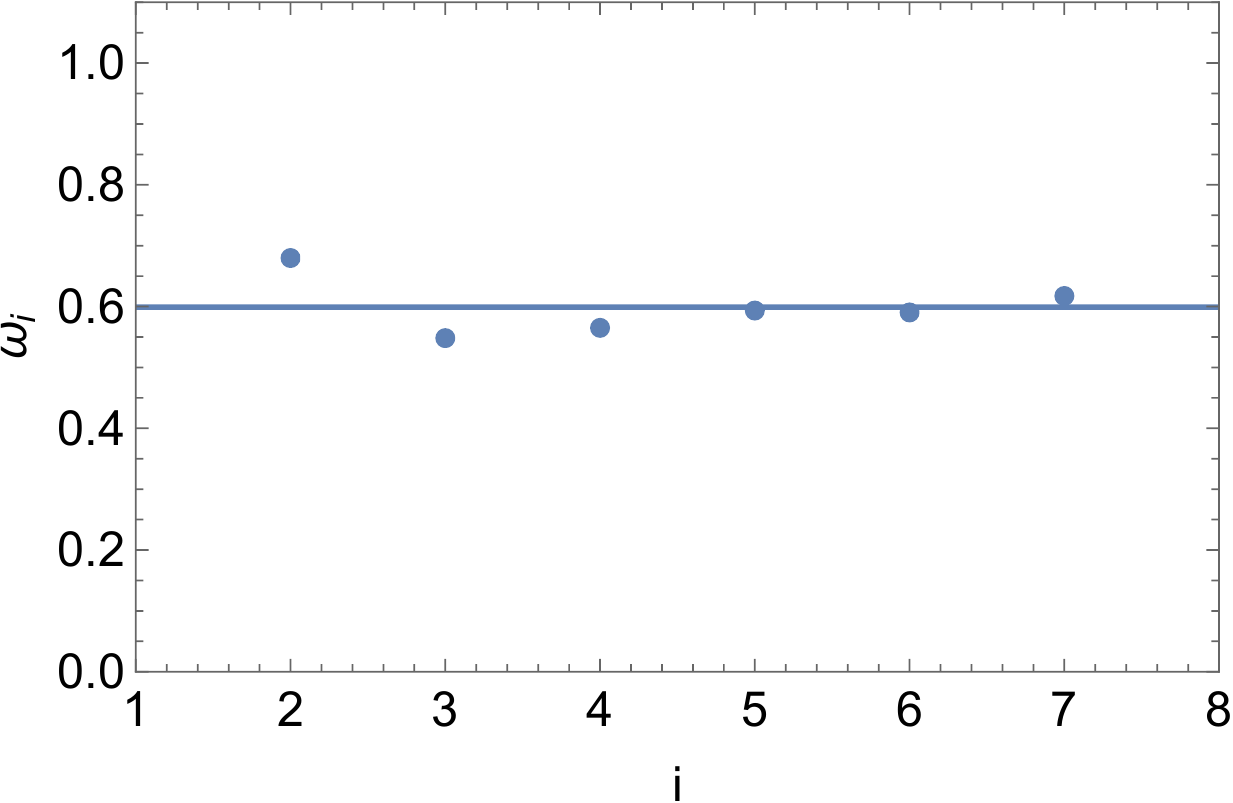} }
	\vspace{6pt}
	\caption{(\textbf{a}) The approximating polynomial $y_{\rm{c}}^{(7)}(x)$. A proper subset of the set of points $ \mathsf{S} $ is shown. (\textbf{b})   Plot of the statistic $\omega_i$.
		(\lineone) $ \overline{\omega_i} \approx 0.6$. \label{fig1}}
\end{figure}

\vspace{-6pt}

\begin{figure}[H]
	
		\centering
		\captionsetup{justification=centering}\subfloat[$A=1.2\times 10^5,\,r=4.36,\,\lambda=0.563,\, \delta=6.06\times 10^{-7}$.]{\label{e_-ax_a}\includegraphics[height=0.2\textheight,valign=m]{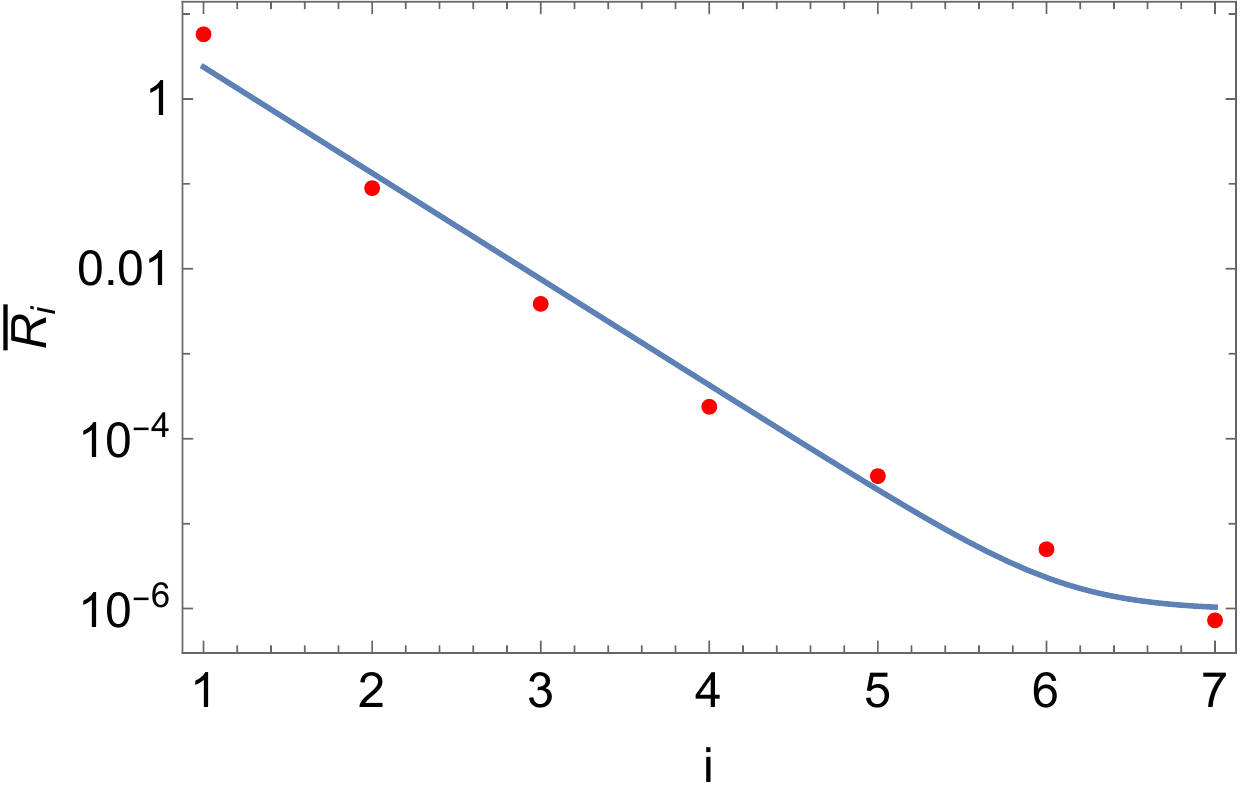} }
		\quad
		\captionsetup{justification=centering}\subfloat[\lineone $\lVert R^{(7)}\rVert_{L^2}$ \linetwo $  |y(x)-y_{\rm{c}}^{(7)}(x)|$\ \linethree $ \overline{R_7} $.]{\label{e_-ax_b}\includegraphics[height=0.2\textheight,valign=m]{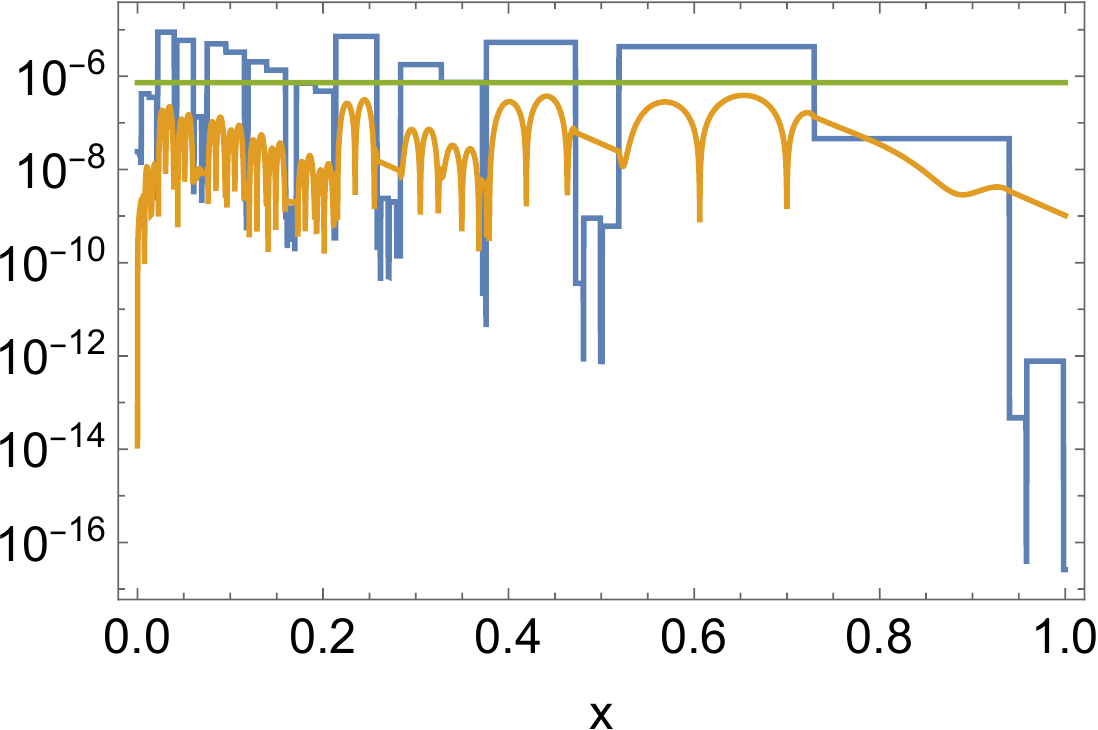} }
		
			\vspace{6pt}
	\caption{ (\textbf{a}) Fitting the upper bound~\eqref{eq:mainthm_gen_coll} with the set $\mathsf{R}$. (\textbf{b}) Visualization of the residual, absolute local approximation error, and the mean value. \label{fig2}}
\end{figure}

\vspace{-6pt}
\begin{figure}[H]
	\centering
	\includegraphics[width=0.5\textwidth,valign=m]{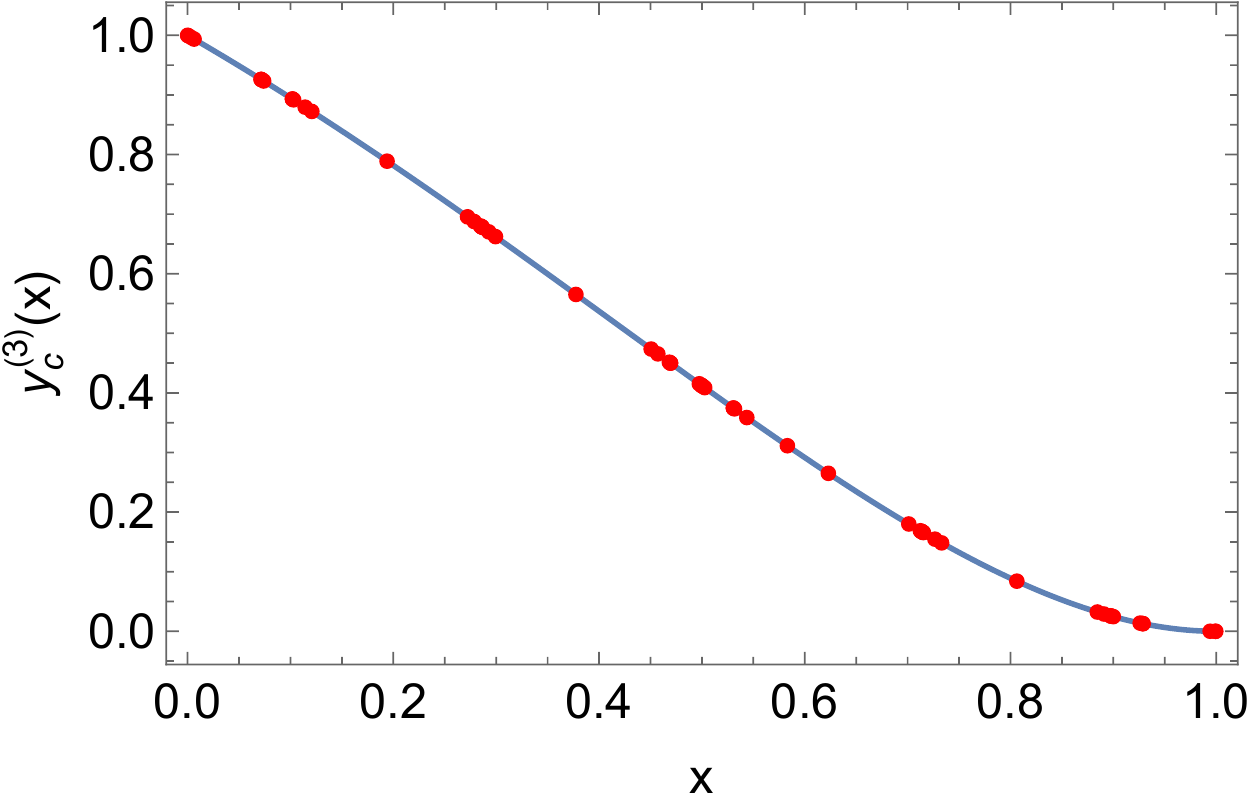}
	\caption{The approximating polynomial $y_{\rm{c}}^{(3)}(x)$. A proper subset of the set of points $ \mathsf{S} $ is shown.}
	\label{fig:hanging_bar_problem_approx}
\end{figure}

\begin{figure}[H]
		\centering
		\subfloat[$A=1.2\times 10^5,\,r=1.7\times 10^1,\,\lambda=0.302,\, \delta=2.22\times 10^{-8}$.]{\label{hanging_bar_a}\includegraphics[height=0.2\textheight,valign=m]{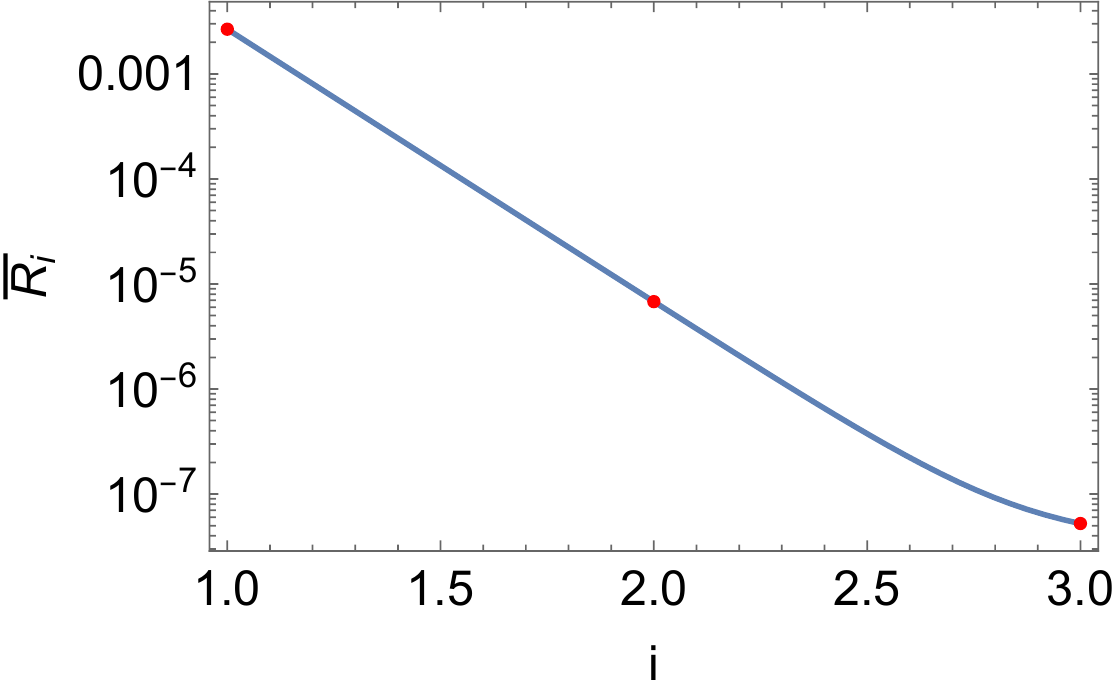} }
		\quad
		\subfloat[\lineone $\lVert R^{(3)}\rVert_{L^2}$ \linetwo $  |y(x)-y_{\rm{c}}^{(3)}(x)|$\ \linethree $ \overline{R_3} $.]{\label{hanging_bar_b}\includegraphics[height=0.2\textheight,valign=m]{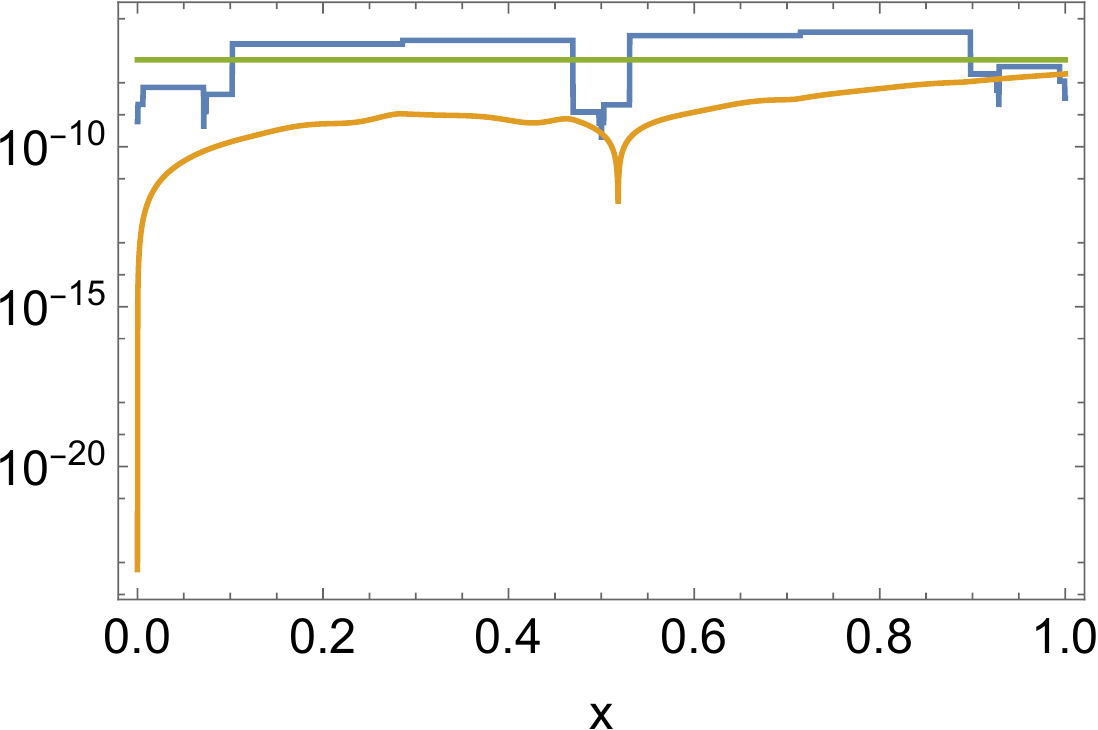} }
		
		\vspace{6pt}
	\caption{(\textbf{a}) Fitting the upper bound~\eqref{eq:mainthm_gen_coll} with the set $\mathsf{R}$. (\textbf{b}) Visualization of the residual, absolute local approximation error and the mean value.}\label{fig4}
\end{figure}

\subsection{Boundary Value Problem}
We discuss a number of stiff BVPs~\cite{Eriksson95_adaptive} based on the general linear second-order BVP
\begin{equation}\label{eq:BVP_ex1-3_ODEs}
	-(a(x)\, y')'+b(x)y'+c(x)\,y=f(x),\quad x\in [0,1],
\end{equation}
where $ a(x)>0,\, b(x),\, c(x) $ are the coefficients, and $ f(x) $ is the source term.

\begin{example}
	We study the BVP~\eqref{eq:BVP_ex1-3_ODEs} with  $ a(x)=x+0.01,\,b(x)=c(x)=0 $, and $ f(x)=1 $.
	The exact solution is
	\begin{equation*}
		y(x)=\mathrm{c}_1{\ln}(1+100x)+\mathrm{c}_2-x,
	\end{equation*}
	where $\mathrm{c}_1=1/{\ln}(101)$ and $\mathrm{c}_2=0$ are obtained from the boundary conditions $y(0)=y(1)=0$.
	The  exact solution experiences a boundary layer near $x=0$.
	
	We set the number of Sinc points to be inserted in all partitions as $m= 2N+1=5$.
	The stopping criterion $ \varepsilon_{\mathrm{stop}}=10^{-6} $ was used.
	The algorithm terminates after $\kappa=10 $ iterations and the number of points $|{\mathsf{S}}|=2055$.
	
	Figure~\ref{fig5}a shows the approximate solution $ y_{\rm{c}}^{(10)}(x) $.
	A proper subset of the set of points $ \mathsf{S} $ is shown as red dots, which are projected onto the approximate solution $ y_{\rm{c}}^{(10)}(x) $.
	We plot the statistic $\omega_i$ as a function of the iteration index $i,\,i=2,{3,}\ldots,10$, in Figure~\ref{fig5}b.
	It is observed that the oscillations are decaying and the statistic $\omega_i$ is converging to an asymptotic value. The mean value $ \overline{\omega_i} \approx 0.64 $.
	
	\begin{figure}[H]
		
			\centering
			
			\subfloat[Visualization of the approximating polynomial $y_{\rm{c}}^{(10)}(x)$.]{\label{ex1_approx}\includegraphics[height=0.2\textheight,valign=m]{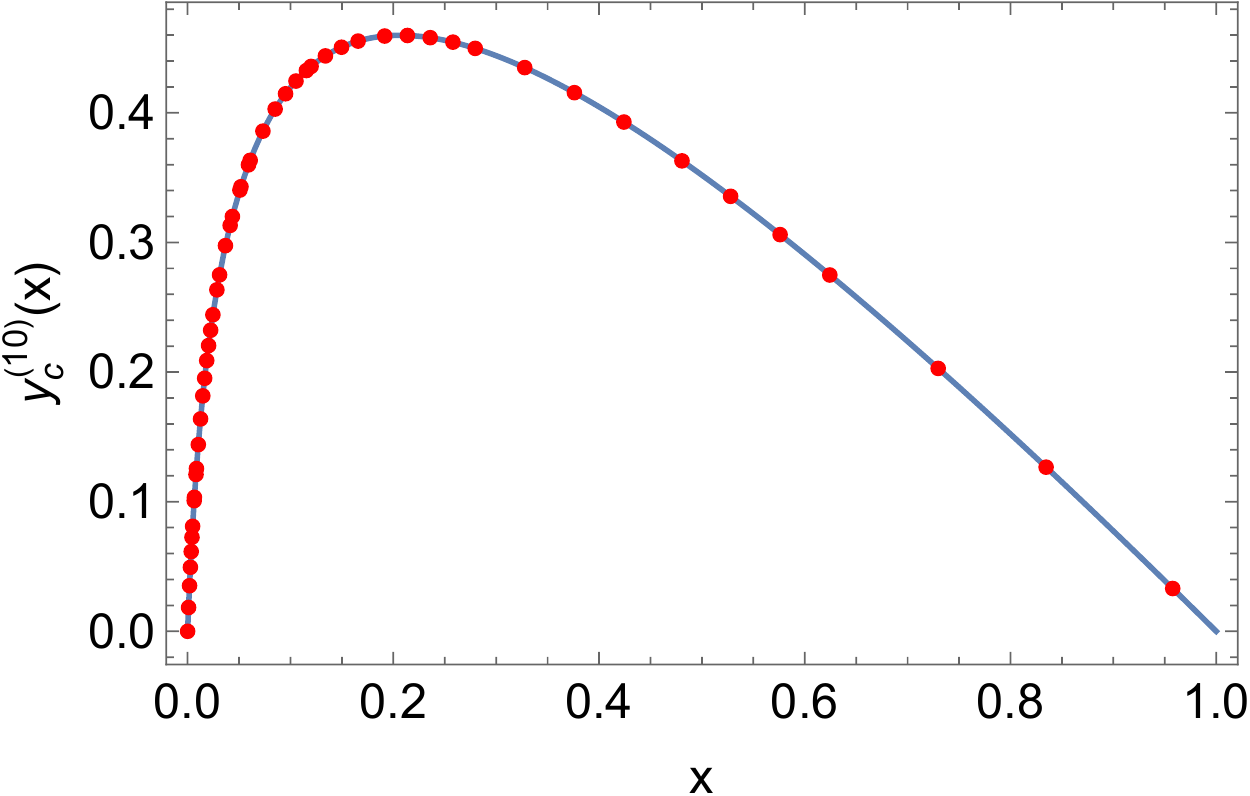} }
			\quad
			\subfloat[Plot of the statistic $\omega_i,\,i=2,{3,}\ldots,10$. 	(\lineone) $ \overline{\omega_i} \approx 0.64$.]{\label{ex1_w_i}\includegraphics[height=0.2\textheight,valign=m]{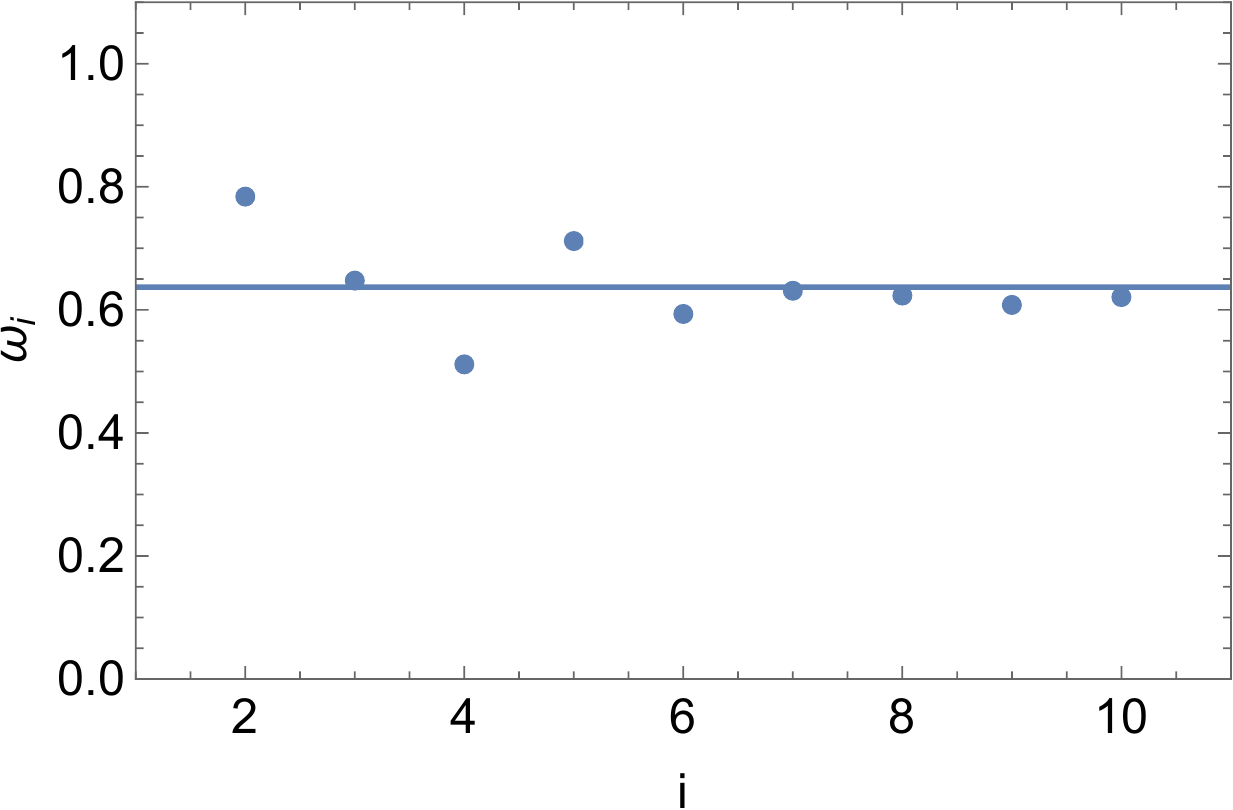} }

			\vspace{6pt}
		\caption{(\textbf{a}) The approximating polynomial $y_{\rm{c}}^{(10)}(x)$. A proper subset of the set of points $ \mathsf{S} $ is shown.  (\textbf{b}) Plot of the statistic $\omega_i$.	(\lineone) $ \overline{\omega_i} \approx 0.64$.\label{fig5}}
	\end{figure}

	We perform the least-squares fitting of the logarithm of the set $ \mathsf{R}$ to the logarithm of the upper bound~\eqref{eq:mainthm_gen_coll}.
	Figure~\ref{fig6}a shows the least-squares fitted model~\eqref{eq:mainthm_gen_coll} to the set $ \mathsf{R}$.
	Figure~\ref{fig6}b shows the residual, absolute local approximation error, and the mean value for the last iteration.
	{The plot of the $L^2$ norm values of the residual over the partitions demonstrates that fine partitions are formed near $x=0$ due to the presence of the boundary layer.}
	The mean value $ \overline{R_{10}} $ is below the threshold value $ 10^{-6} $.
	The $L^2$ norm of the approximation error $\|y(x)-y_{\rm{c}}^{(10)}(x)\|\approx 1.12\times 10^{-8}$.
	The threshold value in~\cite{Eriksson95_adaptive} is $0.05$.
	
	\begin{figure}[H]
		
			\centering
			
			\subfloat[$A=1.2\times 10^5,\,r=4.31,\,\lambda=0.69,\, \delta=4.49\times 10^{-8}$.]{\label{ex1_a_ODE_pw}\includegraphics[height=0.2\textheight,valign=m]{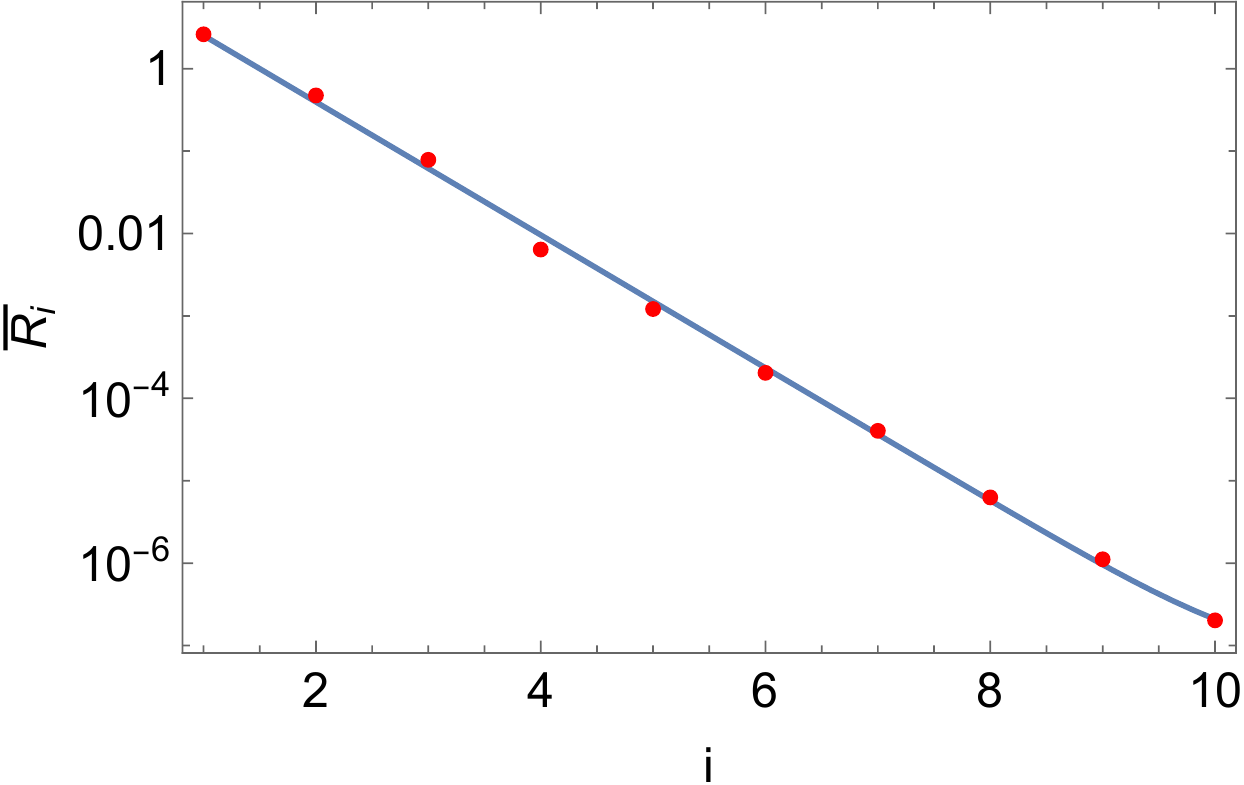} }
			\quad
			\subfloat[\lineone $\lVert R^{(10)}\rVert_{L^2}$ \linetwo $  |y(x)-y_{\rm{c}}^{(10)}(x)|$\ \linethree $ \overline{R_{10}} $.]{\label{ex1_b_ODE_pw}\includegraphics[height=0.2\textheight,valign=m]{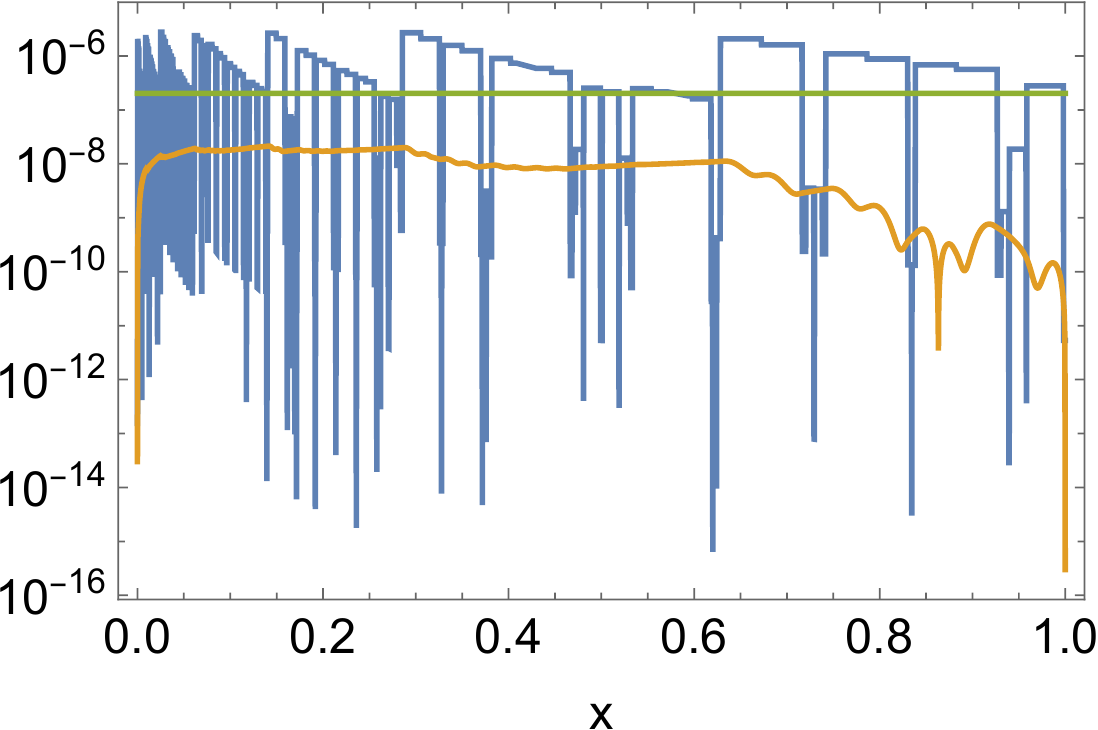} }

			\vspace{6pt}
		\caption{(\textbf{a})  Fitting the upper bound~\eqref{eq:mainthm_gen_coll} with the set $\mathsf{R}$. (\textbf{b})  Visualization of the residual, absolute local approximation error, and the mean value.\label{fig6}}
	\end{figure}

\end{example}

\begin{example}\label{ex:Ex2_EEHJ95}
	We study the BVP~\eqref{eq:BVP_ex1-3_ODEs} with  $ a(x)=0.01,\,b(x)=0,\,c(x)=1 $, and $ f(x)=1/x $.
	Using the variation of parameters method~\cite{BoyceDiPrima12}, the exact solution of this problem is \vspace{-10pt}
	
	\[		y(x)=-5\,\mathrm{Ei}(-10x)\exp(10x)+5\,\mathrm{Ei}(10x)\exp(-10x)+\mathrm{c}_1\exp(10x)+\mathrm{c}_2\exp(-10x),
	\]
	where $ \mathrm{Ei}(x)=\int_{-\infty}^x \frac{\exp(t)}{t}\dd{t} $ is the exponential integral function~\cite{Abramowitz_Stegun72} and
	\[
	\mathrm{c}_1=-\mathrm{c}_2=\dfrac{-5\,\mathrm{Ei}(-10)\exp(10)+5\,\mathrm{Ei}(10)\exp(-10)}{\exp(-10)-\exp(10)}.
	\]
	The  exact solution $y(x)$ experiences a  boundary layer near $x=0$ and a slope change at approximately $ x=0.5 $.
	Equation~\eqref{eq:BVP_ex1-3_ODEs} is multiplied by the factor $ x $ so that the residual $ R(x) $ does not contain a singularity at $ x=0 $.
	
	We set the number of Sinc points to be inserted in all partitions as $ m=2N+1=5$.
	The stopping criterion $ \varepsilon_{\mathrm{stop}}=10^{-6} $ was used.
	The algorithm terminates after $\kappa=9 $ iterations and the number of points $|{\mathsf{S}}|=1630$.
	
	We perform the  least-squares fitting of the logarithm of the set $ \mathsf{R}$ to the logarithm of the upper bound~\eqref{eq:mainthm_gen_coll}.
	Figure~\ref{fig7}a shows the least-squares fitted model~\eqref{eq:mainthm_gen_coll} to the set $ \mathsf{R}$.
	Figure~\ref{fig7}b shows the residual, absolute local approximation error, and the mean value for the last iteration.
	{The plot of the $L^2$ norm values of the residual over the partitions shows that fine partitions are formed near $x=0$ due to the presence of the boundary layer.}
	The mean value $ \overline{R_{9}} $ is below the threshold value $ 10^{-6} $.
	The $L^2$ norm of the approximation error $\|y(x)-y_{\rm{c}}^{(9)}(x)\|\approx 1.6\times 10^{-6}$.
	The threshold value in~\cite{Eriksson95_adaptive} is $0.01$.
	
	The approximating polynomial $ y_{\rm{c}}^{(9)}(x) $ and a proper subset of the set of points $ \mathsf{S} $ are shown in Figure~\ref{fig8}a.
	The corresponding plot for the statistic $ \omega_i $ a is shown in Figure~\ref{fig8}c. The oscillations are decaying and the mean value $ \overline{\omega_i}\approx 0.66  $.
	It was mentioned that Equation~\eqref{eq:BVP_ex1-3_ODEs} was multiplied by the factor $ x $ so that the residual $ R(x) $ does not contain a singularity at $ x=0 $.
	We replace the residual $ R(x) $ with the quantity $ y(x)-y_{\rm{c}}(x) $ in Algorithm~\ref{alg:adaptive_polySinc_alg_ODE} and the BVP~\eqref{eq:BVP_ex1-3_ODEs} contains the term $1/x$.
	We set the number of Sinc points to be inserted in all partitions as $ m=2N+1=5$.
	The stopping criterion $ \varepsilon_{\mathrm{stop}}=10^{-6} $ was used.
	The algorithm terminates after $\kappa=8 $ iterations and the number of points $|{\mathsf{S}}|=730$, which is smaller than the one obtained from multiplying the residual by $ x $.
	This is expected since the exact solution $ y(x) $ is used.
	The approximating polynomial $ y_{\rm{c}}^{(8)}(x) $ and a proper subset of the set of points $ \mathsf{S} $ are shown in Figure~\ref{fig8}b. The corresponding plot for the statistic $ \omega_i $ is shown in Figure~\ref{fig8}d.
	The mean value $ \overline{\omega_i}\approx 0.61  $.
	It is observed that the  statistic $ \omega_i$ oscillates around the mean  $ \overline{\omega_i}\approx 0.61  $.
	
	\begin{figure}[H]
		
			\centering

			\subfloat[$A=1.2\times 10^5,\,r=5.03,\,\lambda=0.692,\, \delta=2.21\times 10^{-7}$.]{\label{ex2_model}\includegraphics[height=0.2\textheight,valign=m]{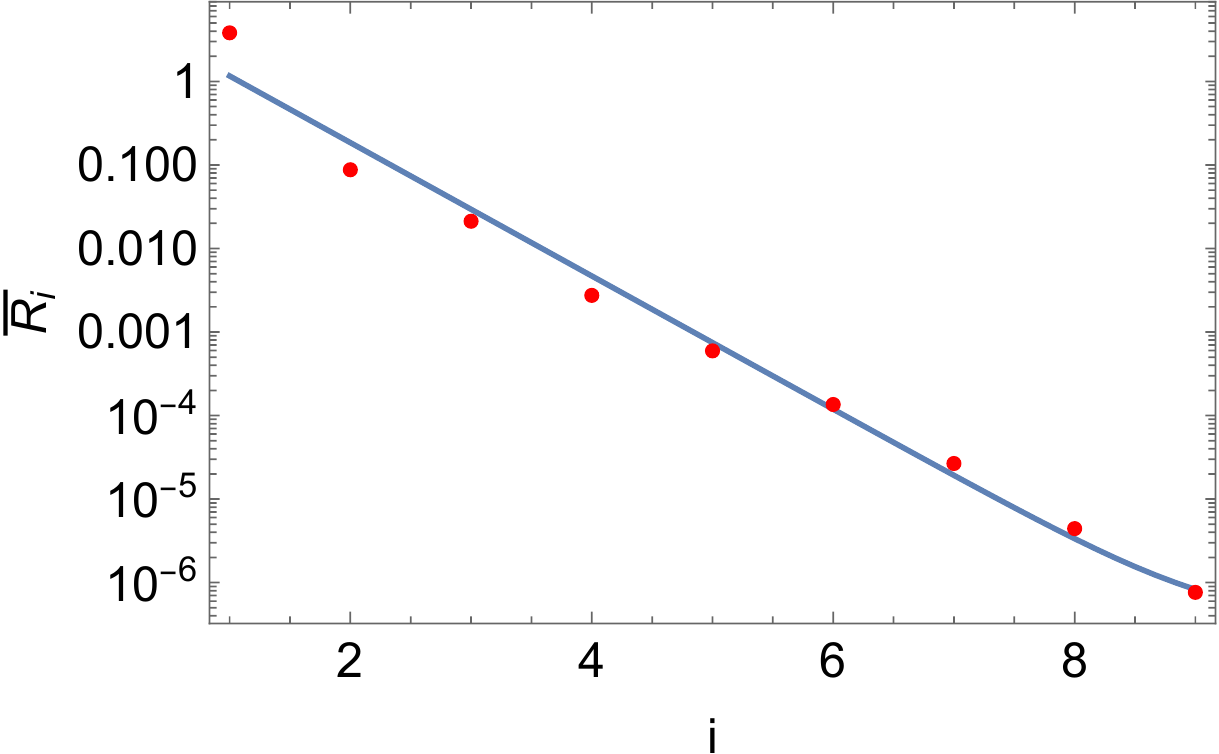} }
			\quad
			\subfloat[\lineone $\lVert R^{(9)}\rVert_{L^2}$ \linetwo $  |y(x)-y_{\rm{c}}^{(9)}(x)|$\ \linethree $ \overline{R_9} $.]{\label{ex2_error}\includegraphics[height=0.2\textheight,valign=m]{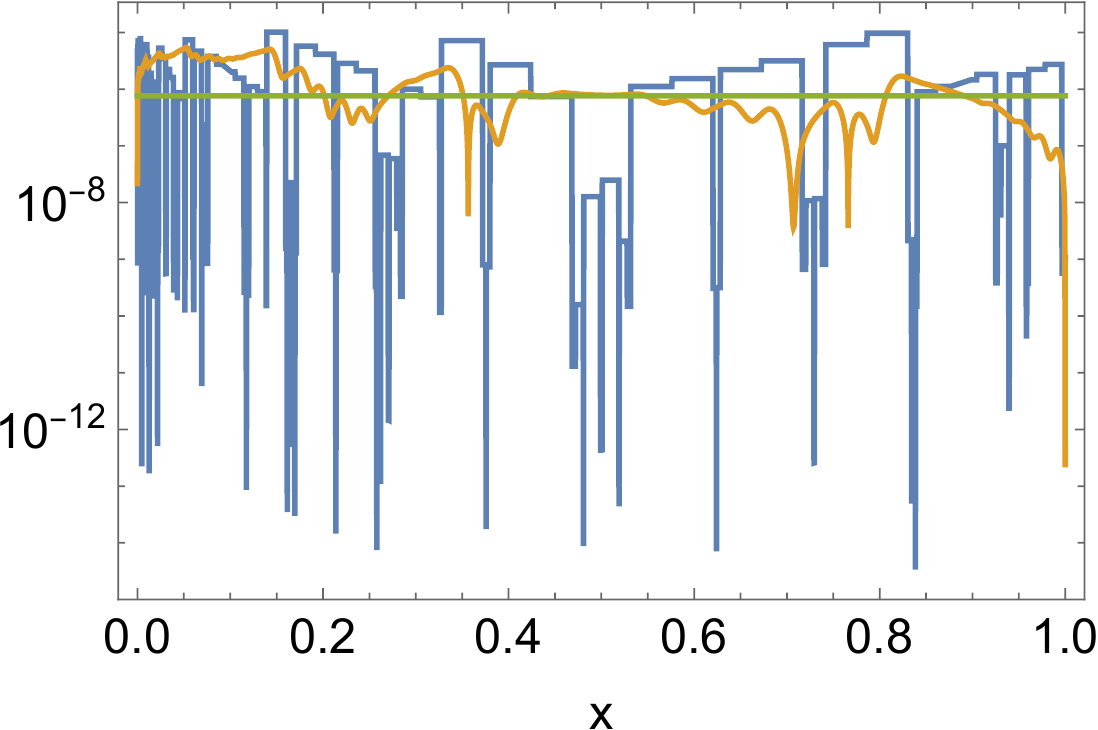} }
		\vspace{6pt}
		\caption{(\textbf{a}) Fitting the upper bound~\eqref{eq:mainthm_gen_coll} with the set $\mathsf{R}$. (\textbf{b})  Visualization of the residual, absolute local approximation error, and the mean value.\label{fig7}}
	\end{figure}

	\begin{figure}[H]
		\centering			
			\captionsetup{justification=centering}\subfloat[]{\label{ex2_approx}\includegraphics[height=0.17\textheight,valign=m]{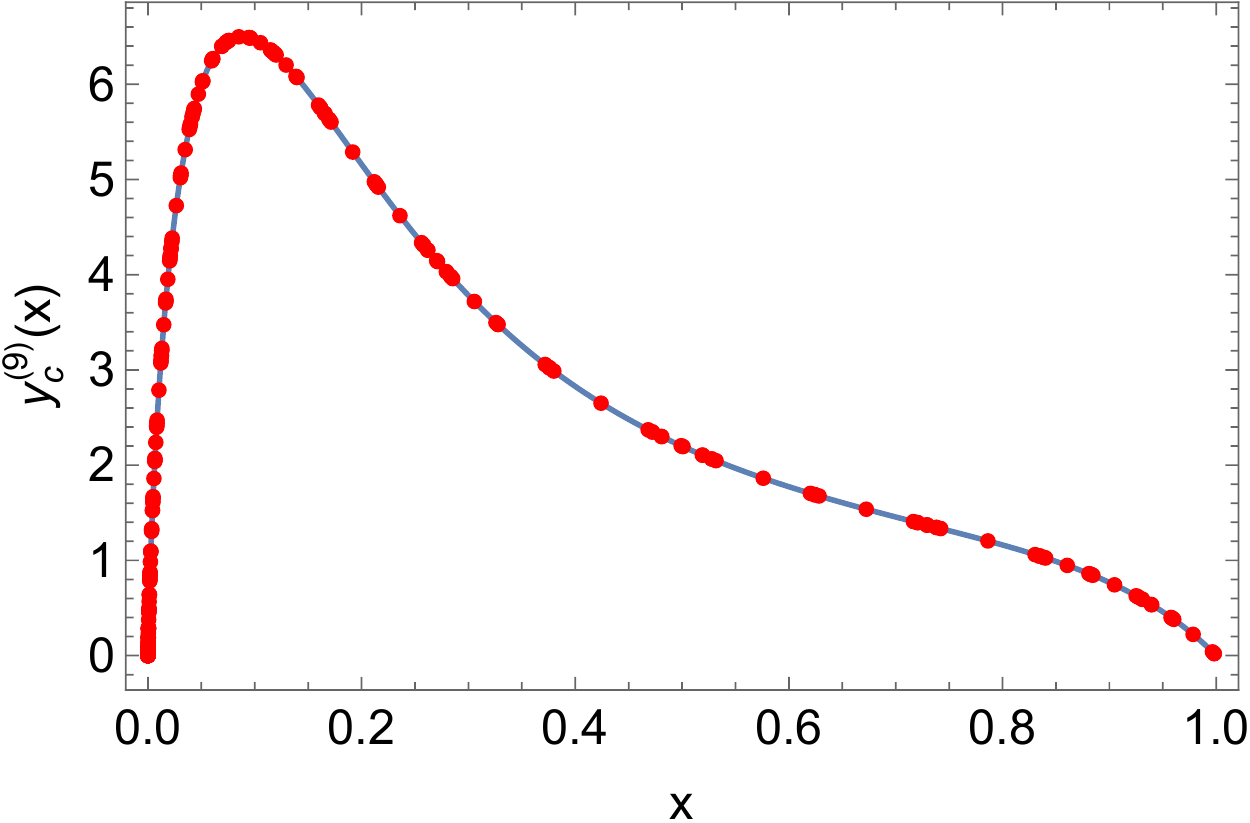} }
			\quad
			\captionsetup{justification=centering}\subfloat[]{\label{ex2_approx_u_uh}\includegraphics[height=0.17\textheight,valign=m]{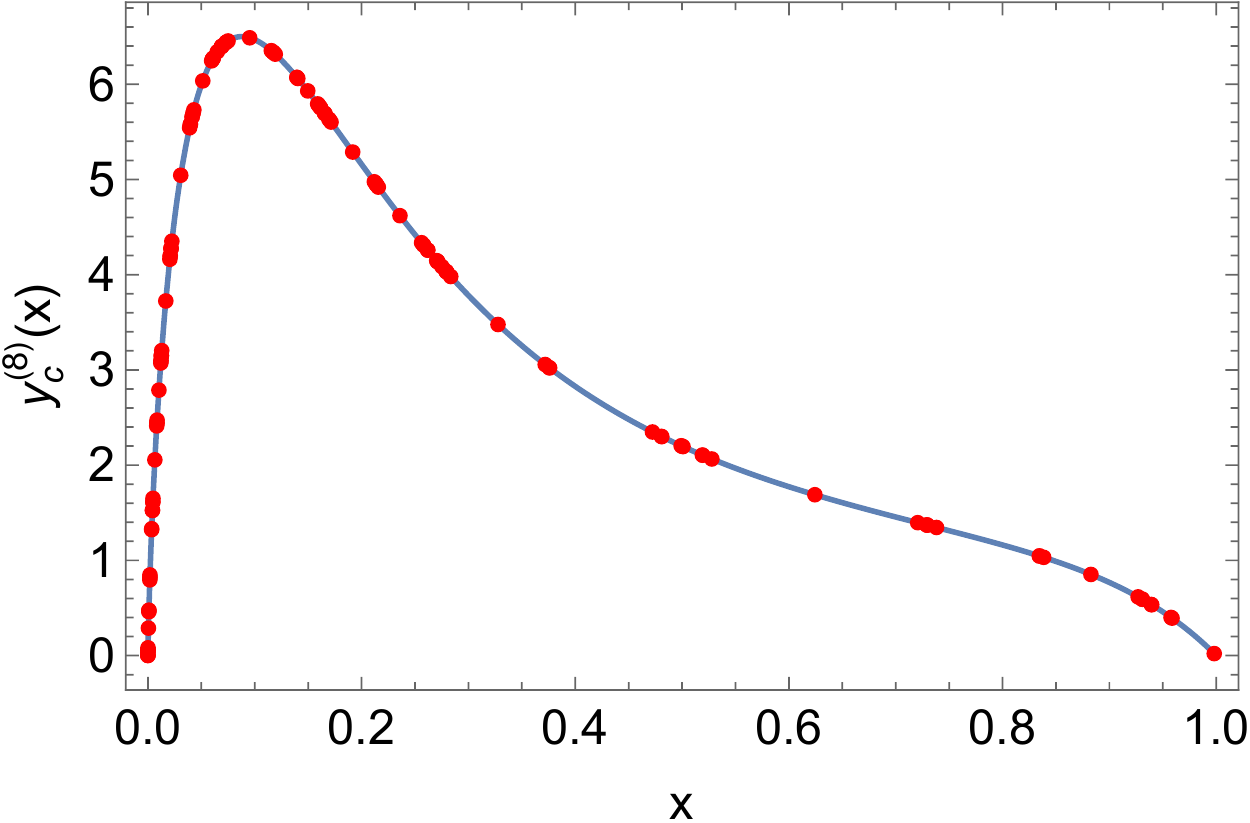} }\\[2em]
			\captionsetup{justification=centering}\subfloat[ ]{\label{ex2_w_i}\includegraphics[height=0.17\textheight,valign=m]{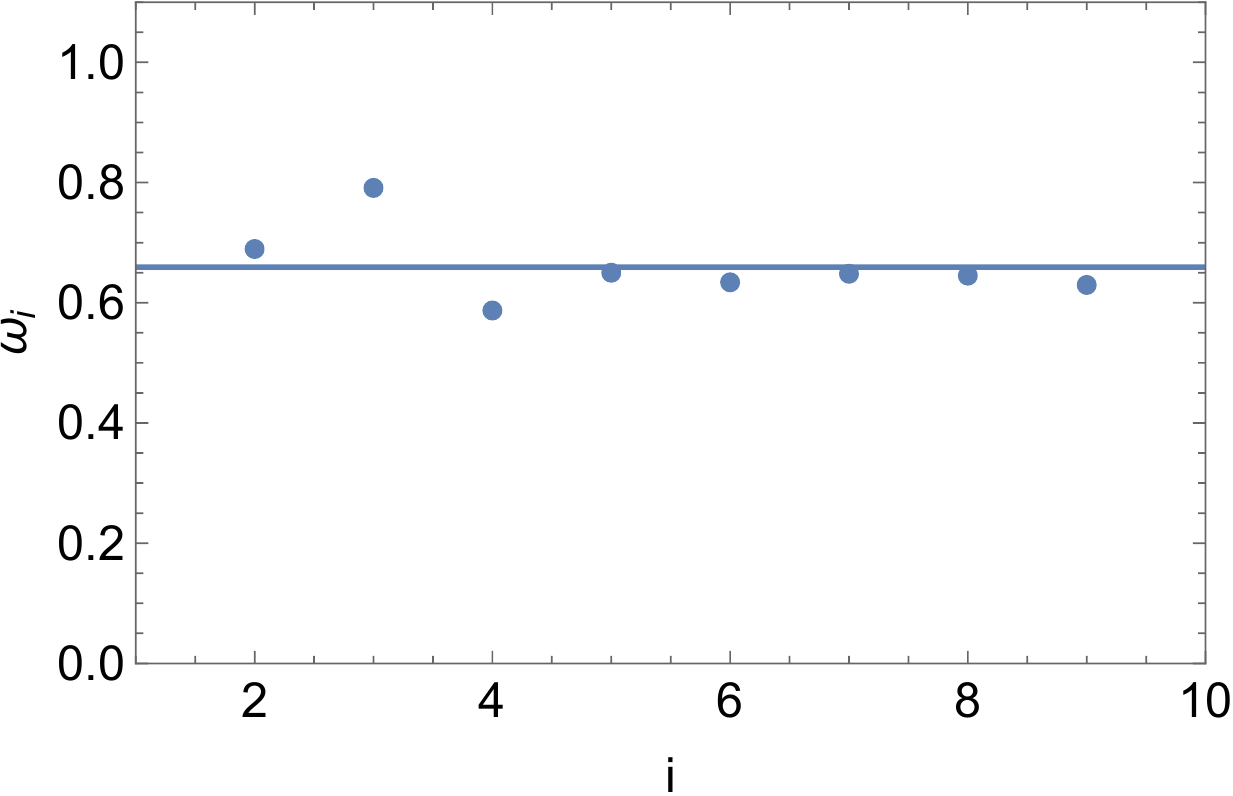} }
			\quad
			\captionsetup{justification=centering}\subfloat[ ]{\label{ex2_w_i_u_uh}\includegraphics[height=0.17\textheight,valign=m]{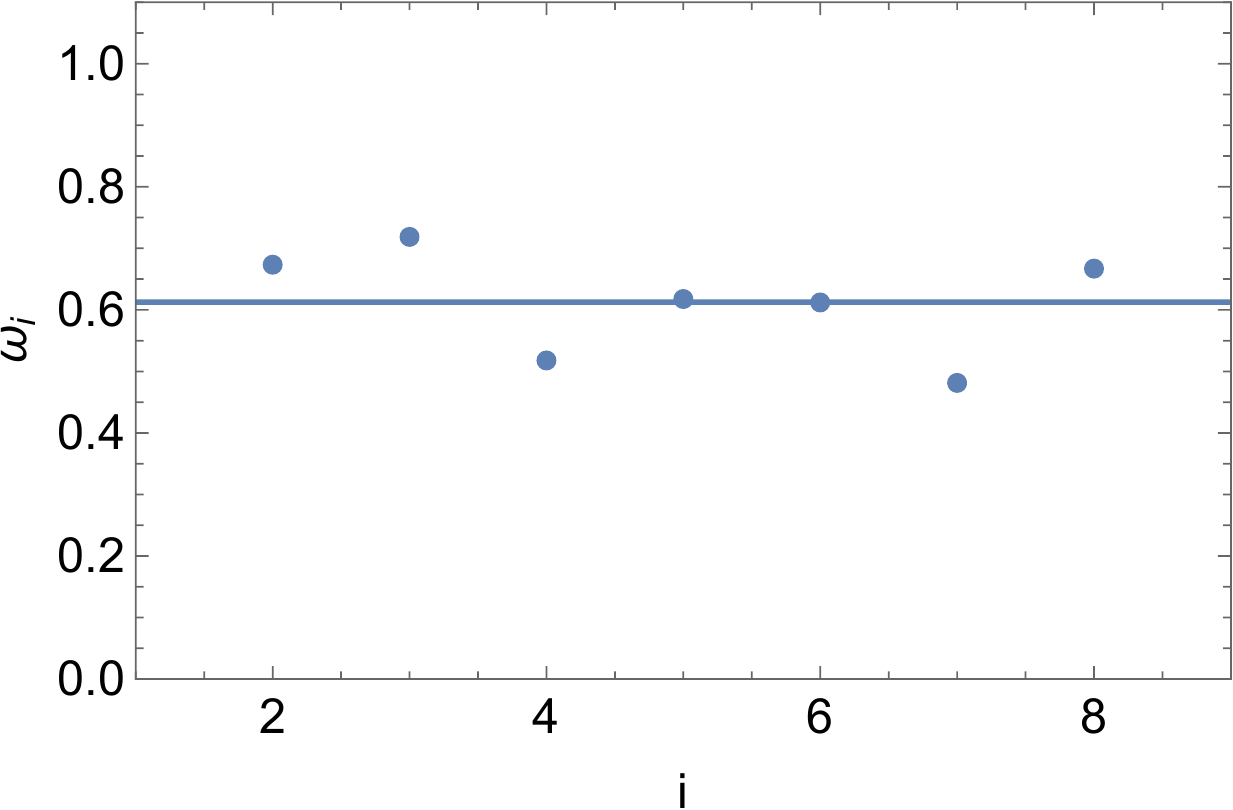} }
		
	\caption{Visualization of the approximating polynomial and statistic $ \omega_i $. (\textbf{a}) {Approximating} polynomial $y_{\rm{c}}^{(9)}(x)$ obtained by multiplying Equation~\eqref{eq:BVP_ex1-3_ODEs} by $ x $. (\textbf{b}) 
		Approximating polynomial $y_{\rm{c}}^{(8)}(x)$ obtained by $  y(x)-y_{\rm{c}}(x)$. (\textbf{c}) Plot of $ \omega_i,\,i=2,{3,}\ldots,9 $.\  (\lineone) $  \overline{\omega_i}\approx 0.66$. (\textbf{d}) Plot of $\omega_i,\,i=2,{3,}\ldots,8 $.\  (\lineone) $ \overline{\omega_i}\approx 0.61 $.}\label{fig8}
	\end{figure}
	
\end{example}

\begin{example}
	We study a variation of Example~\ref{ex:Ex2_EEHJ95}, in which the source term $ f(x)=\frac{1}{\sqrt{x}} $.
	Even though the source term has a singularity at $ x=0 $, its definite integral over the domain $ [0,1] $ is finite.
	The exact solution~\cite{BoyceDiPrima12} is
	\vspace{-16pt}
	
	\begin{multline*}
		y(x)={\rm{c}_1}\exp (-10 x)+{\rm{c}_2}\exp (10 x) +\exp (10 x)\left(-5 \left(\sqrt{\frac{\pi }{10}} \text{erf}\left(\sqrt{10 x}\right)\right)\right) \\-\exp (-10 x)\left(-5 \left(\sqrt{\frac{\pi }{10}} \text{erfi}\left(\sqrt{10 x}\right)\right)\right),
	\end{multline*}
	where
	\[
	{\rm{c}}_1=-{\rm{c}}_2=-\frac{\sqrt{\frac{5 \pi }{2}} \left(e^{20} \text{erf}\left(\sqrt{10}\right)-\text{erfi}\left(\sqrt{10}\right)\right)}{e^{20}-1},
	\]
	$ \mathrm{\erf}(x)=\frac{2}{\sqrt{\pi}}\int_0^x e^{-t^2}\dd{t} $~(\cite{Abramowitz_Stegun72} \S\,7.1.1) is the error function, $\mathrm{erfi}(x)\equiv-\imath \erf(\imath \, x)  =\frac{2}{\sqrt{\pi}}\int_0^x e^{t^2}\dd{t}$~\cite{ng1969table}{,} and $ \imath^2=-1 $.
	The solution $ y(x) $ has a boundary layer near $ x=0 $. Equation~\eqref{eq:BVP_ex1-3_ODEs} is  multiplied by the factor $ \sqrt{x} $ so that the residual $ R(x) $ does not contain a singularity at $ x=0 $.

	We set the number of Sinc points to be inserted in all partitions as $ m=2N+1=5$.
	The stopping criterion $ \varepsilon_{\mathrm{stop}}=10^{-6} $ was used.
	The algorithm terminates after $\kappa=7 $ iterations and the number of points $|{\mathsf{S}}|=1183$.
	
	We performed least-squares fitting of the logarithm of the set $ \mathsf{R}$ to the logarithm of the upper bound~\eqref{eq:mainthm_gen_coll}.
	Figure~\ref{fig9}a shows the least-squares fitted model~\eqref{eq:mainthm_gen_coll} to the set $ \mathsf{R}$.
	Figure~\ref{fig9}b shows the residual, absolute local approximation error, and the mean value for the last iteration.
	{Fine partitions are formed near $x=0$ due to the presence of the boundary layer, as seen in the plot of the $L^2$ norm values of the residual over the partitions.}
	The mean value $ \overline{R_{7}} $ is below the threshold value $ 10^{-6} $.
	The $L^2$ norm of the approximation error $\|y(x)-y_{\rm{c}}^{(7)}(x)\|\approx 2.18\times 10^{-7}$.

	\begin{figure}[H]
		
			\centering

			\subfloat[$A=1.14\times 10^5,\,r=2.81,\,\lambda=0.67,\, \delta=5.02\times 10^{-8}$.]{\label{ex2_model_sqrtx}\includegraphics[height=0.2\textheight,valign=m]{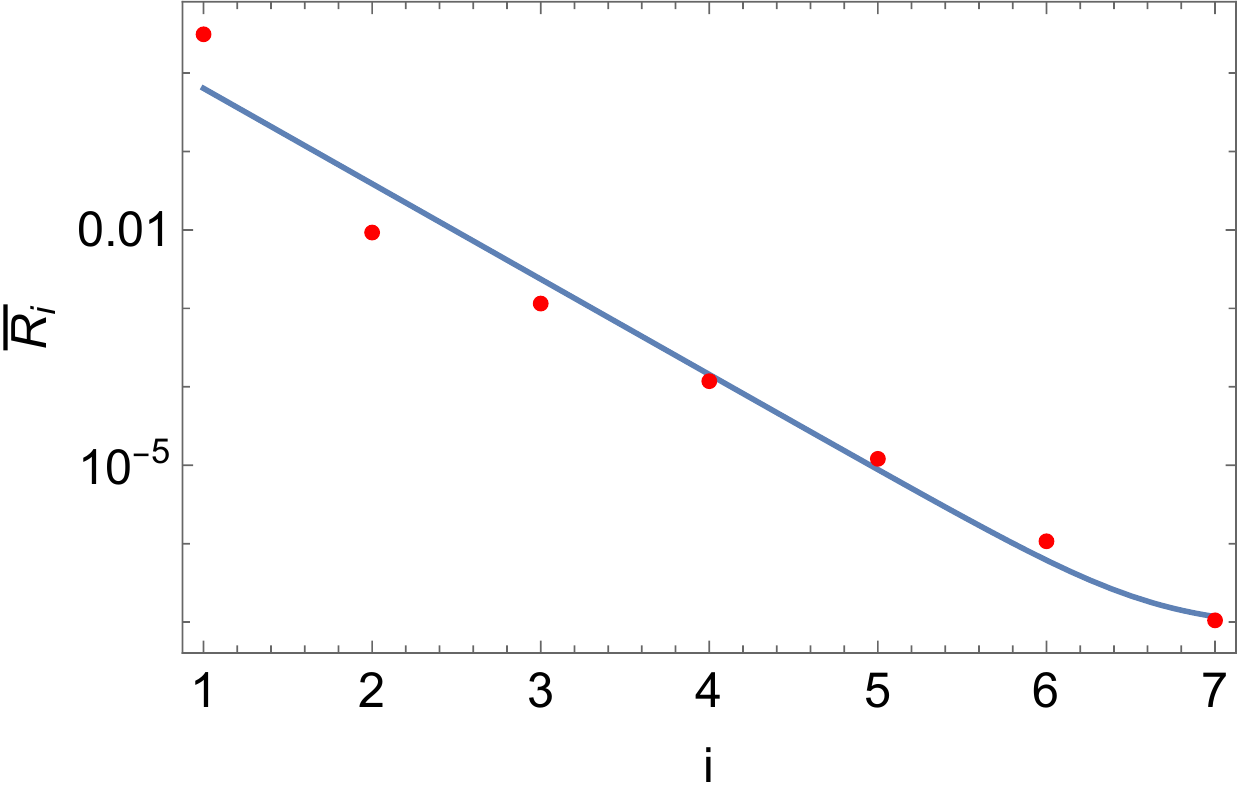} }
			\quad
			\subfloat[\lineone $\lVert R^{(7)}\rVert_{L^2}$ \linetwo $  |y(x)-y_{\rm{c}}^{(7)}(x)|$\ \linethree $ \overline{R_7} $.]{\label{ex2_error_sqrtx}\includegraphics[height=0.2\textheight,valign=m]{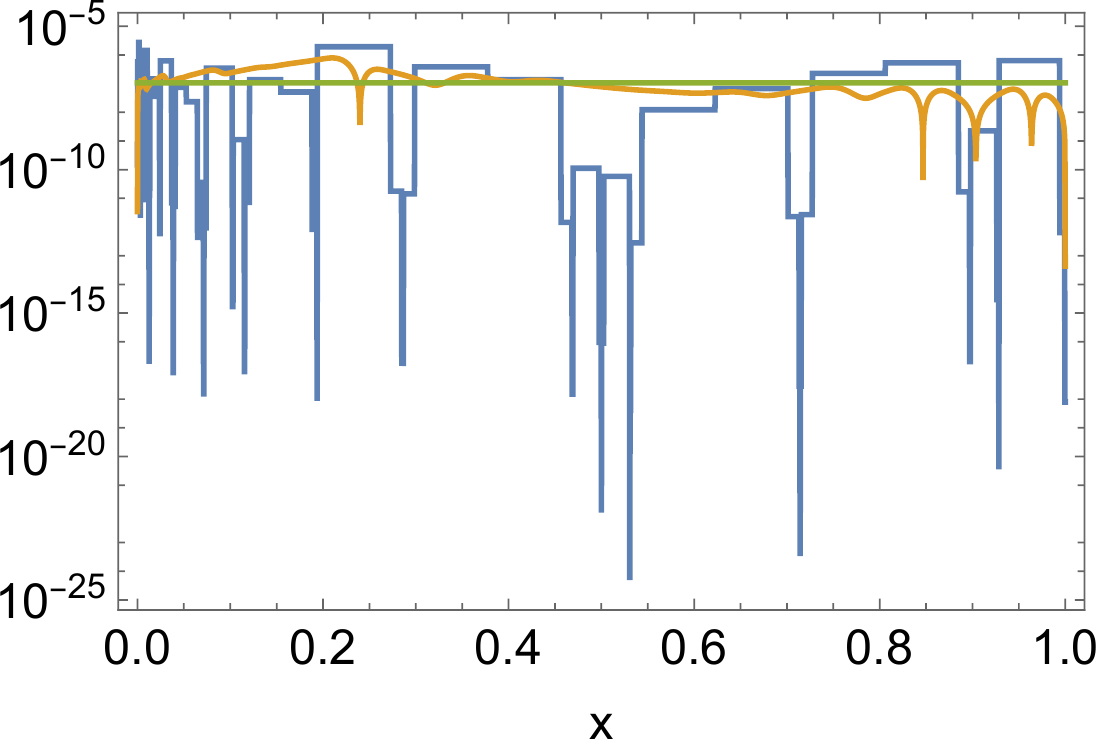} }
		\vspace{6pt}
		
		\caption{ (\textbf{a}) Fitting the upper bound~\eqref{eq:mainthm_gen_coll} with the set $\mathsf{R}$. (\textbf{b})  Visualization of the residual, absolute local approximation error, and the mean value.\label{fig9}}
	\end{figure}
	
	The approximating polynomial ${y_{\rm{c}}^{(7)}(x)} $ and a proper subset of the set of points $ \mathsf{S} $ are shown in Figure~\ref{fig10}a. The corresponding plot for the statistic $ \omega_i $ a is shown in Figure~\ref{fig10}c.
	The statistic $ \omega_i $ oscillates around the median value ${\tilde{\omega}_i}\approx 0.49 $.
	It was mentioned that Equation~\eqref{eq:BVP_ex1-3_ODEs} was multiplied by the factor $ \sqrt{x} $ so that the residual $ R(x) $ does not contain a singularity at $ x=0 $.
	We replace the residual $ R(x) $ with the quantity $ y(x)-y_{\rm{c}}(x) $ in Algorithm~\ref{alg:adaptive_polySinc_alg_ODE} and the BVP~\eqref{eq:BVP_ex1-3_ODEs} contains the term $1/\sqrt{x}$.
	We set the number of Sinc points to be inserted in all partitions as $ m=2N+1=5$.
	The stopping criterion $ \varepsilon_{\mathrm{stop}}=10^{-6} $ was used.
	The algorithm terminates after $\kappa=6 $ iterations and the number of points $|{\mathsf{S}}|=595$, which is smaller than the one obtained by multiplying the residual by $ \sqrt{x} $.
	This is expected since the exact solution $ y(x) $ is used.
	The approximating polynomial $ y_{\rm{c}}^{(6)}(x) $ and a proper subset of the set of points $ \mathsf{S} $ are shown in Figure~\ref{fig10}b. The corresponding plot for the statistic $ \omega_i $ is shown in Figure~\ref{fig10}d.
	The oscillations are decaying and the statistic $\omega_i$ is converging to an asymptotic value. The mean value $ \overline{\omega_i} $ is $ 0.58 $.
	
\end{example}

\begin{example}
	We study a variation of Example~\ref{ex:Ex2_EEHJ95}, in which the source term $ f(x)=\frac{\exp(x)-1}{x} $.
	The source term has a removable singularity at $ x=0 $, since $ \lim\limits_{x\to 0} f(x)=1 $.
	The algorithm can directly solve the BVP, even though the residual contains the term $ \frac{\exp(x)-1}{x} $.
	The exact solution~\cite{BoyceDiPrima12} is
	\begin{multline*}
		y(x)={\rm{c}_1}\exp (-10 x) +{\rm{c}_2}\exp (10 x)-5 \exp (10 x) (\text{Ei}(-9 x)-\text{Ei}(-10 x))\\+5 \exp (-10 x) (\text{Ei}(11 x)-\text{Ei}(10 x)),
	\end{multline*}
	where
	\[
	{\rm{c}}_1=\frac{5 \left(e^{20} \text{Ei}(-10)-e^{20} \text{Ei}(-9)-\text{Ei}(10)+\text{Ei}(11)-e^{20} {\ln} \left(\frac{10}{9}\right)-e^{20} {\ln} \left(\frac{11}{10}\right)\right)}{e^{20}-1}
	\]
	and \vspace{-10pt}

	\[
	{\rm{c}_2}=\frac{5 \left(-e^{20} \text{Ei}(-10)+e^{20} \text{Ei}(-9)+\text{Ei}(10)-\text{Ei}(11)+{\ln} \left(\frac{10}{9}\right)+{\ln} \left(\frac{11}{10}\right)\right)}{e^{20}-1}.
	\]
	
	We set the number of Sinc points to be inserted in all partitions as $ m=2N+1=5$.
	The stopping criterion $ \varepsilon_{\mathrm{stop}}=10^{-6} $ was used.
	The algorithm terminates after $\kappa=8 $ iterations and the number of points $|{\mathsf{S}}|=605$.
	
	Figure~\ref{fig11}a shows the approximate solution $ y_{\rm{c}}^{(8)}(x) $.
	A proper subset of the set of points $ \mathsf{S} $ is shown as red dots, which are projected onto the approximate solution $ y_{\rm{c}}^{(8)}(x) $.
	We plot the statistic $\omega_i$ as a function of the iteration index 
	$i,\,i=2,{3,}\ldots,8$, in Figure~\ref{fig11}b.
	The oscillations are decaying and the statistic $\omega_i$ is converging to an asymptotic value. The mean value $ \overline{\omega_i} \approx 0.65 $.

	\begin{figure}[H]
	\centering
			\captionsetup{justification=centering}\subfloat[]{\label{ex2_approx_sqrtx}\includegraphics[height=0.17\textheight,valign=m]{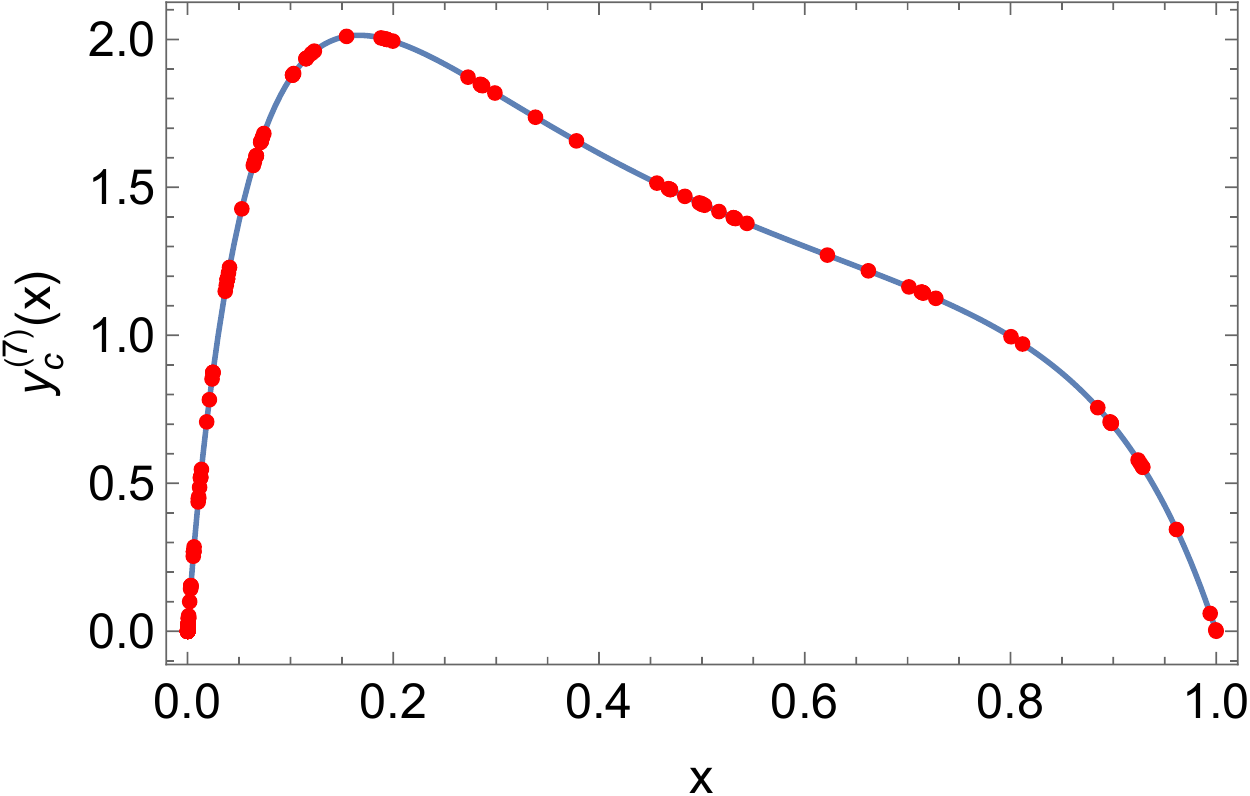} }
			\quad
			\captionsetup{justification=centering}\subfloat[]{\label{ex2_approx_u_uh_sqrtx}\includegraphics[height=0.17\textheight,valign=m]{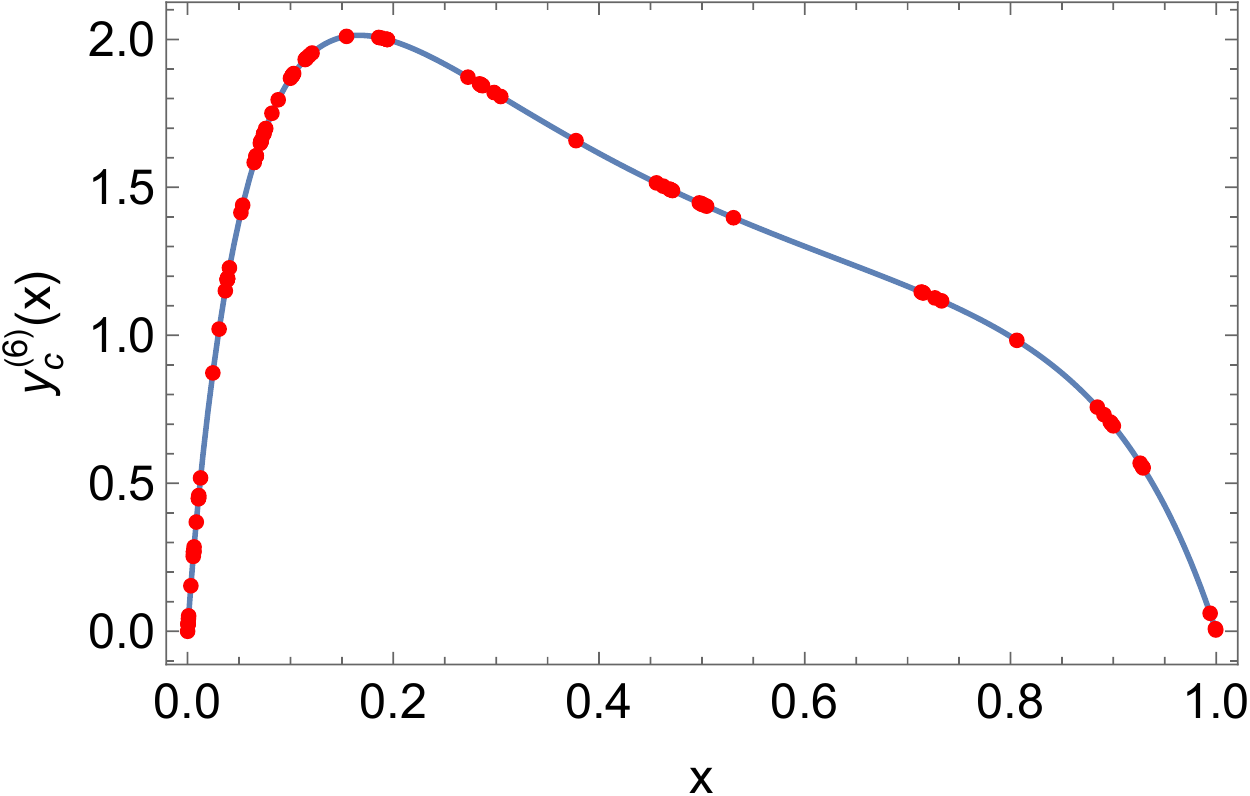} }\\[2em]
			\captionsetup{justification=centering}\subfloat[  ]{\label{ex2_w_i_sqrtx}\includegraphics[height=0.17\textheight,valign=m]{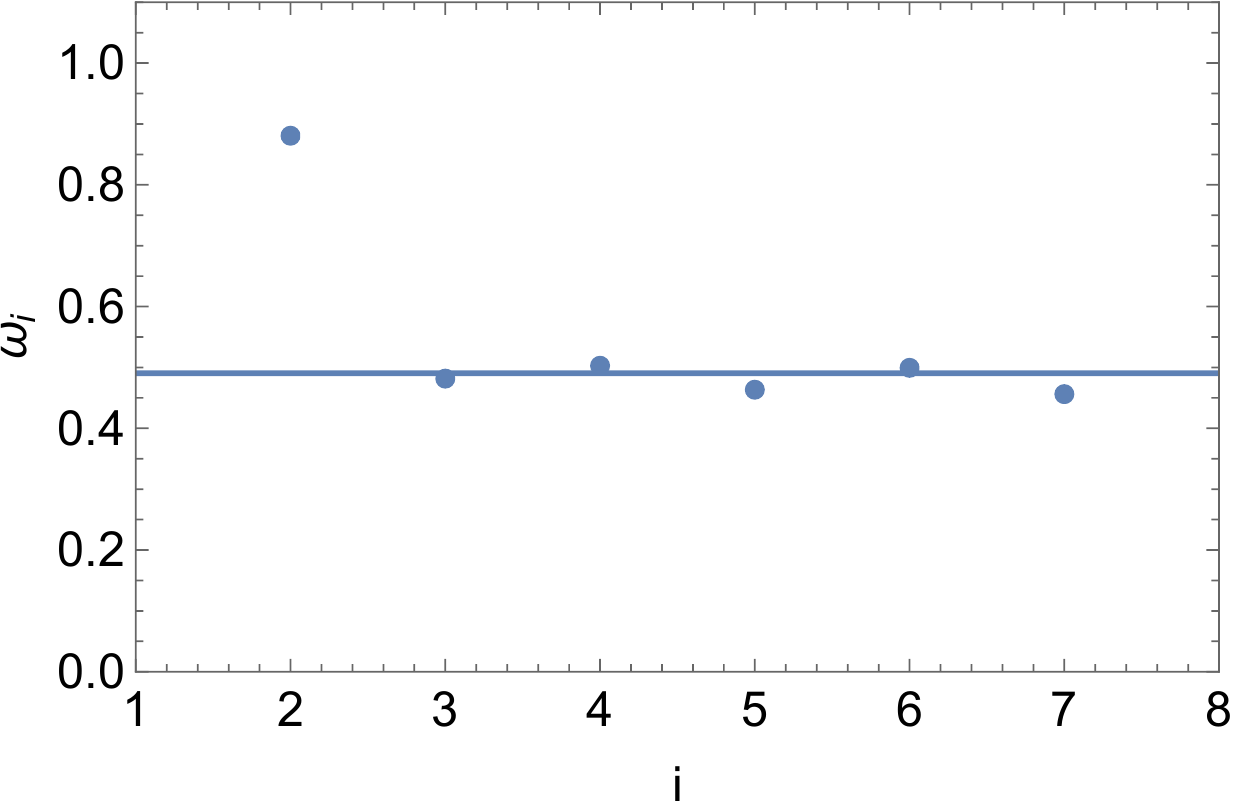} }
			\quad
			\captionsetup{justification=centering}\subfloat[  ]{\label{ex2_w_i_u_uh_sqrtx}\includegraphics[height=0.17\textheight,valign=m]{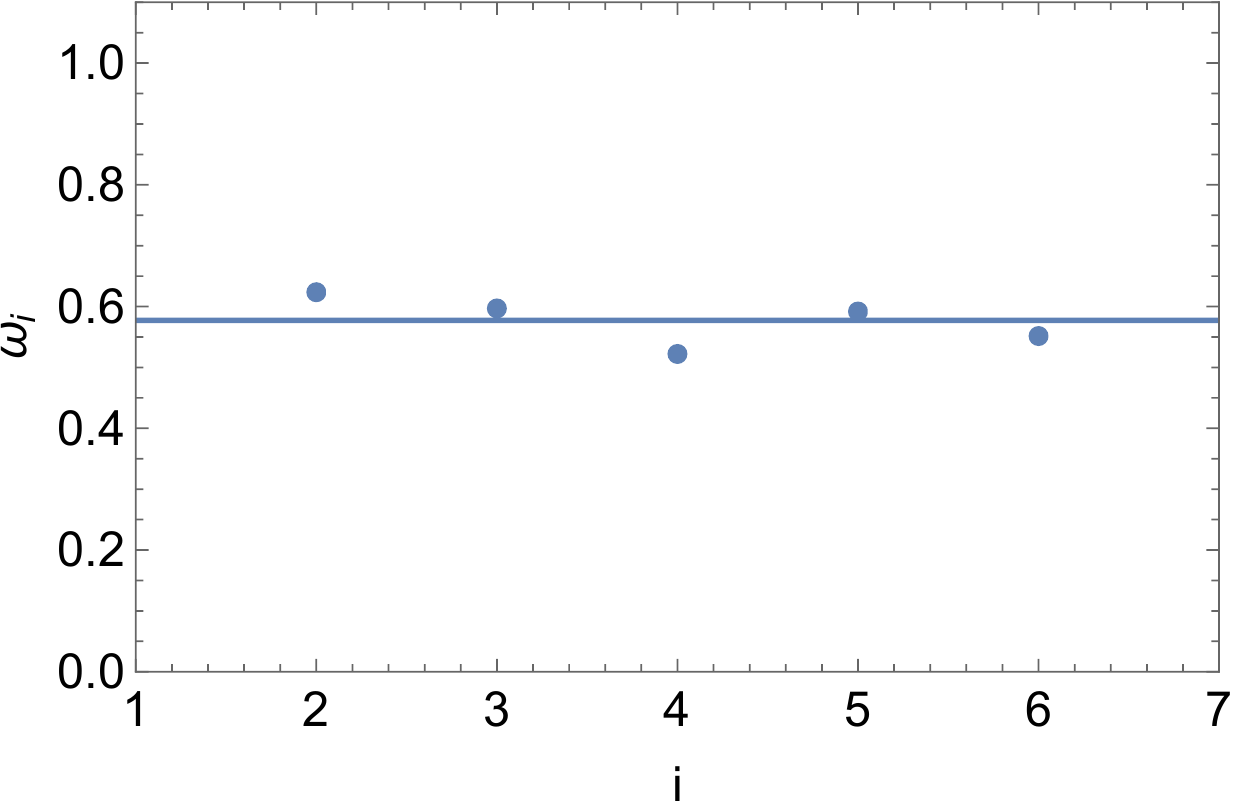} }
		\vspace{6pt}
		\caption{Visualization of the approximating polynomial and the statistic $ \omega_i $. (\textbf{a}) {Approximating} polynomial $y_{\rm{c}}^{(7)}(x)$ obtained from multiplying Equation~\eqref{eq:BVP_ex1-3_ODEs} by $ x $.   (\textbf{b}) Approximating polynomial $y_{\rm{c}}^{(6)}(x)$ obtained from $  y(x)-y_{\rm{c}}(x)$.
		(\textbf{c}) Plot of $ \omega_i,\,i=2,{3,}\ldots,7 $.\  (\lineone) $  {\tilde{\omega}_i}\approx 0.49$.  (\textbf{d}) Plot of $ \omega_i,\,i=2,{3,}\ldots,6 $.\  (\lineone) $ \overline{\omega_i}\approx 0.58 $.\label{fig10}}
	\end{figure}
	\vspace{-11pt}
	\begin{figure}[H]
		
			\centering
			
			\subfloat[The approximating polynomial $y_{\rm{c}}^{(8)}(x)$.]{\label{ex2approx_expx-1}\includegraphics[height=0.2\textheight,valign=m]{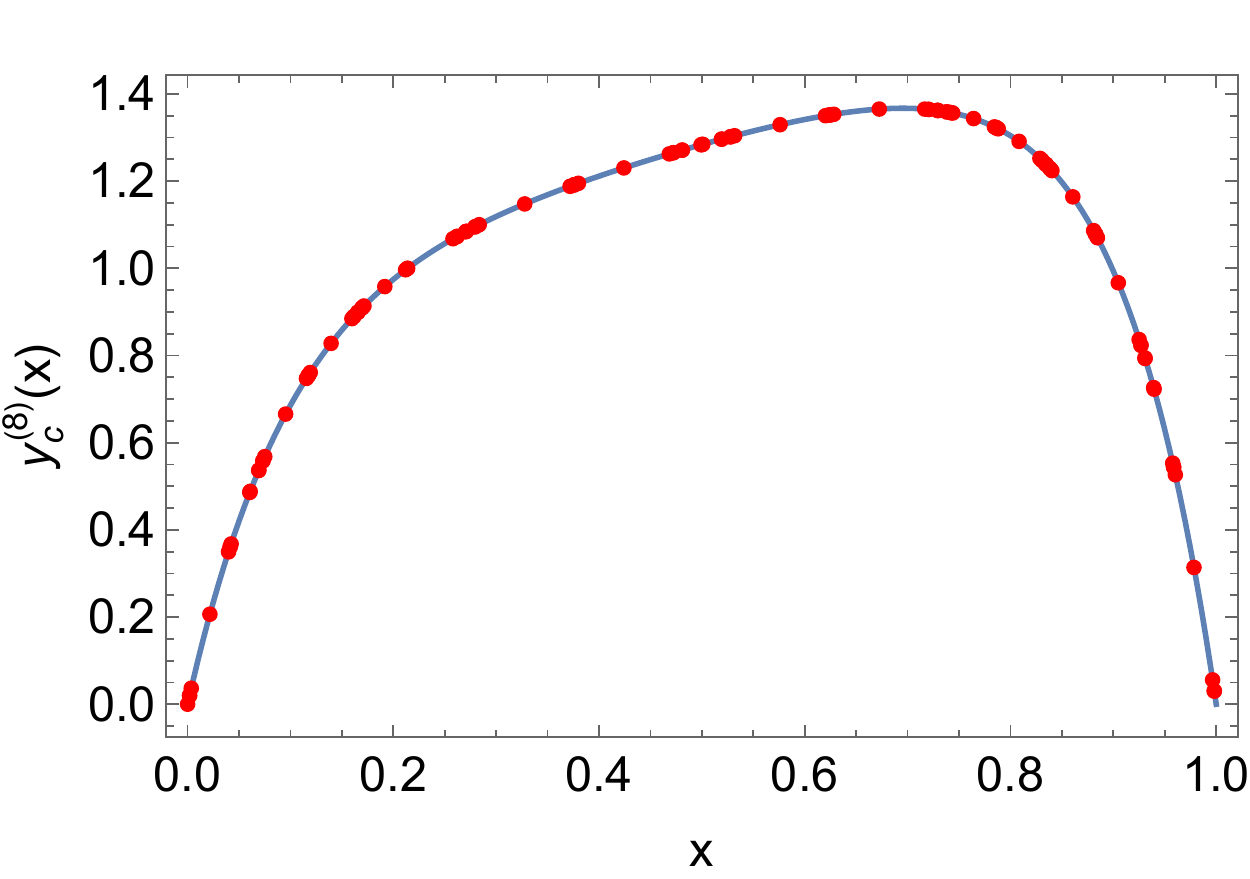} }
			\quad
			\subfloat[Plot of the statistic $\omega_i,\,i=2,{3,}\ldots,8$. 	(\lineone) $ \overline{\omega_i} \approx 0.65$.]{\label{ex2_w_i_expx-1}\includegraphics[height=0.2\textheight,valign=m]{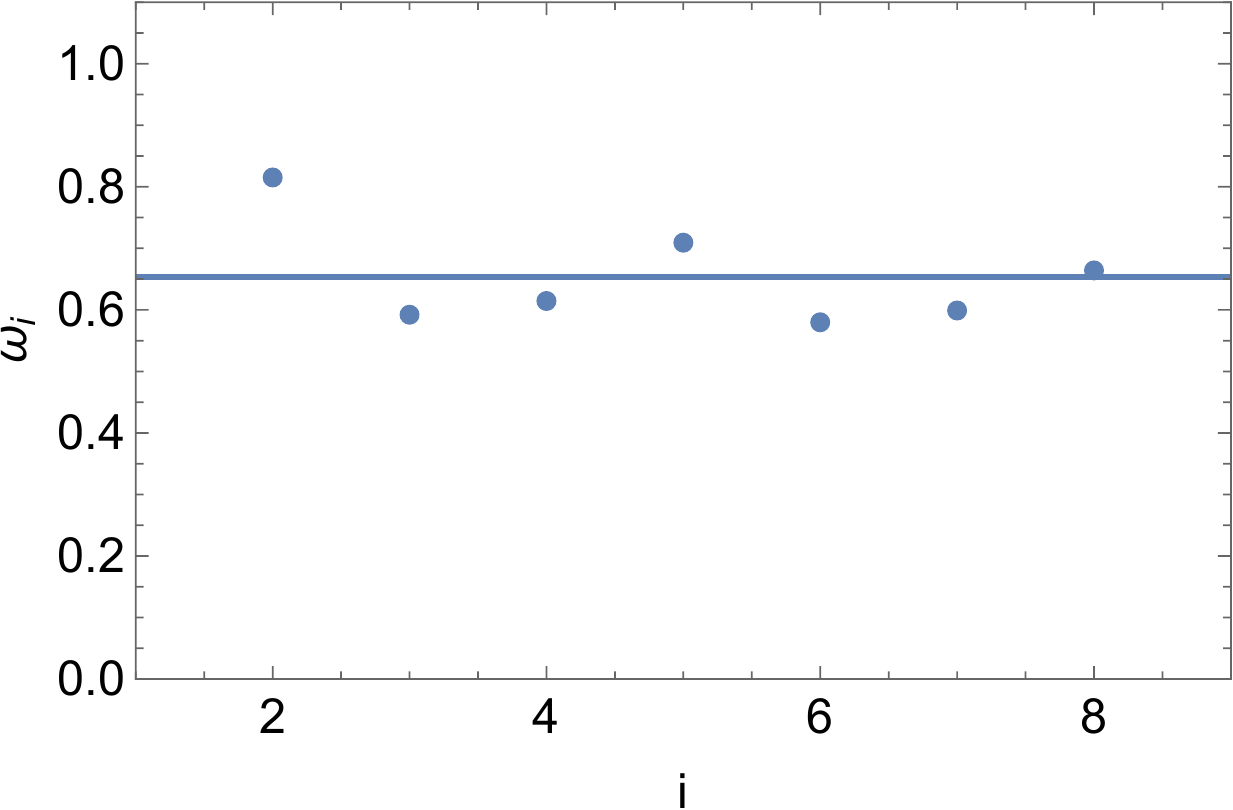} }
			
		\vspace{6pt}
		\caption{ (\textbf{a}) The approximating polynomial $y_{\rm{c}}^{(8)}(x)$. A proper subset of the set of points $ \mathsf{S} $ is shown. (\textbf{b})  Plot of the statistic $\omega_i$.	(\lineone) $ \overline{\omega_i} \approx 0.65$.\label{fig11}}
	\end{figure}

	We perform the least-squares fitting of the logarithm of the set $ \mathsf{R}$ to the logarithm of the upper bound~\eqref{eq:mainthm_gen_coll}.
	Figure~\ref{fig12}a shows the least-squares fitted model~\eqref{eq:mainthm_gen_coll} to the set $ \mathsf{R}$.
	Figure~\ref{fig12}b shows the residual, absolute local approximation error, and the mean value for the last iteration.
	The mean value $ \overline{R_{8}} $ is below the threshold value $10^{-6} $.
	The $L^2$ norm of the approximation error $\|y(x)-y_{\rm{c}}^{(10)}(x)\|\approx 3.1\times 10^{-7}$.

	\begin{figure}[H]
		
			\centering
			
			\subfloat[$A=1.2\times 10^5,\,r=7.3,\,\lambda=0.691,\, \delta=9.12\times 10^{-8}$.]{\label{ex2model_expx-1}\includegraphics[height=0.2\textheight,valign=m]{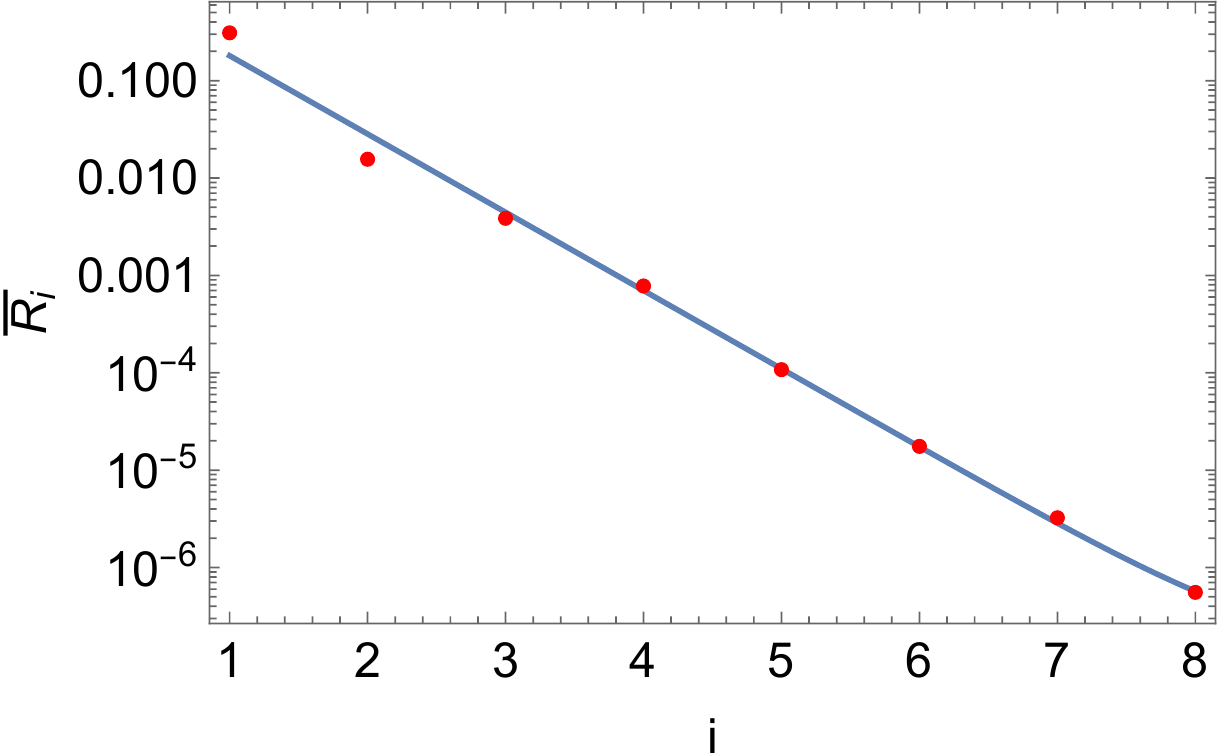} }
			\quad
			\subfloat[\lineone $\lVert R^{(8)}\rVert_{L^2}$ \linetwo $  |y(x)-y_{\rm{c}}^{(8)}(x)|$\ \linethree $ \overline{R_{8}} $.]{\label{ex2error_expx-1}\includegraphics[height=0.2\textheight,valign=m]{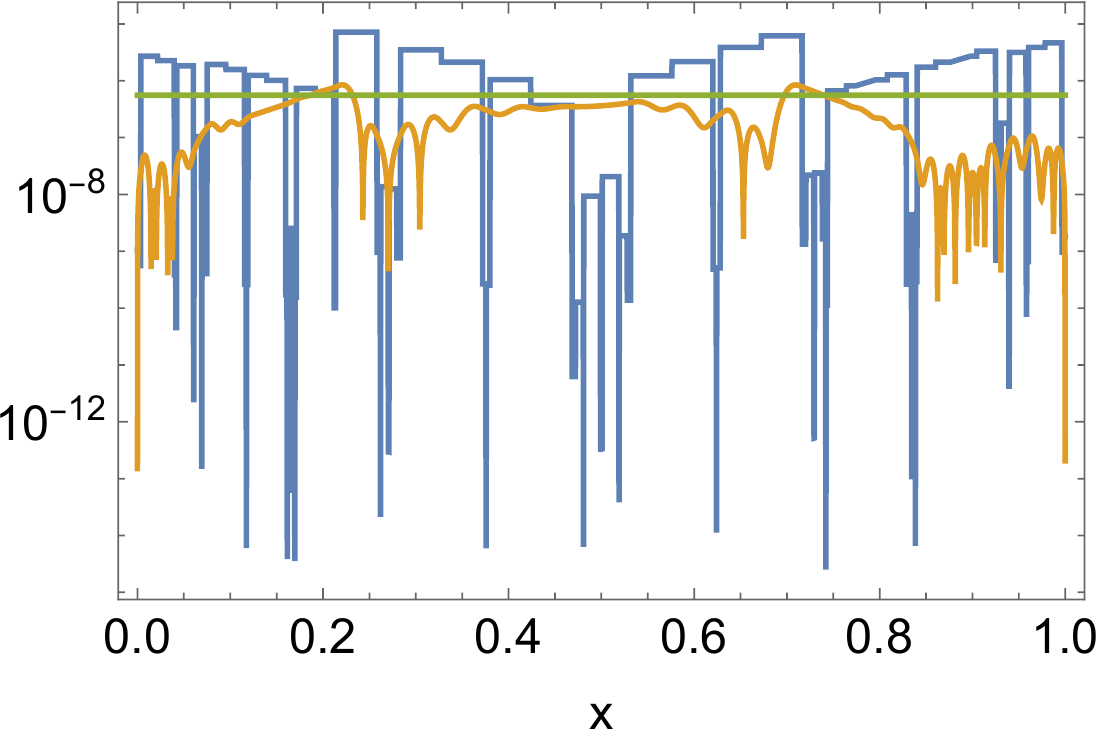} }
			
			\vspace{6pt}
		\caption{(\textbf{a})  Fitting the upper bound~\eqref{eq:mainthm_gen_coll} with the set $\mathsf{R}$. (\textbf{b})  Visualization of the residual, absolute local approximation error, and the mean value.\label{fig12}}
	\end{figure}
\end{example}

\begin{example}
	We study the BVP~\eqref{eq:BVP_ex1-3_ODEs} with  $ a(x)=0.02,\,b(x)=1,\,c(x)=0 $, and $ f(x)=1 $.
	The exact solution is given by
	\begin{equation*}
		y(x)=\mathrm{c}_1+\mathrm{c}_2 \exp(50 x)+x,
	\end{equation*}
	where $\mathrm{c}_1=-\mathrm{c}_2=\frac{1}{-1+\exp(50)}$.
	The solution $y(x)$ has a boundary layer near $x=1$.
	
	We set the number of Sinc points to be inserted in all partitions as $ m=2N+1=5$.
	The stopping criterion $ \varepsilon_{\mathrm{stop}}=10^{-6} $ was used.
	The algorithm terminates after $\kappa=9 $ iterations and the number of points $|{\mathsf{S}}|=1055$.
	
	Figure~\ref{fig13}a shows the approximate solution $ y_{\rm{c}}^{(9)}(x) $.
	A proper subset of the set of points $ \mathsf{S} $ is shown as red dots, which are projected onto the approximate solution $ y_{\rm{c}}^{(9)}(x) $.
	We plot the statistic $\omega_i$ as a function of the iteration index 
	$i,\,i=2,{3,}\ldots,9$, in Figure~\ref{fig13}b.
	The oscillations are decaying and the statistic $\omega_i$ is converging to an asymptotic value. The mean value $ \overline{\omega_i} $ is $ 0.62 $.
	
	\begin{figure}[H]
		
			\centering
			
			%		\captionsetup[subfigure]{justification=centering}
			\subfloat[Visualization of the approximating polynomial $y_{\rm{c}}^{(9)}(x)$.]{\label{ex3_approx}\includegraphics[height=0.2\textheight,valign=m]{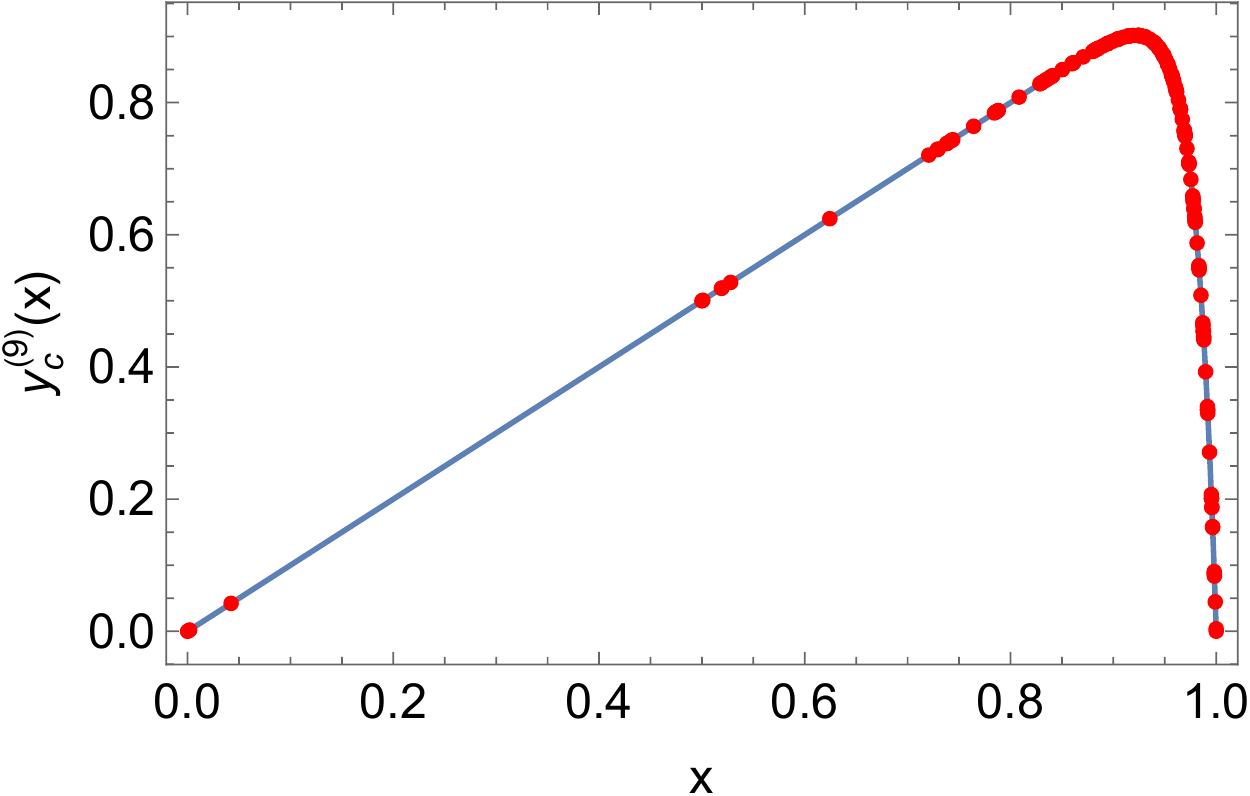} }
			\quad
			\subfloat[ Plot of the statistic $\omega_i,\,i=2,{3,}\ldots,9$.	(\lineone) $ \overline{\omega_i} \approx 0.62$.]{\label{ex3_w_i}\includegraphics[height=0.2\textheight,valign=m]{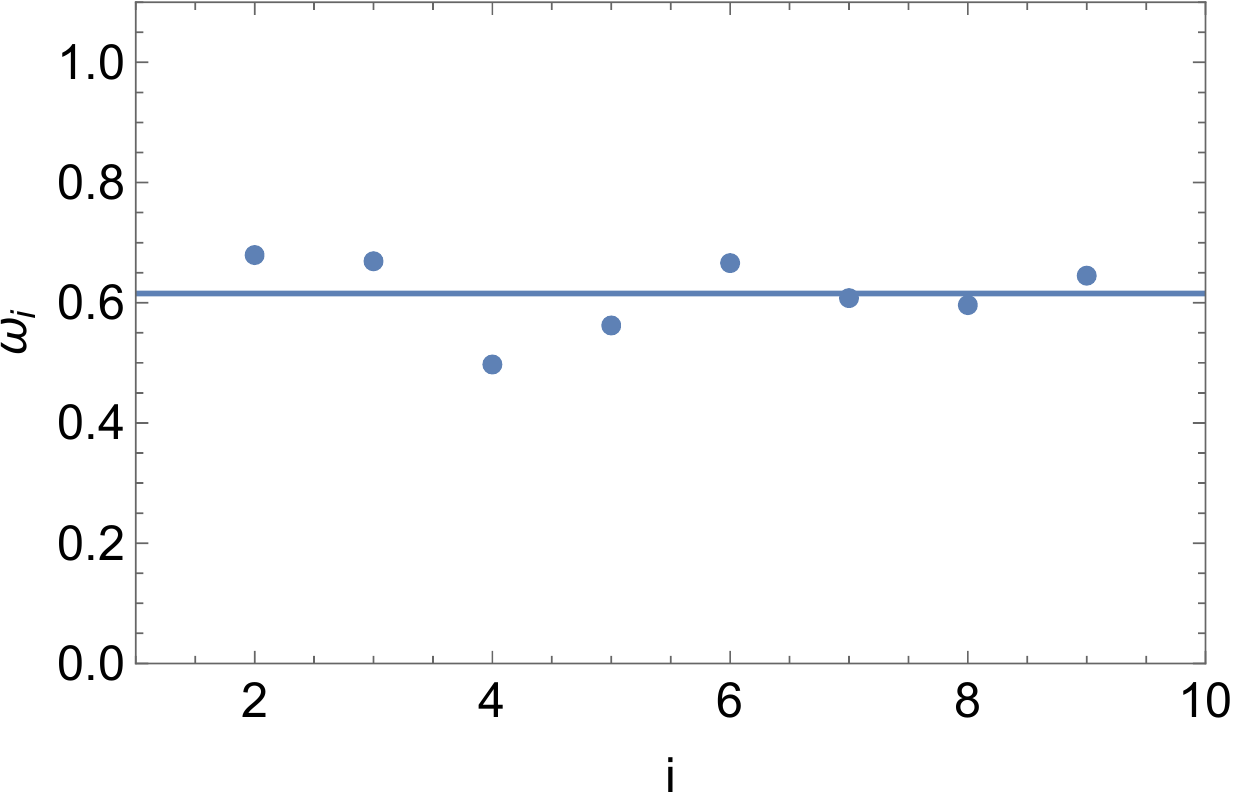} }
			
		\vspace{6pt}
		\caption{(\textbf{a})  The approximating polynomial $y_{\rm{c}}^{(9)}(x)$. A proper subset of the set of points $ \mathsf{S} $ is shown. (\textbf{b})   Plot of the statistic $\omega_i$.
			(\lineone) $ \overline{\omega_i} \approx 0.62$.\label{fig13}}
	\end{figure}

	We perform the least-squares fitting of the logarithm of the set $ \mathsf{R}$ to the logarithm of the upper bound~\eqref{eq:mainthm_gen_coll}.
	Figure~\ref{fig14}a shows the least-squares fitted model~\eqref{eq:mainthm_gen_coll} to the set $ \mathsf{R}$.
	Figure~\ref{fig14}b shows the residual, absolute local approximation error, and the mean value for the last iteration.
	{The plot of the $L^2$ norm values of the residual over the partitions shows that fine partitions are formed near $x=1$ due to the presence of the boundary layer.}
	{One observation is that the}   mean value $ \overline{R_{9}} $ is below the threshold value $ 10^{-6} $.
	The $L^2$ norm of the approximation error $\|y(x)-y_{\rm{c}}^{(9)}(x)\|\approx 2.36\times 10^{-8}$ and the threshold value in~\cite{Eriksson95_adaptive} is $0.02$.
	
	\begin{figure}[H]
		
			\centering
			
			\subfloat[$A=1.2\times 10^5,\,r=3.46,\,\lambda=0.625,\, \delta=2.62\times 10^{-7}$.]{\label{ex3_a_ODE_pw}\includegraphics[height=0.2\textheight,valign=m]{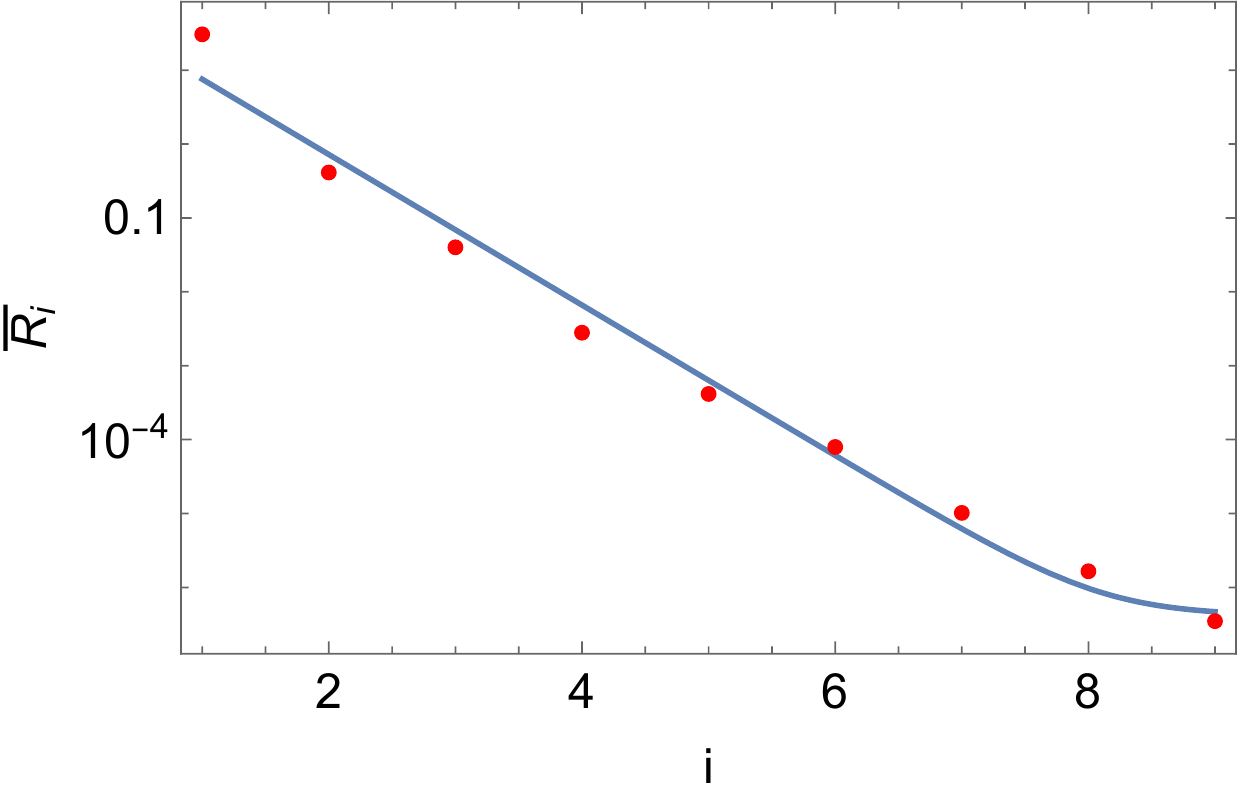} }
			\quad
			\subfloat[\lineone $\lVert R^{(9)}\rVert_{L^2}$ \linetwo $  |y(x)-y_{\rm{c}}^{(9)}(x)|$\ \linethree $ \overline{R_{9}} $.]{\label{ex3_b_ODE_pw}\includegraphics[height=0.2\textheight,valign=m]{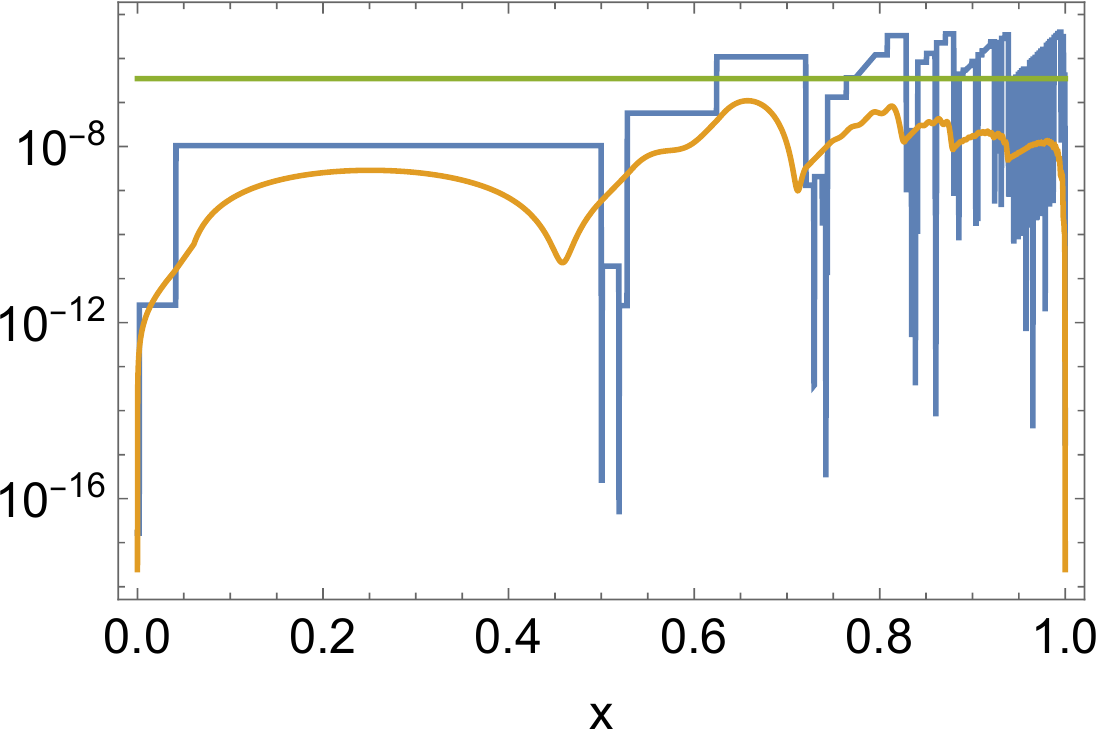} }
			
		\vspace{6pt}
		\caption{(\textbf{a})  Fitting the upper bound~\eqref{eq:mainthm_gen_coll} with the set $\mathsf{R}$. (\textbf{b})  Visualization of the residual, absolute local approximation error, and the mean value.\label{fig14}}
	\end{figure}
	
	In this example, we increase  the number of points per partition to $m=2N+1=7$  to examine the effect of the increase on the convergence of the algorithm.
	The algorithm terminates after $\kappa=5 $ iterations and the number of points $|{\mathsf{S}}|=350$.
	Figure~\ref{fig:ex3_5_7} shows the set $\mathsf{R}$ for  $m=2N+1=5$ Sinc points  and $m=2N+1=7$ Sinc points.
	It is observed that increasing the number of Sinc points per partition leads to faster convergence and a fewer number of iterations.
	
	\begin{figure}[H]
		\centering
		\includegraphics[width=0.5\linewidth]{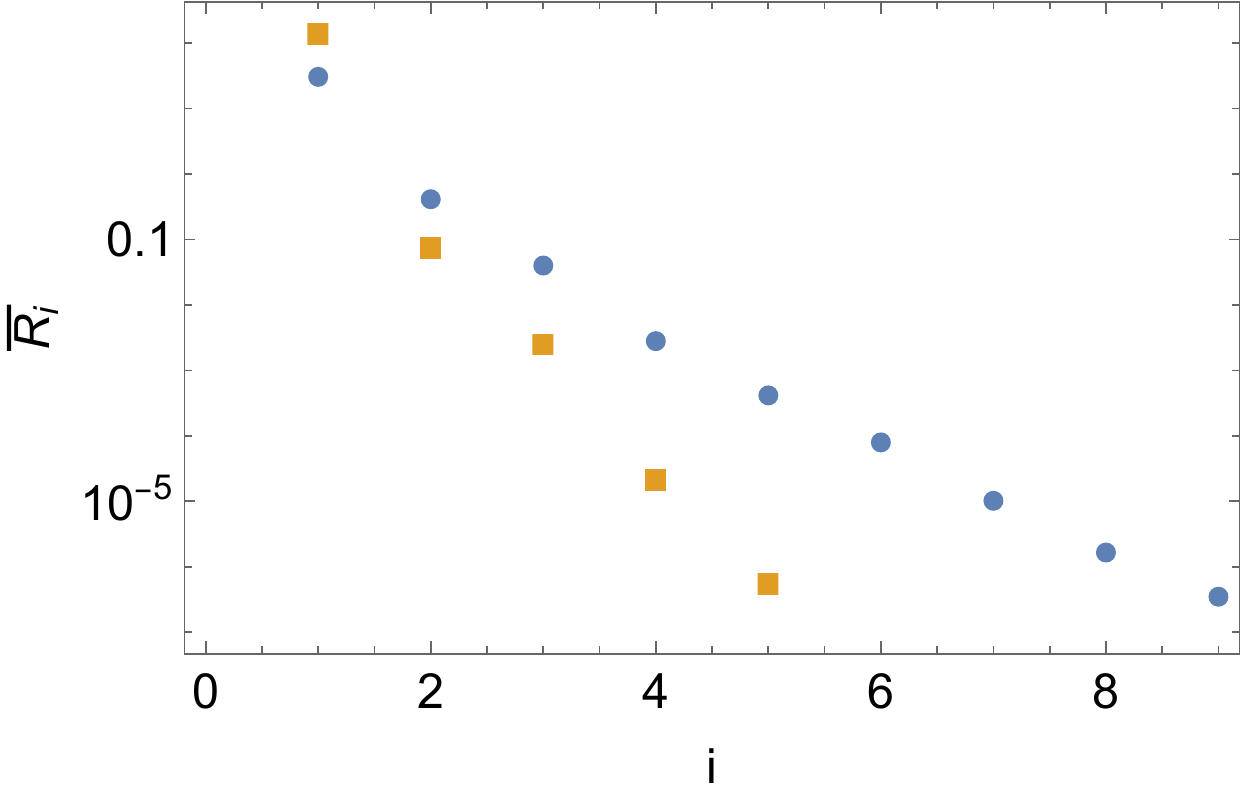}
		\caption{Plotting the set $\mathsf{R}$ for $m=2N+1=5$ Sinc points (\dotSinc) and $m=2N+1=7$ Sinc points (\squareChebT).}
		\label{fig:ex3_5_7}
	\end{figure}

\end{example}

\begin{example}
	We study the following BVP~\cite{Wheeler74,CH81,Fairweather83}
	\begin{equation}\label{eq:BVP_arctan}
		-(\upsilon(x)y')'=2\left[1+\alpha(x-\bar{x})\left(\arctan(\alpha(x-\bar{x}))+\arctan(\alpha\bar{x})\right)\right]
	\end{equation}
	with boundary conditions $ y(0)=y(1)=0$, where $ \alpha>0 $ and
	\[
	\upsilon(x)=\dfrac{1}{\alpha}+\alpha(x-\bar{x})^2.
	\]
	For large values of $ \alpha $, the BVP~\eqref{eq:BVP_arctan} has an interior layer close to $ \bar{x} $~\cite{CH81}. The exact solution is given~by
	\begin{equation*}
		y(x)=(1-x)\left[\arctan(\alpha(x-\bar{x}))+\arctan(\alpha\bar{x})\right].
	\end{equation*}
	We use the values reported in~\cite{Wheeler74}, i.e., $ \alpha=100 $ and $ \bar{x}=0.36388 $. This value of $ \bar{x} $ was chosen so that $ \lim\limits_{\alpha\to\infty}y(\bar{x}^+)\approx 2 $~\cite{Wheeler74}.
	
	We set the number of Sinc points to be inserted in all partitions as $ m=2N+1=7$.
	The stopping criterion $ \varepsilon_{\mathrm{stop}}=10^{-12} $ was used.
	The algorithm terminates after $\kappa=15 $ iterations and the number of points $|{\mathsf{S}}|=$ 21,469.
	
	Figure~\ref{fig16}a shows the approximate solution $ y_{\rm{c}}^{(15)}(x) $. A proper subset of the set of points $ \mathsf{S} $ is shown as red dots, which are projected onto the approximate solution $ y_{\rm{c}}^{(15)}(x) $.
	We plot the statistic $\omega_i$ as a function of the iteration index 
	$i,\,i=2,{3,}\ldots,15$, in Figure~\ref{fig16}b.
	One finding is that the  oscillations are decaying and the statistic $\omega_i$ is converging to an asymptotic value. The mean value $ \overline{\omega_i} $ is $ 0.46 $.

	\begin{figure}[H]
		
			\centering
			
			\subfloat[ Visualization of the approximating polynomial $y_{\rm{c}}^{(15)}(x)$.]{\label{arctan_approx}\includegraphics[height=0.2\textheight,valign=m]{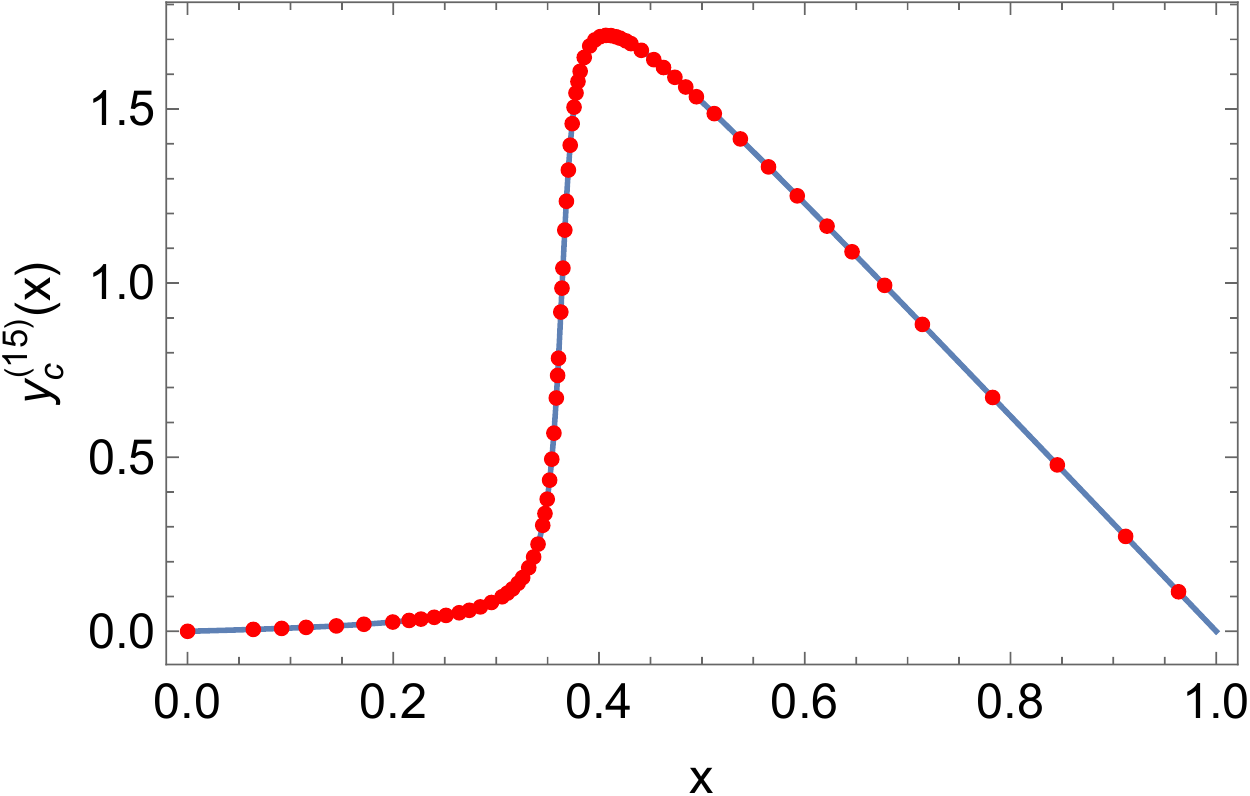} }
			\quad	\subfloat[ Plot of the statistic $\omega_i,\,i=2,{3,}\ldots,15$. 	(\lineone) $ \overline{\omega_i} \approx 0.46$.]{\label{arctan_w_i}\includegraphics[height=0.2\textheight,valign=m]{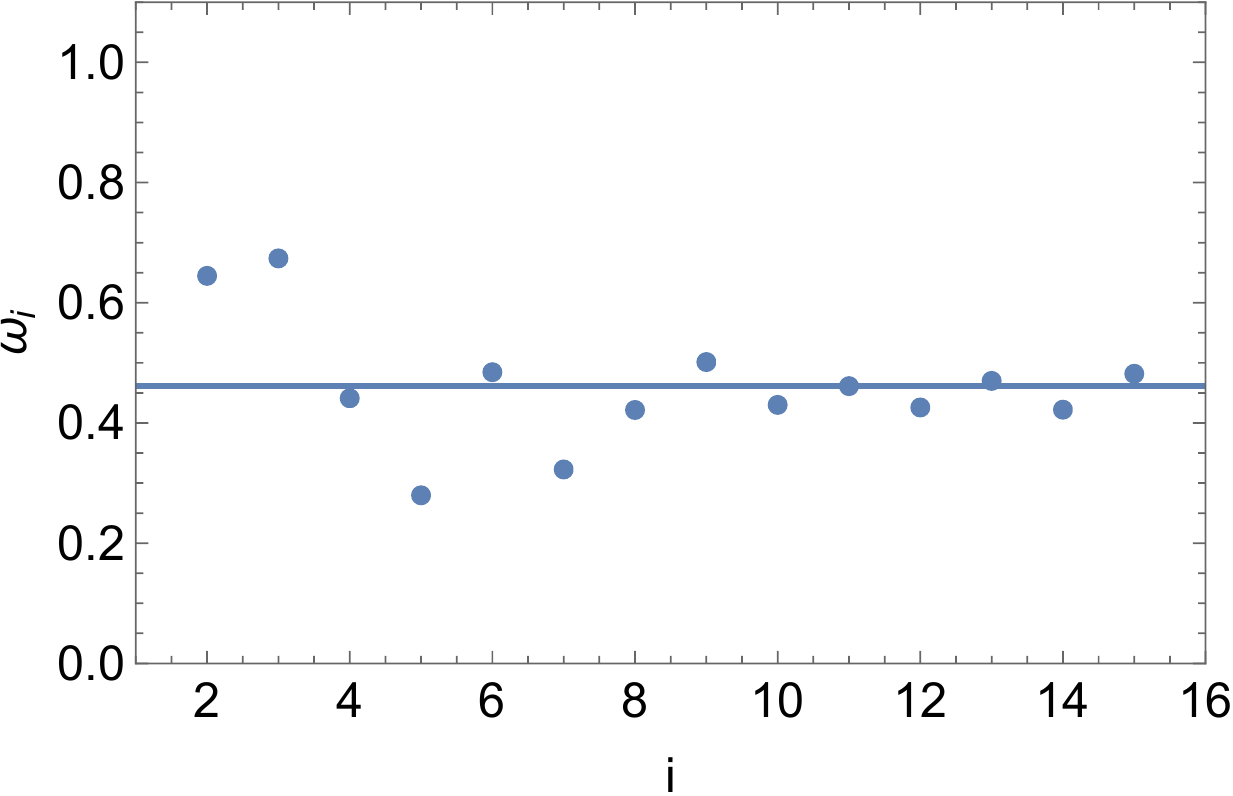} }
			
			\vspace{6pt}
		\caption{(\textbf{a})  The approximating polynomial $y_{\rm{c}}^{(15)}(x)$. A proper subset of the set of points $ \mathsf{S} $ is shown. (\textbf{b})   Plot of the statistic $\omega_i$.	(\lineone) $ \overline{\omega_i} \approx 0.46$.\label{fig16}}
	\end{figure}

	We perform the least-squares fitting of the logarithm of the set $ \mathsf{R}$ to the logarithm of the upper bound~\eqref{eq:mainthm_gen_coll}, where the parameter $ \delta $ is multiplied by $ 10^{-12} $ and $  10^{-12}\delta=O(10^{-19}) $.
	Figure~\ref{fig17}a shows the least-squares fitted model~\eqref{eq:mainthm_gen_coll} to the set $ \mathsf{R} $.
	The residual, absolute local approximation error, and the mean value for the last iteration are shown in Figure~\ref{fig17}b.
	{Fine partitions are formed near $x=\overline{x}$ due to the presence of the interior layer, as shown in the plot of the $L^2$ norm values of the residual over the partitions.}
	The mean value{ }$\overline{R_{15}}${ }is below the threshold $\varepsilon_{\mathsf{stop}}=10^{-12}$.

	We compare the $L^2$ norm of the approximation error of our adaptive piecewise Poly-Sinc method with other methods in Table~\ref{tbl:arctan_methods}.
	Method~\cite{CH81} requires a parameter for the construction of refinement intervals.
	The $ L^2 $ norm value of the approximation error is smaller than those reported in~\cite{Fairweather83,CH81}.
	
	\begin{table}[H]
		\caption{Comparison of the $L^2$ norm of the approximation error for different methods.\label{tbl:arctan_methods}}
		\newcolumntype{C}{>{\centering\arraybackslash}X}
		\begin{tabularx}{\textwidth}{CCCC}
			\toprule
			\textbf{Method} & \textbf{Adaptive PW PS} &\textbf{ \cite{CH81} }&\textbf{ (\cite{Fairweather83} Table~13(g))}\\
			\midrule
			$\|y(x)-y_{\mathrm{method}}(x)\|$&$1.104\times 10^{-14}$&$3\times 10^{-5}$&$1.9\times 10^{-12}$\\
			\bottomrule
		\end{tabularx}
	\end{table}
	\unskip
	
	\begin{figure}[H]
		
			\centering

			\subfloat[$A=1.14\times 10^5,\,r=0.907,\,\lambda=0.698,\, \delta=2.94\times 10^{-7}$.]{\label{arctan_a}\includegraphics[height=0.2\textheight,valign=m]{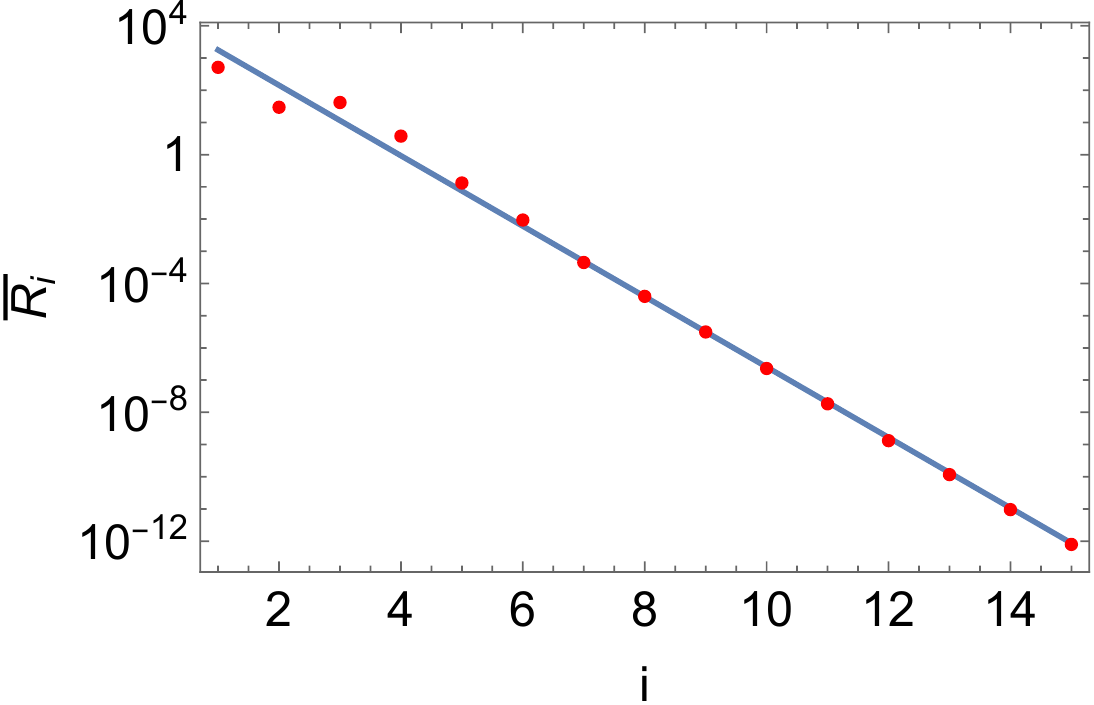} }
			\quad
			\subfloat[\lineone $\lVert R^{(15)}\rVert_{L^2}$ \linetwo $  |y(x)-y_{\rm{c}}^{(15)}(x)|$\ \linethree $ \overline{R_{15}} $.]{\label{arctan_b}\includegraphics[height=0.2\textheight,valign=m]{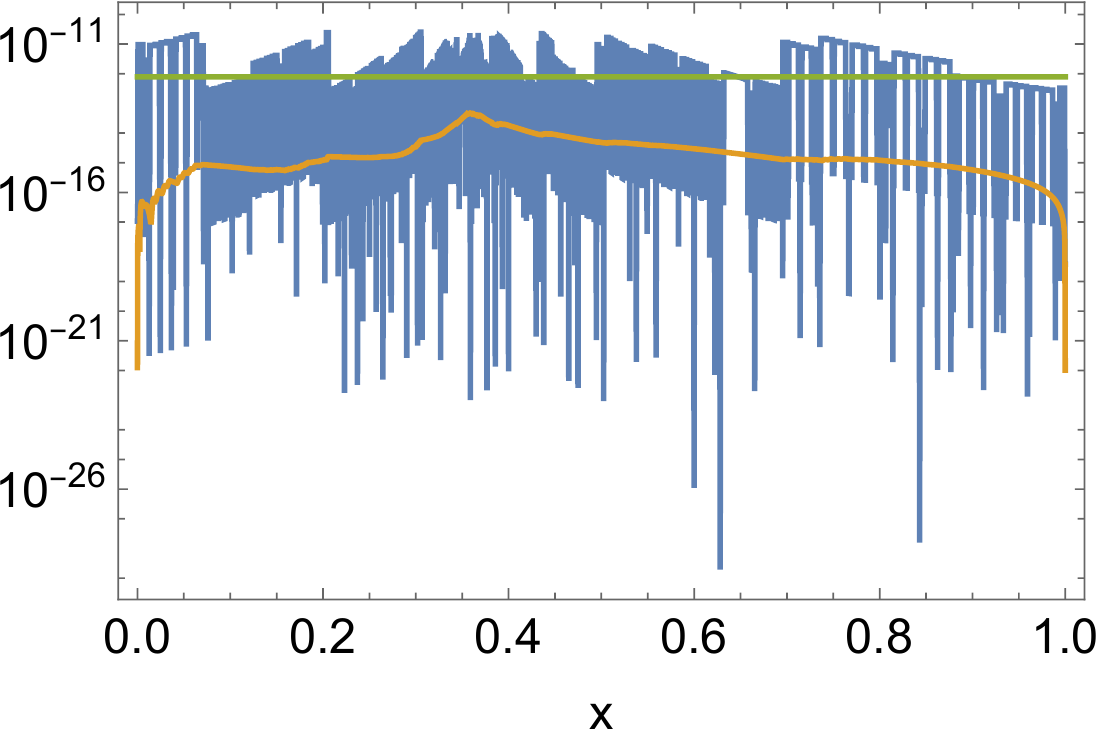} }
			\vspace{6pt}
		\caption{(\textbf{a})  Fitting the upper bound~\eqref{eq:mainthm_gen_coll} with the set $\mathsf{R}$. (\textbf{b})  Visualization of the residual, absolute local approximation error, and the mean value.\label{fig17}}
	\end{figure}
	
\end{example}

\begin{example}
	We consider the BVP~\cite{Hemker77,Ascher79}
	\begin{equation*}
		-\epsilon\, y''-x\,y'=\epsilon\pi^2\cos(\pi x)+\pi x\,\sin(\pi x),\quad  x\in[-1,1],
	\end{equation*}
	with boundary conditions $ y(-1)=-2,\ y(1)=0 $  and $ \epsilon>0 $ is a parameter.
	The exact solution follows as
	\begin{equation*}
		y(x)=	\cos(\pi x)+\dfrac{\erf(x/\sqrt{2\epsilon})}{\erf(1/\sqrt{2\epsilon})}.
	\end{equation*}
	The exact solution has a shock layer near $ x=0 $~\cite{Hemker77}. We set $ \epsilon=10^{-6} $~\cite{Ascher79}.
	
	We set the number of Sinc points to be inserted in all partitions as $ m=2N+1=5$.
	The stopping criterion $ \varepsilon_{\mathrm{stop}}=10^{-11} $ was used.
	The algorithm terminates after $\kappa=16 $ iterations and the number of points $|{\mathsf{S}}|=18530$.
	
	Figure~\ref{fig18}a shows the approximate solution $ y_{\rm{c}}^{(16)}(x) $.
	A proper subset of the set of points $ \mathsf{S} $ is shown as red dots, which are projected onto the approximate solution $ y_{\rm{c}}^{(16)}(x) $.
	We plot the statistic $\omega_i$ as a function of the iteration index 
	$i,\,i=2,{3,}\ldots,16$, in Figure~\ref{fig18}b.
	The oscillations are decaying and the statistic $\omega_i$ is converging to an asymptotic value. The mean value $ \overline{\omega_i}\approx 0.55 $.

	We perform the least-squares fitting of the logarithm of the set $ \mathsf{R}$ to the logarithm of the upper bound~\eqref{eq:mainthm_gen_coll}, where the parameter $ \delta $ is multiplied by the factor $ 10^{-9} $ and $ 10^{-9}\delta=O(10^{-12}) $.
	Figure~\ref{fig19}a shows the least-squares fitted model~\eqref{eq:mainthm_gen_coll} to the set $ \mathsf{R} $.
	The residual, absolute local approximation error, and the mean value for the last iteration are shown in Figure~\ref{fig19}b.
	{The plot of the $L^2$ norm values of the residual over the partitions show that fine partitions are formed near $x=0$ due to the presence of the shock layer.}
	The mean value{ }$\overline{R_{16}}${ }is below the threshold $\varepsilon_{\mathsf{stop}}=10^{-11}$.
	
	We compare the supremum norm of the approximation error of our adaptive piecewise Poly-Sinc method with other methods in Table~\ref{tbl:spike_methods}.
	$B-$splines were used as basis functions~\cite{Ascher79}.
	The {supremum} norm result is in the same order as that of~\cite{Ascher79}.
	Our adaptive method can reach a smaller value if we set $\varepsilon_{\mathsf{stop}}<10^{-11}$.
	
	\begin{figure}[H]
		
			\centering
			
			\captionsetup{justification=centering}\subfloat[Visualization of the approximating polynomial $y_{\rm{c}}^{(16)}(x)$.]{\label{spike_approx}\includegraphics[height=0.2\textheight,valign=m]{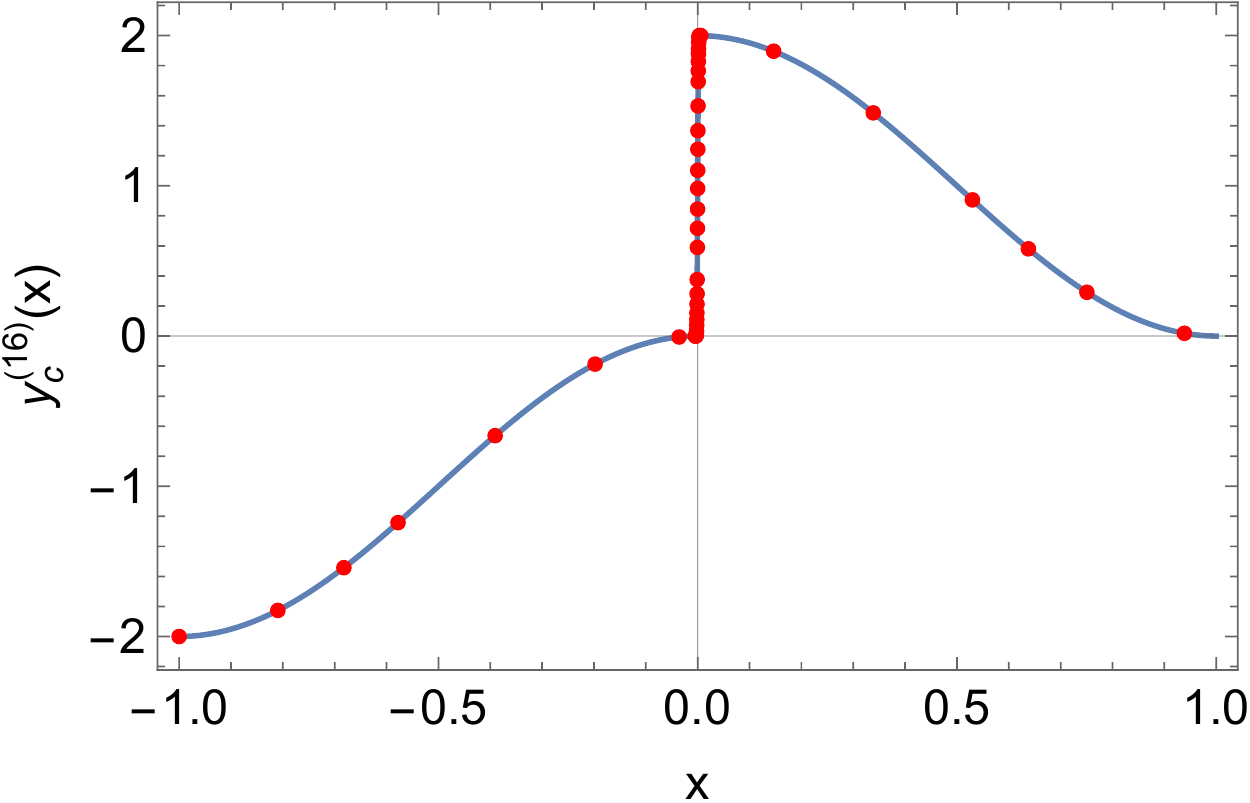} }
			\captionsetup{justification=centering}	\quad	\subfloat[Plot of the statistic $\omega_i,\,i=2,{3,}\ldots,16$.]{\label{spike_w_i}\includegraphics[height=0.2\textheight,valign=m]{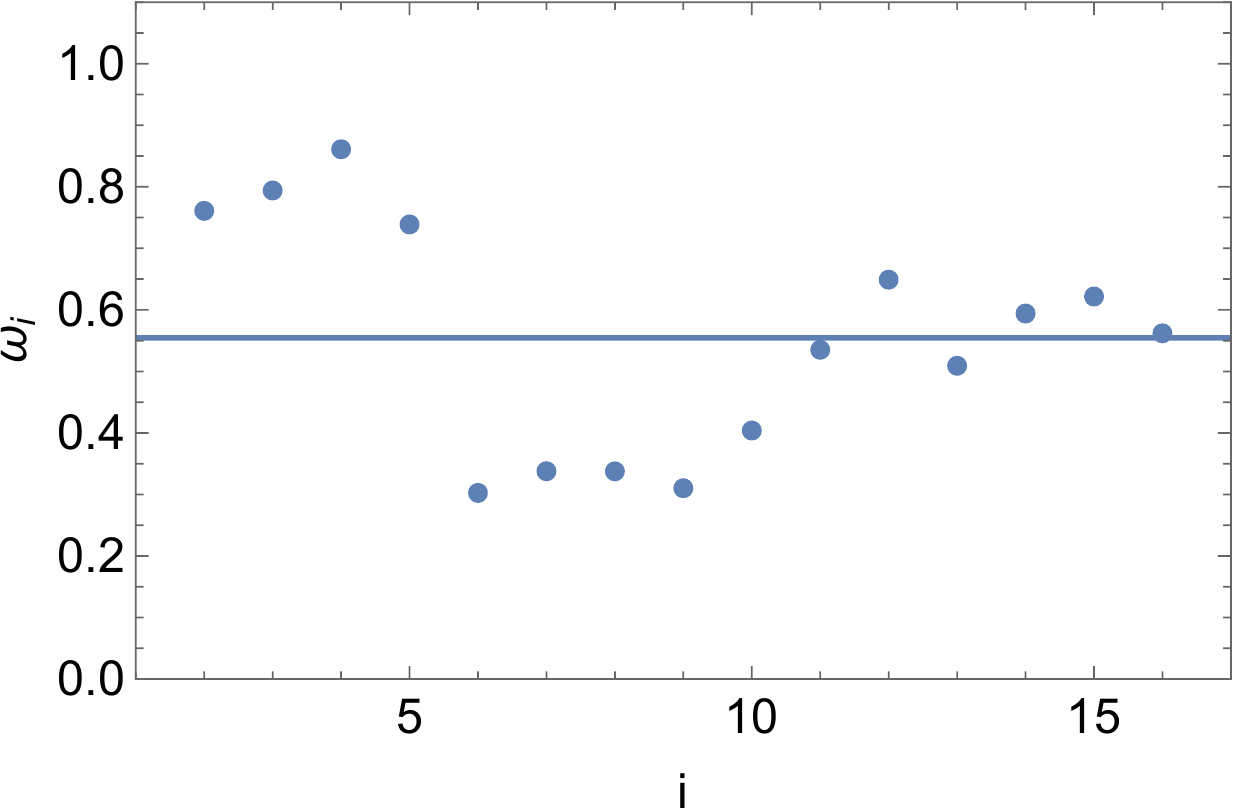} }
			
			\vspace{6pt}
		\caption{(\textbf{a}) The approximating polynomial $y_{\rm{c}}^{(16)}(x)$. A proper subset of the set of points $ \mathsf{S} $ is shown. (\textbf{b})  Plot of the statistic $\omega_i$.\label{fig18}}
	\end{figure}
	\unskip
	
	\begin{table}[H]
		\caption{Comparison of the {supremum} norm of the approximation error for different methods.\label{tbl:spike_methods}}
		\newcolumntype{C}{>{\centering\arraybackslash}X}
		\begin{tabularx}{\textwidth}{CCC}
			\toprule
			\textbf{Method} &\textbf{ Adaptive PW PS} &\textbf{ \cite{Ascher79} } \\
			\midrule
			$\|y(x)-y_{\mathrm{method}}(x)\|_{\infty}$&$1.215 \times 10^{-10}$&$4.3\times 10^{-10}$\\
			\bottomrule
		\end{tabularx}
	\end{table}
\end{example}

{We observe that the processing times differ among the many problems we look at.
	This is due to two factors: first, the fact that we are aiming for a very exact outcome; and second, the fact that there are many sorts of challenges.
	Therefore, listing the computing time in seconds is meaningless since it would only indicate the computational power used on our machine, which will vary for various users.
	Overall, it is evident that using adaptive techniques takes longer than using a simple collocation approach or finite element methods with a lower accuracy.
	In our examples, we used a stopping criterion of $10^{-6}$ or less.
	As a result, our goal is to complete the computation in the most accurate way feasible rather than the quickest way possible.
	This will inevitably lengthen the processing time for some problems, such as stiff or layer problems. }
\vspace{-4pt}

\begin{figure}[H]
	
		\centering
		
		\subfloat[$A=1.2\times 10^5,\,r=2.8,\,\lambda=0.629,\, \delta=2.11\times 10^{-3}$.]{\label{spike_a}\includegraphics[height=0.2\textheight,valign=m]{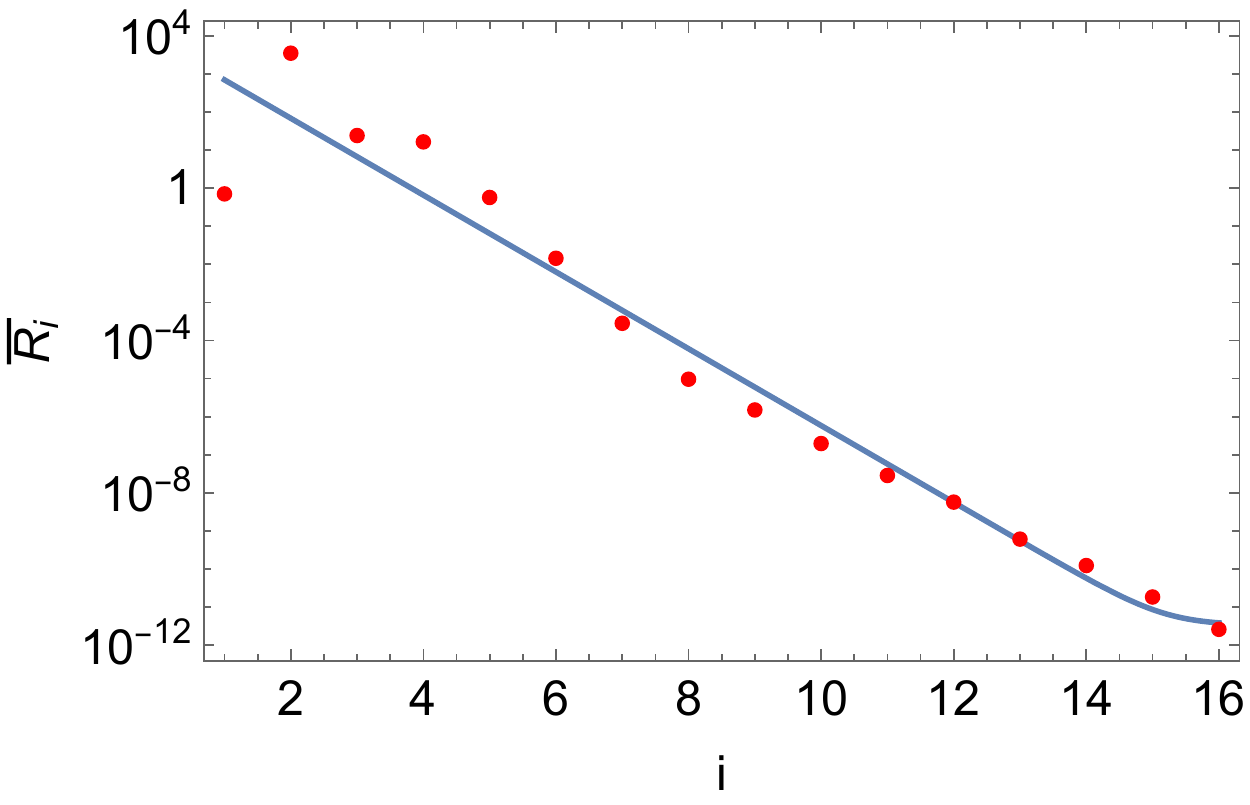} }
		\quad
		\subfloat[\lineone $\lVert R^{(16)}\rVert_{L^2}$ \linetwo $  |y(x)-y_{\rm{c}}^{(16)}(x)|$ \linethree $ \overline{R_{16}} $.]{\label{spike_b}\includegraphics[height=0.2\textheight,valign=m]{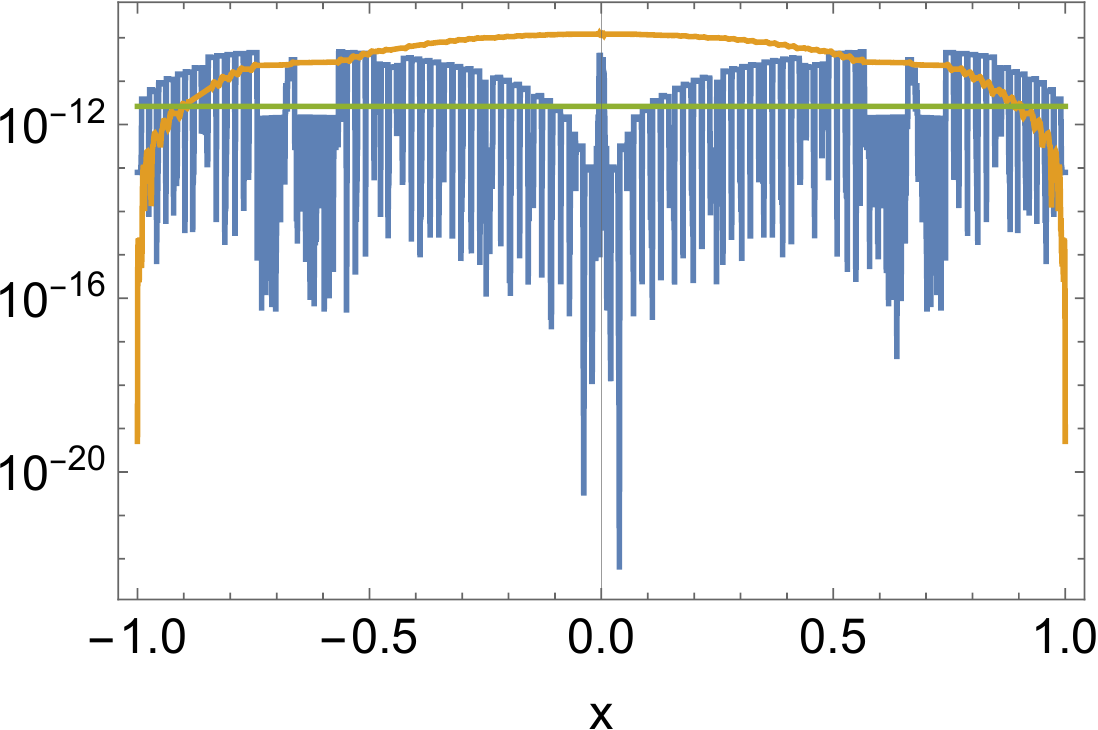} }
		
	\vspace{6pt}
	\caption{(\textbf{a})  Fitting the upper bound~\eqref{eq:mainthm_gen_coll} with the set $\mathsf{R}$. (\textbf{b})  Visualization of the residual, absolute local approximation error, and the mean value.\label{fig19}}
\end{figure}

\section{Conclusions}\label{sec:conc}

In this paper, we developed an adaptive piecewise collocation method based on Poly-Sinc interpolation for the approximation of solutions to ODEs.
We showed the exponential convergence in the number of iterations of the {\it a priori} error estimate obtained from the piecewise collocation method, and provided that a good estimate of the exact solution $y(x)$ at the Sinc points exists.
We used a statistical approach for partition refinement,  in which we computed the fraction of a standard deviation as the ratio of the mean absolute deviation to the sample standard deviation.
We demonstrated by several examples that an exponential error decay is observed for regular ODEs and ODEs whose exact solutions exhibit an interior layer, a boundary layer, and a shock layer.
We showed that our adaptive algorithm can deliver results with high accuracy at the expense of the slower computation for stiff ODEs.

\vspace{6pt}

\noindent\textbf{Author Contributions:} Conceptualization, M.Y. and G.B.; methodology, O.K., M.Y. and G.B.; software, O.K. and G.B.; validation, O.K., H.E.-S., M.Y. and G.B.; formal analysis, O.K. and H.E.-S.; investigation, H.E.-S. and G.B.; resources, not applicable; data curation, not applicable; writing---original draft preparation, O.K.; writing---review and editing, G.B. and O.K.; visualization, O.K.; supervision, H.E.-S., M.Y. and G.B.; project administration, H.E.-S.; funding acquisition, there is no funding of this article. All authors have read and agreed to the published version of the manuscript.\\[.5em]
\noindent\textbf{Funding:} M.Y. is funded by BMBF under the contract 05M20VSA.\\[.5em]
\noindent\textbf{Institutional Review Board Statement: }Not applicable.\\[.5em]
\noindent\textbf{Informed Consent Statement: } Not applicable.\\[.5em]
\noindent\textbf{Data Availability Statement: }Not applicable.\\[.5em]
\noindent\textbf{Acknowledgments: }The authors thank Frank Stenger for his insightful conversations and continuous guidance.\\[.5em]
\noindent\textbf{Conflicts of Interest: }The authors declare no conflict of interest.

\bibliographystyle{IEEEtran}
\bibliography{Adaptive_Piecewise_Poly_Sinc_Methods_for_ODEs-arXiv}
\end{document}